\documentclass[a4paper]{article}
\usepackage{geometry}
\usepackage[affil-it]{authblk}
\usepackage{amsfonts,amsthm,amsmath,mathrsfs,epsfig}
\RequirePackage[colorlinks,citecolor=blue,urlcolor=blue]{hyperref}
\usepackage{mathabx}
\usepackage{multirow}
\usepackage{tabularx}
\usepackage{enumitem}
\usepackage{bbm}
\usepackage{multirow}
\usepackage{makecell}
\numberwithin{equation}{section}
\newcommand{\Var}{{\rm Var}}
\newcommand{\Cov}{{\rm Cov}}

\newcommand{\R}{\mathbb{R}}
\newcommand{\N}{\mathbb{N}}

\newcommand{\E}{\mathrm {E}}
\newcommand{\B}{\mathbb{B}}

\newcommand{\KS}{{\rm \operatorname {KS}}}
\newcommand{\CvM}{{\rm \operatorname {CvM}}}
\newcommand{\I}{{\rm \operatorname I}}

\newtheorem{thm}{Theorem}
\newtheorem{lem}{Lemma}
\newtheorem{cor}{Corollary}

\newtheorem{rmk}{Remark}

\makeatletter
\newcommand{\vast}{\bBigg@{4}}
\newcommand{\Vast}{\bBigg@{5}}
\makeatother

\def\AA{ \mathfrak{A} }

\def\BB{ \mathfrak{B} }

\theoremstyle{definition}

\usepackage{lipsum}

\newcommand\blfootnote[1]{%
  \begingroup
  \renewcommand\thefootnote{}\footnote{#1}%
  \addtocounter{footnote}{-1}%
  \endgroup
}

\begin{document}

\title{Testing hypotheses about mixture distributions using not identically distributed data}

\author{Daniel Gaigall}

\affil{Department of Mathematics, Heinrich-Heine-University  D\"usseldorf, Universit\"atsstr. 1, 40225 D\"usseldorf, Germany.}

\date{}

\maketitle

\blfootnote{Email adress: Daniel.Gaigall@uni-duesseldorf.de}

\begin{abstract}
Testing hypotheses of goodness-of-fit about mixture distributions on the basis of independent but not necessarily identically distributed random vectors is considered. The hypotheses are given by a specific distribution or by a family of distributions. Moreover, testing hypotheses formulated by Hadamard differentiable functionals is discussed in this situation, in particular the hypothesis of  central symmetry, homogeneity and independence. Kolmogorov-Smirnov or Cram\'er-von-Mises type statistics are suggested as well as methods to determine critical values. The focus of the investigation is on asymptotic properties of the test statistics. Further, outcomes of simulations for finite sample sizes are given. Applications to models with not identically distributed errors are presented. The results imply that the tests are of asymptotically exact size and consistent. 



\end{abstract}

\section{Introduction}\label{einfuhr}
Not identically distributed errors are discussed in various  theoretical and practical contexts, see, e.g., the works of Lu et al. \cite{lu},  Kuljus and Zwanzig \cite{kul}, G\"ornitz et al.  \cite{gor}, Eiker \cite{eic} and Delaigle and Meister \cite{del}. Consider here the following situation as a motivation. Suppose  independent and identically distributed data are underlying, where an independent  but not identically distributed noise is present.  It is assumed that the difference in the distribution of the noise vanishes if the number of observations increases. The interesting statistical problem is a testing problem of goodness-of-fit formulated with the independent and identically distributed original data.

\medskip
More precisely, consider a sequence of independent and identically distributed real valued random variables $Y_1,Y_2,\dots$ with unknown underlying distribution $\mathcal L(Y_1)$. Suppose the user has to treat the testing problem of goodness-of-fit
\begin{equation}\label{testpr0}
\mathrm H:\mathcal L(Y_1)=\mathcal L(Y_0),~\mathrm K:\mathcal L(Y_1)\neq \mathcal L(Y_0),
\end{equation}
where $Y_0$ is a real valued random variable with known distribution $\mathcal L(Y_0)$. Furthermore, let $Z_1,Z_2,\dots$ be another sequence of independent but not necessarily identically distributed real valued random variables with known distributions such that the convergence in distribution
\begin{equation*}
Z_i\overset{\mathrm d}\longrightarrow Z~\text{as}~i\rightarrow\infty
\end{equation*}
holds, with  a real valued random variable $Z$. Assume  $Y_1,Y_2,\dots$ and $Z_1,Z_2,\dots$ are independent. In addition, a sequence of known and measurable maps $e_i:\R\times \R\rightarrow \R$, $i\in\N$, is given and fulfills the convergence
\begin{equation*}
\forall z\in\R:\lim_{i\rightarrow\infty} e_i(\cdot,z)=e(\cdot,z)~\text{uniformly on}~\R,
\end{equation*}
where $e:\R\times\R \rightarrow \R$ is a measurable map. E.g.,
\begin{equation}\label{eg1}
e_i(y,z)=e(y,z)=y+z,~(y,z)\in \R\times \R,~i\in\N.
\end{equation}
Suppose  the user cannot observe directly $Y_1,Y_2,\dots$, but observes
\begin{equation*}
e_{1}(Y_1,Z_1),e_{2}(Y_2,Z_2),\dots,
\end{equation*}
i.e., $e_i(\cdot,Z_i)$ represents a random noise in measurement $i\in\N$. Therefore, the user has to treat testing problem (\ref{testpr0}) on the basis of the independent but not necessarily identically distributed data $X_1,\dots,X_n$,
\begin{equation*}
X_i:=e_i(Y_i,Z_i),~i\in\N,
\end{equation*}
where $n\in\N$ is a given sample size. In example (\ref{eg1}),
\begin{equation*}
X_i=Y_i+Z_i,~i\in\N.
\end{equation*}
Denote by $F_i$ the distribution function of $X_i=e_i(Y_i,Z_i)$, $i\in\N$, and by $F$ the distribution function of $e(Y_1,Z)$. Let $H_i$ be the distribution function of $e_i(Y_0,Z_i)$, $i\in\N$, and let $G$ be the distribution function of $e(Y_0,Z)$. Assume that for all $z\in\R$, the map $e(\cdot,z)$ is injective. Then, testing problem (\ref{testpr0}) is equivalent to
\begin{equation*}
\mathrm H:F=G,~\mathrm K:F\neq G.
\end{equation*}
Moreover, if the hypothesis H is true, $F_i=H_i$, $i\in\N$.

\medskip
In order to get a permutation invariant test statistic, i.e., invariant under transformation of the data from the set
\begin{equation*}
\begin{split}
&\lbrace T_n:\times_{i=1}^n\R\rightarrow \times_{i=1}^n\R;T_n(x_1,\dots,x_n)=(x_{\pi_n(1)},\dots,x_{\pi_n(n)}),\\
&\pi_n:\lbrace 1,\dots,n\rbrace\rightarrow\lbrace 1,\dots,n\rbrace~\text{bijective}\big\rbrace,
\end{split}
\end{equation*}
define the arithmetically averaged distribution function
\begin{equation*}
\mathbb F_n(x):=\frac 1 n\sum_{i=1}^n F_i(x),~x\in\R,
\end{equation*}
based on the distribution functions $F_1,\dots,F_n$ of $X_1,\dots,X_n$ and the arithmetically averaged distribution function
\begin{equation*}
G_n(x):=\frac 1 n \sum_{i=1}^nH_i(x),~x\in\R,
\end{equation*}
based on the distribution functions $H_1,\dots,H_n$, the distribution functions of $X_1,\dots,X_n$ if the hypothesis H is true. Under H, $\mathbb F_n=G_n$. Consider the empirical distribution function based on $X_1,\dots,X_n$,
\begin{equation*}
\hat{\mathbb F}_n(x):=\frac 1 n \sum_{i=1}^n\I(X_i\le x),~x\in\R,
\end{equation*}
and the  Kolmogorov-Smirnov type statistic
\begin{equation}\label{KS0}
\KS_n:=\sqrt n\sup_{x\in\R}|\hat{\mathbb F}_n(x) -G_n(x)|
\end{equation}
as well as the Cram\'er-von-Mises type statistic 
\begin{equation}\label{CvM0}
\CvM_n:=n\int\big(\hat{\mathbb F}_n(x) -  G_n(x)\big)^2   G_n(\mathrm d x).
\end{equation}
Assume $F$ is continuous. In fact, the test statistics (\ref{KS0}) and (\ref{CvM0}) can be used to construct asymptotically exact  and  consistent tests for testing problem (\ref{testpr0}). See the results in Section \ref{simplegood1}. Note that the distribution function $G_n$ depends on the sample size $n$ and that the test statistics (\ref{KS0}) and (\ref{CvM0}) are no common Kolmogorov-Smirnov  and  Cram\'er-von-Mises statistics. The replacement of  $G_n$ with $G$ in the test statistics (\ref{KS0}) and (\ref{CvM0}), i.e., the common Kolmogorov-Smirnov  or  Cram\'er-von-Mises statistic, yields no applicable test for testing problem (\ref{testpr0}). This follows easily from the results in Section \ref{simplegood1}.

\medskip
In Section \ref{simplegood1}, testing goodness-of-fit based on not identically distributed data is discussed. In comparison to the situation above the model is more general  and the hypothesis is formulated in a more general way. Therefore, generalizations of the test statistics (\ref{KS0}) and (\ref{CvM0}) are considered and methods to determine critical values are suggested. The tests are applicable, e.g., in the situation presented above. It is shown that the tests are of asymptotically exact size and consistent. Section \ref{simplegood2} extends the model in Section \ref{simplegood1} to hypotheses of families of distributions. Analogous results are established. Section \ref{had} treats  hypotheses formulated by Hadamard differentiable functionals in the same setting. Similar results as in the goodness-of-fit setting hold for a broad class of tests. Special cases are the hypothesis of homogeneity, central symmetry and independence and applications are, e.g., situations with not identically distributed errors. In general, the presented Kolmogorov-Smirnov  and  Cram\'er-von-Mises type tests based on not identically distributed data and can be regarded as generalizations of the related Kolmogorov-Smirnov  and  Cram\'er-von-Mises type tests in the identically distributed case.

\medskip
The original Kolmogorov-Smirnov  and  Cram\'er-von-Mises statistics are statistics for testing goodness-of-fit,  described by Anderson and Darling \cite{and}. The idea of these statistics can be adopted for testing hypothesis of  homogeneity, symmetry and independence, see, e.g.,  Rosenblatt \cite{ros},  Butler \cite{but} and Blum et al. \cite{blu}, respectively. An important generalization of the mentioned goodness-of-fit statistics for hypotheses of families of distributions is considered by Stute et al. \cite{stu}. Some papers treat other generalizations of Kolmogorov-Smirnov type tests or Cram\'er-von-Mises type tests. Weiss \cite{wei} modified the classical Kolmogorov-Smirnov statistic for the application to correlated data. Chicheportiche et al. \cite{chi} consider the case of dependent and identically distributed observations and the Kolmogorov-Smirnov test as well as the Cram\'er-von-Mises test. An interesting paper is \cite{cha}. The authors Chatterjee and Sen study Kolmogorov-Smirnov type tests for testing the hypothesis of symmetry based on independent but not necessarily identically distributed real valued random variables and investigate the limit distributions of their test statistics with an application of the random walk model. 

\medskip
Further, some works deal with goodness-of-fit tests based on independent but not necessarily identically distributed data. Goodness-of-fit tests in this situation are discussed by Gosh and Basu \cite{gos} in a fully parametric setting. These tests use the density power divergence. Testing for normality in a model with independent but not necessarily identically distributed observations is considered by Sarkadi \cite{sar}. In the discrete case, Collings et al. \cite{col}, H\"usler \cite{hue} and Conover et al. \cite{con} treat a goodness-of-fit test for the Poisson assumption, the binomial test and the chi-square goodness-of-fit test, respectively, in the case of independent but not necessarily identically distributed data.

\section{Testing goodness-of-fit with hypotheses given by a specific distribution}\label{simplegood1}
The following generalization of the model introduced in Section \ref{einfuhr} is considered. Let $X_1,X_2,\dots$ be a sequence of independent but not necessarily identically distributed random vectors  with values in  $\R^m$, $m\in\N$. The random vector $X_i$ has the distribution function $F_i$ defined on $\overline \R^m$, $i\in\N$, with $\overline \R^l:=[-\infty,\infty]^l$, $l\in\N$. Henceforth, regard elements of $\overline \R^l$ as column vectors. Assume the existence of an uniformly continuous distribution function $F$ defined on $\overline \R^m$ with
\begin{equation*}
\lim_{i\rightarrow\infty}F_i=F~\text{uniformly on}~\overline\R^m.
\end{equation*}
For $n\in\N$ the given sample size, let  $G_n$ be another distribution function defined on $\overline\R^m$ as well as $G$, where $G$ is uniformly continuous and 
\begin{equation*}
\lim_{n\rightarrow\infty}G_n=G~\text{uniformly on}~\overline\R^m.
\end{equation*}
In addition, let $\alpha_{1,n},\dots,\alpha_{n,n}$ be some weights, in particular real numbers, such that 
\begin{equation}\label{seq}
\alpha_{i,n}\ge 0,~i=1,\dots,n,~\sum_{i=1}^n\alpha_{i,n}=1,~\lim_{n\rightarrow\infty}\sqrt n\max_{1\le i\le n}\alpha_{i,n}=0,~\lim_{n\rightarrow\infty}n\sum_{i=1}^n\alpha_{i,n}^2=\kappa\in[1,\infty).
\end{equation}
Consider the mixture distribution function
\begin{equation*}
\mathbb F_n(x):=\sum_{i=1}^n \alpha_{i,n}F_i(x),~x\in\overline\R^m,
\end{equation*}
based on the distribution functions $F_1,\dots,F_n$ of $X_1,\dots,X_n$. Suppose  $F_1,\dots,F_n$ are unknown, $ G_n$ and $\alpha_{1,n},\dots,\alpha_{n,n}$ are known and that the user has to verify the hypothesis of goodness-of-fit
\begin{equation}\label{testpr1}
\mathrm H_n:\mathbb F_n=G_n
\end{equation}
on the basis of the observations $X_1,\dots,X_n$.

\medskip
Define the weighted empirical distribution function
\begin{equation}\label{empmix}
\hat{\mathbb F}_n(x):=\sum_{i=1}^n\alpha_{i,n}\I(X_i\le x),~x\in\overline\R^m,
\end{equation}
based on $X_1,\dots,X_n$. For elements $a,b\in\overline\R^m$, $a=(a_1,\dots,a_m)'$, $b=(b_1,\dots,b_m)'$, let the inequality $a\le b$ be equivalent to $a_j\le b_j$ for $j=1,\dots,m$. Further, let $\mathbb U_n:=(U_n(x);x\in\overline\R^m)$ be the process
\begin{equation}\label{proc1}
U_n(x):=\sqrt n \big( \hat {\mathbb F}_n(x)-  G_n(x)\big),~x\in\overline\R^m.
\end{equation}
Consider the  Kolmogorov-Smirnov type statistic
\begin{equation}\label{KS1}
\KS_n:=\sup_{x\in\overline\R^m}|U_n(x)|=\sqrt n\sup_{x\in\overline\R^m}|\hat {\mathbb F}_n(x)-  G_n(x)|
\end{equation}
and the Cram\'er-von-Mises type statistic 
\begin{equation}\label{CvM1}
\CvM_n:=\int U^2_n(x)   G_n(\mathrm d x)=n\int\big(\hat {\mathbb F}_n(x)-  G_n(x)\big)^2   G_n(\mathrm d x).
\end{equation}
In general, the statistics  (\ref{KS1}) and (\ref{CvM1}) are not distribution free, neither in the case of $\mathbb F_n= G_n$. In order to approximate the distribution of the statistics $\KS_n$ and $\CvM_n$ if $\mathbb F_n= G_n$, a Monte-Carlo procedure is suggested. Therefor, simulate independently  observations with joint distribution function
\begin{equation*}
(x_1,\dots,x_n)\longmapsto \prod_{i=1}^n G_n(x_i),~(x_1,\dots,x_n)\in \times_{i=1}^n\overline\R^m.
\end{equation*}
Determine a significance level  $\alpha\in (0,1)$ and denote by $c_{n;1-\alpha}$ the $(1-\alpha)$-quantile of the distribution of $\KS_n$  and  by $d_{n;1-\alpha}$ the $(1-\alpha)$-quantile of the distribution of $\CvM_n$ if $(F_1,\dots,F_n)=( G_n,\dots, G_n)$.  The calculation procedure for practice is described above. Then, testing procedure
\begin{equation}\label{Test1}
\text{``Reject}~\mathrm H_n,~\text{iff}~\KS_n\ge c_{n;1-\alpha}\text{''}~\text{or testing procedure}~\text{``Reject}~\mathrm H_n,~\text{iff}~\CvM_n\ge d_{n;1-\alpha}\text{''} 
\end{equation}
is suggested. If  $\alpha_{i,n}=\frac 1 n$, $i=1,\dots,n$, testing procedure (\ref{Test1}), in particular the statistic (\ref{KS1}) or (\ref{CvM1}), is invariant under transformation of the data from the set
\begin{equation}\label{traf}
\begin{split}
\mathcal T _n:=&\lbrace T_n:\times_{i=1}^n\R^m\rightarrow \times_{i=1}^n\R^m;T_n(x_1,\dots,x_n)=(x_{\pi_n(1)},\dots,x_{\pi_n(n)}),\\
&\pi_n:\lbrace 1,\dots,n\rbrace\rightarrow\lbrace 1,\dots,n\rbrace~\text{bijective}\big\rbrace.
\end{split}
\end{equation}
\begin{rmk}
Assume another sequence of known distribution functions $(H_i)_{i\in\N}$ defined on $\overline \R^m$ is given and  $ G_n=\sum_{i=1}^n \alpha_{i,n}H_i$. 
\begin{itemize}
\item[a)]Consider the model introduced in Section \ref{einfuhr}. Put $\alpha_{i,n}=\frac 1 n$,  $X_i=e_i(Y_i,Z_i)$ and let $H_i$ be the distribution function of $e_i(Y_0,Z_i)$, $i=1,\dots,n$. Then, the test (\ref{Test1}) is applicable to testing problem (\ref{testpr0}). The test is of asymptotically exact size $\alpha$ and  consistent  for testing problem (\ref{testpr0}). This follows from the results in this section. 
\item[b)] The test (\ref{Test1}) is applicable to the hypothesis 
\begin{equation*}
\forall i\in\lbrace 1,\dots,n\rbrace :F_i=H_i.
\end{equation*}
 Putting $\alpha_{i,n}=\frac 1 n$, $i=1,\dots,n$, the test (\ref{Test1}) is also applicable to the hypothesis 
\begin{equation*}
\exists \pi_n:\lbrace 1,\dots,n\rbrace\rightarrow\lbrace 1,\dots,n\rbrace~\text{bijective}~\forall i\in\lbrace 1,\dots,n\rbrace :F_i=H_{\pi_n(i)}.
\end{equation*}
The test is of asymptotically exact size $\alpha$ and  consistent  with respect to suitable alternatives to this hypotheses. See the results in this section. 
\item[c)] Assume $F_i=F$ and $H_i=G$, $i\in\N$. Putting $\alpha_{i,n}=\frac 1 n$, $i=1,\dots,n$, the hypothesis (\ref{testpr1}) is a common hypothesis of goodness-of-fit 
\begin{equation*}
\mathrm H:F=G,
\end{equation*}
the test statistic (\ref{KS1}) or (\ref{CvM1}) is a common Kolmogorov-Smirnov  or  Cram\'er-von-Mises statistic and the test (\ref{Test1}) is a common  Kolmogorov-Smirnov  or  Cram\'er-von-Mises test in a multivariate setting. In this sense, a generalization of the identically distributed case  is considered.
\end{itemize}
\end{rmk}
\begin{rmk}
Note that the distribution function $G_n$ depends on the sample size $n$ and that the test statistics (\ref{KS1}) and (\ref{CvM1}) are no common Kolmogorov-Smirnov  and  Cram\'er-von-Mises statistics, neither in the case of $\alpha_{i,n}=\frac 1 n$, $i=1,\dots,n$. In fact, the replacement of  $G_n$ with $G$ in the test statistics (\ref{KS1}) and (\ref{CvM1}), i.e., the common Kolmogorov-Smirnov  or  Cram\'er-von-Mises statistic, yields no  applicable test for testing problem (\ref{testpr1}). This follows easily from the results in this section.
\end{rmk}

\subsection{Limit results under the hypothesis}

A possible first step to find limit distributions of  statistics of type Kolmogorov-Smirnov and Cram\'er-von-Mises is to show the convergence in distribution of the process in the supremum and integral to a limit element. Let $T$ be a  non-empty set and let $\mathbb W_n:=(W_n(t);t\in T)$ be a stochastic process on a probability space $(\Omega,\AA,P)$ with sample paths in 
\begin{equation*}
 \ell^\infty(T):=\lbrace f:T\rightarrow\R; ||f||_T<\infty\rbrace,
\end{equation*}
where $||f||_T:=\sup_{t\in T}|f(t)|$, $f:T\rightarrow\R$. The map $(f_1,f_2)\mapsto ||f_1-f_2||_T$, $(f_1,f_2)\in\ell^\infty(T)\times \ell^\infty(T)$, defines a metric on $\ell^\infty(T)$. Let $\mathbb W:=(W(t);t\in T)$ be another stochastic process on the probability space  with sample paths in $\ell^\infty(T)$. Then, $\mathbb W_n$ converges in distribution to $\mathbb W$, in notation $\mathbb W_n\overset{\mathrm d}{\rightarrow}\mathbb W$ as $n\rightarrow\infty$, iff $\mathbb W$ is $(\AA,\BB( \ell^\infty(T))$-measurable and
\begin{equation*}
\forall f: \ell^\infty(T)\rightarrow \R~\text{continuous and bounded}:\lim_{n\rightarrow\infty}\E^*\big(f(\mathbb W_n)\big)=\E\big(f(\mathbb W)\big).
\end{equation*}
$\BB( \ell^\infty(T))$ denotes the Borel $\sigma$-field on $( \ell^\infty(T),||\cdot||_T)$. In addition, $\E^*$ is the outer expectation. Details are given by van der Vaart and Wellner \cite{van}.

\medskip
Let $\rho$ be a pseudometric on $T$ such that $(T,\rho)$ is a totally bounded pseudometric space. Then,
\begin{equation*}
U_b(T,\rho):=\lbrace f\in \ell^\infty(T);f~\text{is uniformly $\rho$-continuous}\rbrace
\end{equation*}
is a separable subspace of $\ell^\infty(T)$. If the process $\mathbb W$ has sample paths a.s. in $ U_b(T,\rho)$, the process $\mathbb W$ is $(\AA,\BB( \ell^\infty(T))$-measurable and the distribution of $\mathbb W$ is uniquely determined by its finite dimensional marginal distributions.

\medskip
The asymptotic equicontinuity is crucial for the convergence in distribution of the process $\mathbb W_n$, that is
\begin{equation*}
\forall\varepsilon>0:\lim_{\delta\downarrow 0}\limsup_{n\rightarrow\infty}P^*\bigg(\sup_{s,t\in T,\rho(s,t)\le\delta}|W_n(s)-W_n(t)|>\varepsilon\bigg)=0,
\end{equation*}
where $P^*$ is the outer probability. See \cite{van} for details. Assume the process  $\mathbb W_n$ is asymptotically equicontinuous and all finite dimensional projections of  $\mathbb W_n$ converge in distribution to finite dimensional random vectors. Then, there exists a process $\mathbb W$ with sample paths a.s. in $U_b(T,\rho)$ such that $\mathbb W_n\overset{\mathrm d}{\rightarrow}\mathbb W$ as $n\rightarrow\infty$. Particularly, the process $\mathbb W$ is $(\AA,\BB( \ell^\infty(T))$-measurable and the distribution of $\mathbb W$ is uniquely determined by its finite dimensional marginal distributions, the distributions of those finite dimensional random vectors. See Theorem 1.5.4 in \cite{van}. 

\medskip
Now, a triangular array $\xi_{1,n},\dots,\xi_{n,n}$ of  row-wise independent but not necessarily identically distributed  random vectors  with values in  $\R^m$ on the probability space is given.  Let $K_{i,n}$ be the distribution function of $\xi_{i,n}$ defined on $\overline\R^m$, $i=1,\dots,n$. In the relevant applications, $K_{i,n}$ depends either only on $ i$ or only on $n$, $i=1,\dots,n$. Assume  the existence of an uniformly continuous distribution function $ K$ defined on $\overline\R^m$ such that 
\begin{equation*}
\forall \varepsilon>0~\exists i_\varepsilon\in\N~\exists n_\varepsilon\in\N,~i_\varepsilon\le n_\varepsilon,~\forall i>i_\varepsilon~\forall n>n_\varepsilon,~i\le n:\sup_{x\in \overline\R^m}|K_{i,n}(x)- K(x)|\le \varepsilon.
\end{equation*}
The interesting process is the weighted empirical process  $\mathbb W_n$ given by
\begin{equation*}
W_n(x):=\sqrt n\bigg(\sum_{i=1}^n\alpha_{i,n}\I(\xi_{i,n}\le x)- \sum_{i=1}^n\alpha_{i,n}P(\xi_{i,n}\le x)\bigg),~x\in \overline\R^m.
\end{equation*}
Such types of processes, partly special cases, are studied by Shorack \cite{sho}, Shorack and Wellner \cite{sho2}, Alexander \cite{ale}, Pollard \cite{pol}, Ziegler \cite{zie} and Kosorok \cite{kos1}, \cite{kos2}.  Let $\Phi$ be the distribution function of the $m$-dimensional standard normal distribution. Define a metric
\begin{equation}\label{metr}
\rho(x,y):=\int|\I(w\le x)-\I(w\le y)|\Phi(\mathrm d w),~x,y\in \overline\R^m.
\end{equation}
In fact, $(\overline\R^m,\rho)$ is a totally bounded metric space. With the help of the Vapnik-\v{C}hervonenkis class  theory, it can be shown that the following result holds.
\begin{lem} \label{lem0}
The process $\mathbb W_n$ is asymptotically equicontinuous with respect to the metric space $(\overline\R^m,\rho)$.
\end{lem}
Now, the basic results under the hypothesis will be formulated. Therefor, let $\mathbb U=(U(x);x\in\overline\R^m)$ be a Gaussian process with expectation function identically equal to zero, covariance function
\begin{equation}\label{cov1}
c(x,y):=\kappa\Big(  G\big(\min(x,y)\big)-  G(x)  G(y)\Big),~x,y\in\overline\R^m,
\end{equation}
and a.s. uniformly $\rho$-continuous sample paths. For elements $a,b\in\overline\R^m$, $a=(a_1,\dots,a_m)'$, $b=(b_1,\dots,b_m)'$, put $\min(a,b):=(\min(a_1,b_1),\dots,\min(a_m,b_m))'$.
\begin{thm} \label{thm1} Assume ${\mathbb F}_n= G_n$  for $n$ sufficiently large. Then,
\begin{equation*}
\mathbb U_n\overset{\mathrm d}{\longrightarrow} \mathbb U~\text{as}~n\rightarrow\infty.
\end{equation*}
\end{thm}
\begin{cor}\label{cor1}
Assume ${\mathbb F}_n= G_n$  for $n$ sufficiently large. It is
\begin{equation*}
\KS_n\overset{\mathrm d}{\longrightarrow}\sup_{x\in\overline\R^m}|U(x)|~\text{and}~\CvM_n\overset{\mathrm d}{\longrightarrow}\int U^2(x)   G(\mathrm d x)~\text{as}~n\rightarrow\infty.
\end{equation*}
\end{cor}
In order to show that testing procedure (\ref{Test1}) works asymptotically, it is necessary to study the asymptotic distribution of the statistics  (\ref{KS1}) and (\ref{CvM1}) if $(F_1,\dots,F_n)=( G_n,\dots, G_n)$, too. Well, let $X_1^{(n)},X_2^{(n)},\dots$ be a sequence of independent and identically distributed random vectors with underlying distribution function $ G_n$ and let $\mathbb U_n^{(n)}=(U_n^{(n)}(x);x\in\overline\R^m)$ be the process (\ref{proc1}) based on the random vectors $X_1^{(n)},\dots,X_n^{(n)}$. Denote by $\KS_n^{(n)}$ and $\CvM_n^{(n)}$ the statistics (\ref{KS1}) and (\ref{CvM1}) based on the random vectors $X_1^{(n)},\dots,X_n^{(n)}$, respectively. 
\begin{thm} \label{thm2} It is
\begin{equation*}
\mathbb U_n^{(n)}\overset{\mathrm d}{\longrightarrow} \mathbb U~\text{as}~n\rightarrow\infty.
\end{equation*}
\end{thm}
\begin{cor}\label{cor2}
It follows that
\begin{equation*}
\KS_n^{(n)}\overset{\mathrm d}{\longrightarrow}\sup_{x\in\overline\R^m}|U(x)|~\text{and}~\CvM_n^{(n)}\overset{\mathrm d}{\longrightarrow}\int U^2(x)  G(\mathrm d x)~\text{as}~n\rightarrow\infty.
\end{equation*}
\end{cor}
The following result implies that the test (\ref{Test1}) is a test of asymptotically exact size $\alpha$.
\begin{cor}\label{cor3}
Suppose ${\mathbb F}_n= G_n$  for $n$ sufficiently large. Then,
\begin{equation*} 
\lim_{n\rightarrow\infty}P(\KS_n\ge c_{n;1-\alpha})=\alpha~\text{and}~\lim_{n\rightarrow\infty}P(\CvM_n\ge d_{n;1-\alpha})=\alpha.
\end{equation*}
\end{cor}

\subsection{Limit results under alternatives}
Now, a general result gives the limit behavior  in probability of the test statistics   (\ref{KS1}) and (\ref{CvM1}). The result holds under the hypotheses as well as under alternatives.
\begin{thm} \label{thm3}It is
\begin{equation*}
\forall\varepsilon >0:P\bigg(\frac{1}{\sqrt n}\KS_n\ge\sup_{x\in\overline\R^m}|{ F}(x)- G(x)|-\varepsilon+o_p(1)\bigg)\longrightarrow 1~\text{as}~n\rightarrow\infty,
\end{equation*}
and
\begin{equation*}
\forall\varepsilon >0:P\bigg(\frac 1 n \CvM_n\ge \int\big({ F}(x)- G(x)\big)^2   G(\mathrm d x)-\varepsilon+o_p(1) \bigg)\longrightarrow 1~\text{as}~n\rightarrow\infty.
\end{equation*}
\end{thm}
The following corollary gives information about the consistency of the test (\ref{Test1}).
\begin{cor}\label{cor4}
Suppose $ F-  G$ does not vanish everywhere or $ G$-almost everywhere. Then,
\begin{equation*} 
\lim_{n\rightarrow\infty}P(\KS_n\ge c_{n;1-\alpha})=1~\text{or}~\lim_{n\rightarrow\infty}P(\CvM_n\ge d_{n;1-\alpha})=1.
\end{equation*}
\end{cor}

\section{Testing goodness-of-fit with hypotheses given by a family of distributions}\label{simplegood2}

In this section, the model introduced in Section \ref{simplegood1} will be extended by adding a parameter. Under regularity conditions, all results still hold in this model.
Regard the definitions at the beginning of Section \ref{simplegood1}. For $d\in\N$ and a given non-empty  set $\Theta\subseteq \R^d$, maps $ G_n:\overline\R^m\times\Theta\rightarrow \R$ and $ G:\overline\R^m\times\Theta\rightarrow \R$ are given such that for all $\vartheta\in\Theta$  $ G_n(\cdot,\vartheta)$ and $ G(\cdot,\vartheta)$ are distribution functions, $ G(\cdot,\vartheta)$ is uniformly continuous and
\begin{equation*}
\forall\vartheta\in\Theta:\lim_{n\rightarrow\infty}G_n(\cdot,\vartheta)=G(\cdot,\vartheta)~\text{uniformly on}~\overline\R^m.
\end{equation*}
Suppose $F_1,\dots,F_n$ are unknown, $ G_n$ and $\alpha_{1,n},\dots,\alpha_{n,n}$ are known and that the user has to verify the hypothesis of goodness-of-fit
\begin{equation}\label{testpr2}
\mathrm H_n:\exists\vartheta\in\Theta:\mathbb F_n=G_n(\cdot,\vartheta)
\end{equation}
on the basis of the observations $X_1,\dots,X_n$.

\medskip
Define a process $\mathbb U_n(\vartheta):=(U_n(x,\vartheta);x\in\overline\R^m)$ by
\begin{equation}\label{proc2}
U_n(x,\vartheta):=\sqrt n \big( \mathbb F_n(x)-  G_n(x, \vartheta)\big),~x\in\overline\R^m.
\end{equation}
Further, let $t_n$ be a $(\otimes_{j=1}^n\BB^m,\BB_{\Theta}^d)$-measurable map $t_n:\times_{j=1}^n\R^m\rightarrow \Theta$ and put  $\hat\vartheta_n:=t_n(X_1,\dots,X_n)$. $\BB^m$ denotes the Borel $\sigma$-field on $\R^m$ and $\BB_{\Theta}^d$ denotes the Borel $\sigma$-field on $\R^d$ restricted on $\Theta$. Consider the  Kolmogorov-Smirnov type statistic
\begin{equation}\label{KS2}
\KS_n:=\sup_{x\in\overline\R^m}|U_n(x,\hat \vartheta_n)|=\sqrt n\sup_{x\in\overline\R^m}|\mathbb F_n(x)-  G_n(x,\hat \vartheta_n)|
\end{equation}
and the Cram\'er-von-Mises type statistic 
\begin{equation}\label{CvM2}
\CvM_n:=\int U^2_n(x,\hat \vartheta_n)   G_n(\mathrm d x,\hat \vartheta_n)=n\int\big(\mathbb F_n(x)-  G_n(x,\hat \vartheta_n)\big)^2   G_n(\mathrm d x,\hat \vartheta_n).
\end{equation}
In general, the statistics  (\ref{KS2}) and (\ref{CvM2}) are not distribution free, neither in the case of the existence of a $\vartheta\in\Theta$ with ${\mathbb F}_n= G_n(\cdot,\vartheta)$.  In order to approximate the distribution of the statistics $\KS_n$ and $\CvM_n$ if the existence of a $\vartheta\in\Theta$ with ${\mathbb F}_n= G_n(\cdot,\vartheta)$ is fulfilled, a Monte-Carlo procedure is suggested. Therefor, simulate independently  observations with joint distribution function
\begin{equation*}
(x_1,\dots,x_n)\longmapsto \prod_{i=1}^n\mathbb G_n(x_i,\hat\vartheta_n),~(x_1,\dots,x_n)\in \times_{i=1}^n\overline\R^m.
\end{equation*}
Determine a significance level  $\alpha\in (0,1)$ and denote by $c_{n;1-\alpha}$ the $(1-\alpha)$-quantile of the distribution of  $\KS_n$ and by $d_{n;1-\alpha}$ the $(1-\alpha)$-quantile of the distribution of $\CvM_n$ if $(F_1,\dots,F_n)=( G_n(\cdot,\hat\vartheta_n),\dots, G_n(\cdot,\hat\vartheta_n))$. The calculation procedure for practice is described above. The testing procedure
\begin{equation}\label{Test2}
\text{``Reject}~\mathrm H_n,~\text{iff}~\KS_n\ge c_{n;1-\alpha}\text{''}~\text{or testing procedure}~\text{``Reject}~\mathrm H_n,~\text{iff}~\CvM_n\ge d_{n;1-\alpha}\text{''} 
\end{equation}
is suggested. 
\begin{rmk}
Assume another sequence of known maps  $ H_i:\overline\R^m\times\Theta\rightarrow \R$, $i\in\N$, is given such that for all $\vartheta\in\Theta$  $ H_i(\cdot,\vartheta)$ is a distribution function and $ G_n=\sum_{i=1}^n \alpha_{i,n}H_i$.
\begin{itemize}
\item[a)]Extend the model introduced in Section \ref{einfuhr} in the following way. Assume the hypothesis is given by a family of distributions $\lbrace\mathcal L(Y_0,\vartheta_1); \vartheta_1\in\Theta_1\rbrace$, with a non-empty  set  $\Theta_1\subset \R^{d_1}$, $d_1\in\N$, i.e., the user has to treat the testing problem of goodness-of-fit
\begin{equation*}
\mathrm H:\exists \vartheta_1\in\Theta_1:\mathcal L(Y_1)=\mathcal L(Y_0,\vartheta_1),~\mathrm K:\forall \vartheta_1\in\Theta_1:\mathcal L(Y_1)\neq \mathcal L(Y_0,\vartheta_1).
\end{equation*}
In addition, suppose the error variables $(Z_i)_{i\in\N}$ come from  families of distributions $\lbrace(\mathcal L (Z_i,\vartheta_2))_{i\in\N};\vartheta_2\in\Theta_2\rbrace$, $i\in\N$, with a non-empty  set $\Theta_2\subset \R^{d_2}$, $d_2\in\N$, and that the true parameter $\vartheta_2\in\Theta_2$ is unknown. Moreover, assume the error functions $(e_i)_{i\in\N}$ come from families of functions $\lbrace (e_i(\cdot,\vartheta_3))_{i\in\N}; \vartheta_3\in\Theta_3\rbrace$, $i\in\N$, with a non-empty  set  $\Theta_3\subset \R^{d_3}$, $d_3\in\N$, and that the true parameter $\vartheta_3\in\Theta_3$ is unknown, too. Put $d=d_1+d_2+d_3$, $\Theta=\lbrace (\vartheta_1',\vartheta_2',\vartheta_3')';  (\vartheta_1,\vartheta_2,\vartheta_3)\in\Theta_1\times\Theta_2\times\Theta_3\rbrace$, $\alpha_{i,n}=\frac 1 n$,  $X_i=e_i((Y_i,Z_i),\vartheta_3)$ and let $H_i$ be the distribution function of $e_i((Y_0,Z_i,\vartheta_3))$, $i=1,\dots,n$, where $\vartheta_3\in\Theta_3$ is the true parameter. Then, the test (\ref{Test2}) is applicable to this testing problem. Under regularity conditions, it follows from the results in this section that the test is of asymptotically exact size $\alpha$ and consistent.
\item[b)] The test (\ref{Test2}) is applicable to the hypothesis 
\begin{equation*}
\exists \vartheta\in\Theta~\forall i\in\lbrace 1,\dots,n\rbrace :F_i=H_i(\cdot,\vartheta).
\end{equation*}
 Putting $\alpha_{i,n}=\frac 1 n$, $i=1,\dots,n$,  the test (\ref{Test2}) is also applicable to the hypothesis 
\begin{equation*}
\exists \vartheta\in\Theta~\exists \pi_n:\lbrace 1,\dots,n\rbrace\rightarrow\lbrace 1,\dots,n\rbrace~\text{bijective}~\forall i\in\lbrace 1,\dots,n\rbrace :F_i=H_{\pi_n(i)}(\cdot,\vartheta).
\end{equation*}
Under regularity conditions, it follows from the results in this section that the test is of asymptotically exact size $\alpha$ and consistent with respect to suitable alternatives to this hypotheses.
\item[c)] Assume $F_i=F$ and $H_i=G$, $i\in\N$. Putting $\alpha_{i,n}=\frac 1 n$, $i=1,\dots,n$, the hypothesis (\ref{testpr2}) is a common hypothesis of goodness-of-fit with families of distributions
\begin{equation*}
\mathrm H:\exists \vartheta\in\Theta:F=G(\cdot,\vartheta),
\end{equation*}
the test statistic (\ref{KS2}) or (\ref{CvM2}) is a common Kolmogorov-Smirnov  or  Cram\'er-von-Mises statistic with estimated parameter and the test (\ref{Test2}) is a common  Kolmogorov-Smirnov  or  Cram\'er-von-Mises test  with families of distributions in a multivariate setting. In this sense, a generalization of the identically distributed case  is considered.
\end{itemize}
\end{rmk}
\begin{rmk}
Note that the distribution function $G_n$ depends on the sample size $n$ and that the test statistics (\ref{KS2}) and (\ref{CvM2}) are no common Kolmogorov-Smirnov  and  Cram\'er-von-Mises statisticswith estimated parameter, neither in the case of $\alpha_{i,n}=\frac 1 n$, $i=1,\dots,n$. In fact, the replacement of  $G_n$ with $G$ in the test statistics (\ref{KS2}) and (\ref{CvM2}), i.e., the common Kolmogorov-Smirnov  or  Cram\'er-von-Mises statistic, yields no  applicable test for testing problem (\ref{testpr2}). This follows easily from the results in this section.
\end{rmk}

\subsection{Limit results under the hypothesis}

Fix a parameter $\vartheta\in\Theta$. In order to establish the results of Section  \ref{simplegood1} to this model, some technical assumptions are required. The first two are regularity conditions to the families of distributions given by the hypothesis.
\begin{itemize}
\item[(C1)]There exists an open and convex set $U_{\vartheta }\subset\Theta$ such that $\vartheta\in U_{\vartheta}$ and for all $x\in\overline\R^m$ $ G_n(x,\cdot):\Theta\rightarrow \R$ is continuous differentiable on $U_{\vartheta }$.
\end{itemize}
Define $ g_n(x,\tilde\vartheta):=\nabla G_n(x, \tilde\vartheta)$, $(x,\tilde\vartheta)\in\overline\R^m\times U_{\vartheta }$, where $\nabla:=(\frac {\partial}{\partial\tilde\vartheta_1},\dots,\frac {\partial}{\partial\tilde\vartheta_d})'$.
\begin{itemize}
\item[(C2)]There exists a map $ g:\overline\R^m\times U_{\vartheta}\rightarrow \R^d$ such that $\lim_{n\rightarrow\infty} g_n= g$ uniformly on $\overline\R^m\times U_{\vartheta}$, $ g$  is uniformly continuous  on $\overline\R^m\times U_{\vartheta}$ and $ g(\cdot,\vartheta):\overline\R^m\rightarrow \R^d$ has bounded components on $\overline\R^m$. 
\end{itemize}
Now, a regularity condition to the parameter estimator is formulated. Particular, an asymptotic expansion is needed. The condition is fulfilled, e.g., for a weighted maximum-likelihood estimator, see Remark \ref{bem3}.
\begin{itemize}
\item[(C3)] If ${\mathbb F}_n= G_n(\cdot,\vartheta)$ for $n$ sufficiently large, there exists a $(\BB^m,\BB^d)$-measurable map $\ell_n(\cdot,\vartheta):\R^m \rightarrow \R^d$ such that $\int \ell_n(x,\vartheta) F_i(\mathrm d x)$ and $\int \ell_n(x,\vartheta)\ell_n'(x,\vartheta) F_i(\mathrm d x)$ exist and have finite components for $i=1,\dots,n$, $\int \ell_n(x,\vartheta){\mathbb F}_n(\mathrm d x)=0$ and
\begin{equation*}
\sqrt n (\hat \vartheta_n-\vartheta)={\sqrt n}\sum_{i=1}^n\alpha_{i,n} \ell_n(X_i,\vartheta)+o_p(1)~\text{as}~n\rightarrow\infty.
\end{equation*}
\end{itemize}
Put
\begin{equation*}
 v_{i,n}( \vartheta):=\int \Big(\ell_n(x, \vartheta)-\int \ell_n(y,\vartheta) F_i(\mathrm d y)\Big)\Big(\ell_n(x, \vartheta)-\int \ell_n(y,\vartheta) F_i(\mathrm d y)\Big)'{ F}_i(\mathrm d x),~i\in\N,
\end{equation*}
and 
\begin{equation*}
 w_{i,n}(y, \vartheta):= \int \ell_n(x, \vartheta)\I(x\le y){ F}_i(\mathrm d x)-\int \ell_n(x,\vartheta) F_i(\mathrm d x)F_i(y),~y\in \overline\R^m,~i\in\N.
\end{equation*}
The following condition guarantees the convergence of the covariance function of the  process $\mathbb U_n(\vartheta)$  and that Lindeberg's condition is fulfilled.
\begin{itemize}
\item[(C4)] If ${\mathbb F}_n= G_n(\cdot,\vartheta)$ for $n$ sufficiently large, there exists a $ v(\vartheta)\in\R^{d\times d}$ and a  $ w(\cdot,\vartheta):\overline\R^m\rightarrow\R^d$ such that  $\lim_{n\rightarrow\infty} n\sum_{i=1}^n\alpha_{i,n}^2v_{i,n}( \vartheta)=\kappa v( \vartheta)$ and   $\lim_{n\rightarrow\infty} n\sum_{i=1}^n\alpha_{i,n}^2 w_{i,n}(x, \vartheta)=\kappa w(x, \vartheta)$ for all $x\in\overline\R^m$. Well, if ${\mathbb F}_n= G_n(\cdot,\vartheta)$ for $n$ sufficiently large,
\begin{equation*}
\begin{split}
\forall a\in\R^d~\forall t>0:&\lim_{n\rightarrow\infty} n \sum_{i=1}^n \alpha_{i,n}^2\int\bigg(1+ a'\Big(\ell_n(x,\vartheta)-\int \ell_n(y,\vartheta)F_i(\mathrm d y)\Big)\\
&\Big(\ell_n(x,\vartheta)-\int \ell_n(y,\vartheta)F_i(\mathrm d y)\Big)'a\bigg)\\
&\I\bigg((\sqrt n\max_{1\le j\le n}\alpha_{j,n})^2a'\Big(\ell_n(x,\vartheta)-\int \ell_n(y,\vartheta)F_i(\mathrm d y)\Big)\\
&\Big(\ell_n(x,\vartheta)-\int \ell_n(y,\vartheta)F_i(\mathrm d y)\Big)'a>  t\bigg) F_i( \mathrm d x)=0.
\end{split}
\end{equation*}
\end{itemize}

Ignore the $o_P(1)$ term in (C3). Then, the estimator $\hat \vartheta_n$ is invariant under transformation of the data from the set $\mathcal T _n$ if  (C3) holds and $\alpha_{i,n}=\frac 1 n$, $i=1,\dots,n$, and testing procedure (\ref{Test2}), in particular the statistic (\ref{KS2}) or (\ref{CvM2}), has this invariance property, too.

\medskip
The following lemma shows the weak consistency of the parameter estimator.
\begin{lem} \label{lem2} Assume ${\mathbb F}_n= G_n(\cdot,\vartheta)$ for $n$ sufficiently large and (C1) - (C4). Then,\linebreak ${\sqrt n}\sum_{i=1}^n \alpha_{i,n}(\ell_n(X_i,\vartheta)-\int \ell_n(x,\vartheta)F_i(\mathrm d x))$ converges in distribution to a centered  $d$-dimen\-sional normal distribution as $n\rightarrow\infty$ and
\begin{equation*}
\lim_{n\rightarrow\infty}\hat \vartheta_n=\vartheta~\text{in probability}.
\end{equation*}
\end{lem}
Now, the basic results under the hypothesis will be formulated. Therefor, let $\mathbb U(\vartheta)=(U(x,\vartheta);x\in\overline\R^m)$ be a Gaussian process with expectation function identically equal to zero, covariance function
\begin{equation}\label{cov3}
\begin{split}
c(x,y,\vartheta):=&\kappa\Big(  G\big(\min(x,y), \vartheta\big)-  G(x, \vartheta) G(y, \vartheta)\\
&- g'(x, \vartheta) w(y, \vartheta)- g'(y, \vartheta) w(x, \vartheta)+ g'(x, \vartheta) v(\vartheta) g(y, \vartheta)\Big),~x,y\in\overline\R^m,
\end{split}
\end{equation}
and a.s. uniformly $\rho$-continuous sample paths.
\begin{thm}\label{thm4} Assume (C1) - (C4) and ${\mathbb F}_n= G_n(\cdot,\vartheta)$  for $n$ sufficiently large. Then,
\begin{equation*}
\mathbb U_n(\hat \vartheta_n)\overset{\mathrm d}{\longrightarrow} \mathbb U(\vartheta)~\text{as}~n\rightarrow\infty.
\end{equation*}
\end{thm}
\begin{cor}\label{cor5}Assume (C1) - (C4) and ${\mathbb F}_n= G_n(\cdot,\vartheta)$  for $n$ sufficiently large. It follows that
\begin{equation*}
\KS_n\overset{\mathrm d}{\longrightarrow}\sup_{x\in\overline\R^m}|U(x,\vartheta)|~\text{and}~\CvM_n\overset{\mathrm d}{\longrightarrow}\int U^2(x,\vartheta)   G(\mathrm d x,\vartheta)~\text{as}~n\rightarrow\infty.
\end{equation*}
\end{cor}
In order to show that testing procedure (\ref{Test2}) works asymptotically, it is necessary to study the asymptotic distribution of the statistics  (\ref{KS2}) and (\ref{CvM2}) if $(F_1,\dots,F_n) =\linebreak(  G_n(\cdot,\hat\vartheta_n),\dots, G_n(\cdot,\hat\vartheta_n))$, too.

\medskip
Well, let $(\vartheta_n)_{n\in\N}$ be an arbitrary sequence of parameters with $\vartheta_n\in\Theta$ for all $n\in\N$ and $\lim_{n\rightarrow\infty}\vartheta_n=\vartheta$. Further, let $X_1^{(n)},X_2^{(n)},\dots$ be a sequence of independent and identically distributed random vectors with underlying distribution function $\mathbb G_n(\cdot,\vartheta_n)$, put  $\hat \vartheta_n^{(n)}:=t_n(X_1^{(n)},\dots,X_n^{(n)})$ and let $\mathbb U_n^{(n)}(\hat \vartheta_n^{(n)})=(U_n^{(n)}(x,\hat \vartheta_n^{(n)});x\in\overline\R^m)$ be the  process (\ref{proc2}) based on the random vectors $X_1^{(n)},\dots,X_n^{(n)}$. Denote by $\KS_n^{(n)}$ and $\CvM_n^{(n)}$ the statistics (\ref{KS2}) and (\ref{CvM2}) based on the random vectors $X_1^{(n)},\dots,X_n^{(n)}$, respectively. In addition to the conditions (C1) and (C2), consider the following two conditions. This conditions are modifications of the conditions (C3) and (C4).
\begin{itemize}
\item[(C5)] There exists a map $\ell_n:\R^m\times U_{\vartheta} \rightarrow \R^d$ such that for all $\tilde\vartheta\in U_{\vartheta} $ $\ell_n(\cdot,\tilde\vartheta)$ is $( \BB^m,\BB^d)$-measurable, $\int \ell_n(x,\tilde\vartheta)  G_n(\mathrm d x,\tilde \vartheta)$ and $\int \ell_n(x,\tilde\vartheta)\ell_n'(x,\tilde\vartheta)  G_n(\mathrm d x,\tilde \vartheta)$ exist  and have finite components,  $\int \ell_n(x,\tilde\vartheta) G_n(\mathrm d x,\tilde \vartheta) =0$ and
\begin{equation*}
\sqrt n (\hat \vartheta_n^{(n)}-\vartheta_n)={\sqrt n}\sum_{i=1}^n \alpha_{i,n}\ell_n(X_i^{(n)},\vartheta_n)+o_p(1)~\text{as}~n\rightarrow\infty.
\end{equation*}
\end{itemize}
Put 
\begin{equation*}
 v_{n}^{()}( \tilde\vartheta):=\int \ell_n(x, \tilde\vartheta)\ell'_n(x, \tilde \vartheta) G_n(\mathrm d x,\tilde\vartheta),~\tilde\vartheta\in U_{\vartheta},
\end{equation*}
and 
\begin{equation*}
 w_{n}^{()}(y, \tilde\vartheta):= \int \ell_n(x, \tilde\vartheta) \I(x\le y) G_n(\mathrm d x,\tilde\vartheta),~y\in \overline\R^m,~\tilde\vartheta\in U_{\vartheta}.
\end{equation*}
\begin{itemize}
\item[(C6)] There exists a  $ v^{()}(\vartheta)\in\R^{d\times d}$ such that $\lim_{n\rightarrow\infty} v^{()}_n( \vartheta_n)= v^{()}( \vartheta)$. In addition, there exists a map $ w^{()}(\cdot,\vartheta):\overline\R^m\rightarrow\R^d$ such that for all $x\in\overline\R^m$ $\lim_{n\rightarrow\infty} w^{()}_n(x, \vartheta_n)= w^{()}(x, \vartheta)$. Well,
\begin{equation*}
\begin{split}
\forall a\in\R^d~\forall t>0:&\lim_{n\rightarrow\infty}\int \big(1+a'\ell_n(x,\vartheta_n)\ell_n'(x,\vartheta_n)a\big)\\
&\I\big((\sqrt n\max_{1\le j\le n}\alpha_{j,n})^2 a'\ell_n(x,\vartheta_n)\ell'_n(x,\vartheta_n)a>  t\big)   G_n( \mathrm d x,\vartheta_n)=0.
\end{split}
\end{equation*}
\end{itemize}
\begin{rmk}\label{bem3}
Under regularity conditions, a weighted maximum-likelihood estimation function $t_n$ fulfills the mentioned conditions. Weighted maximum-likelihood estimators are considered, e.g., by Wang and Zidek \cite{wan}, Hu and Zidek \cite{hu} and Wang, van Eeden and Zidek \cite{wan2}. Assume $\vartheta\in\Theta$ has the property ${\mathbb F}_n= G_n(\cdot,\vartheta)$  for $n$ sufficiently large. Consider $n$ sufficiently large. Suppose the existence of a density $h_n(\cdot,\tilde\vartheta)$ of  $ G_n(\cdot,\tilde\vartheta)$ for all $\tilde\vartheta\in\Theta$  with respect to a dominating $\sigma$-finite measure $\mu$ on $( \R^m,\BB^m)$ and assume  for all $x\in\R^m$ the map $\tilde \vartheta\mapsto h_n(x,\tilde\vartheta)$, $\tilde\vartheta\in\Theta$,  is two times continuous differentiable. Moreover, let $t_n$ be a weighted maximum-likelihood estimation function for $\vartheta$ with the property
\begin{equation*}
\sup_{\tilde\vartheta\in\Theta}\prod_{i=1}^n h_n^{ \alpha_{i,n}}(x_i,\tilde\vartheta)=\prod_{i=1}^n h_n^{ \alpha_{i,n}}\big(x_i,t_n(x_1,\dots,x_n)\big),~(x_1,\dots,x_n)\in \times_{i=1}^n\R^m.
\end{equation*}
Put $L_n(\cdot,\tilde\vartheta):=\log h_n(\cdot,\tilde\vartheta)$. Assume $\lbrace y\in\R^m; h_n(y,\tilde\vartheta)>0 \rbrace$ does not depend on $\tilde\vartheta\in\Theta$ and $\Theta$ is open and convex. Taylor expansion yields
\begin{equation*}
\begin{split}
&{\sqrt n}\sum_{i=1}^n    \alpha_{i,n}\frac {\partial }{\partial \tilde \vartheta_j}L_n( X_i^{(n)},\tilde\vartheta)_{|_{\tilde\vartheta= \vartheta_n}}+ \sum_{i=1}^n  \alpha_{i,n}\nabla'\frac {\partial}{\partial \tilde\vartheta_j}L_n(X_i^{(n)},\tilde \vartheta)_{|_{ \tilde\vartheta=\overline \vartheta_{j,n}}}\sqrt n(\hat \vartheta _n^{(n)} -\vartheta_n)\\
=&{\sqrt n} \sum_{i=1}^n \alpha_{i,n}\frac {\partial }{\partial \tilde\vartheta_j} L_n(X_i^{(n)}, \tilde\vartheta)_{|_{\tilde \vartheta=\hat \vartheta_n^{(n)}}},~j=1,\dots,d,
\end{split}
\end{equation*}
where $\overline \vartheta_{j,n}$ is on the line between $\hat \vartheta_n^{(n)}$ and $\vartheta_n$ ($\overline \vartheta_{j,n}$ can be chosen measurable, see Hilfssatz 6.7 of M\"uller-Funk and Witting \cite{wit}), $j=1,\dots,d$. This is equivalent to
\begin{equation*}
\begin{split}
&{\sqrt n}\sum_{i=1}^n   \alpha_{i,n}\nabla L_n( X_i^{(n)},\tilde\vartheta)_{|_{\tilde\vartheta= \vartheta_n}}
+
\begin{pmatrix}
 \sum_{i=1}^n  \alpha_{i,n}\nabla'\frac {\partial}{\partial \tilde\vartheta_1}L_n(X_i^{(n)}, \tilde\vartheta)_{|_{ \tilde \vartheta=\overline \vartheta_{1,n}}}\\
\vdots \\
 \sum_{i=1}^n  \alpha_{i,n}\nabla'\frac {\partial}{\partial \tilde\vartheta_d}L_n(X_i^{(n)}, \tilde\vartheta)_{|_{\tilde \vartheta=\overline \vartheta_{d,n}}}
 \end{pmatrix}
\sqrt n(\hat \vartheta _n^{(n)} -\vartheta_n)\\
=&{\sqrt n} \sum_{i=1}^n \alpha_{i,n} \nabla L_n(X_i^{(n)}, \tilde\vartheta)_{|_{\tilde \vartheta=\hat \vartheta_n^{(n)}}}.
\end{split}
\end{equation*}
Suppose  $\hat\vartheta^{(n)}_n$ has the property
\begin{equation*}
\lim_{n\rightarrow\infty}{\sqrt n} \sum_{i=1}^n  \alpha_{i,n}\nabla L_n(X_i^{(n)}, \tilde\vartheta)_{|_{ \tilde\vartheta=\hat \vartheta_n^{(n)}}}=0~\text{in probability}
\end{equation*}
and assume the existence of a matrix $i(\vartheta_n)\in\R^{d\times d}$ such that $i(\vartheta_n)$ is invertible and
\begin{equation*}
\lim_{n\rightarrow\infty}\left(
\begin{pmatrix}
 \sum_{i=1}^n  \alpha_{i,n}\nabla'\frac {\partial}{\partial \tilde\vartheta_1}L_n(X_i^{(n)}, \tilde\vartheta)_{|_{ \tilde \vartheta=\overline \vartheta_{1,n}}}\\
\vdots \\
\sum_{i=1}^n  \alpha_{i,n}\nabla'\frac {\partial}{\partial \tilde\vartheta_d}L_n(X_i^{(n)}, \tilde\vartheta)_{|_{\tilde \vartheta=\overline \vartheta_{d,n}}}
 \end{pmatrix}
+i(\vartheta_n)\right)=0~\text{in probability}.
\end{equation*}
Typically, $i$ is the Fisher information matrix with respect to the distribution function $G(\cdot,\tilde\vartheta), \tilde\vartheta\in\Theta$. Put $\ell_n(x,\tilde\vartheta):=i^{-1}(\tilde\vartheta)\nabla L_n(x,\tilde\vartheta)$, $(x,\tilde\vartheta)\in\overline \R ^m\times \Theta$, and note that
\begin{equation*}
\sqrt n(\hat \vartheta _n^{(n)} -\vartheta_n)= {\sqrt n}\sum_{i=1}^n  \alpha_{i,n} \ell_n( X_i^{(n)},\vartheta_n)+o_P(1)~\text{as}~n\rightarrow\infty.
\end{equation*}
Further, assume  for all $\tilde\vartheta\in\Theta$ $\int \ell_n(x,\tilde\vartheta)  G_n(\mathrm d x,\tilde \vartheta)$ and $\int \ell_n(x,\tilde\vartheta)\ell_n'(x,\tilde\vartheta)  G_n(\mathrm d x,\tilde \vartheta)$ exist and have finite components and
\begin{equation}\label{mleq}
\int \ell_n(x,\tilde\vartheta) G_n(\mathrm d x,\tilde \vartheta)=\int \ell_n(x,\tilde\vartheta)h_n(x,\tilde \vartheta)\mu(\mathrm d x)=0~\text{for all}~\tilde\vartheta\in\Theta.
\end{equation}
(\ref{mleq}) is a typical condition related to maximum-likelihood. Thus, the estimator $\hat \vartheta_n^{(n)}$ has the structure required in (C5). Now, assume $P(X_i\in\lbrace y\in\R^m; h_n(y,\tilde\vartheta)>0 \rbrace)=1$, $i=1,\dots,n$. Analogous,
\begin{equation*}
\begin{split}
&{\sqrt n}\sum_{i=1}^n  \alpha_{i,n}  \frac {\partial }{\partial \tilde \vartheta_j}L_n( X_i,\tilde\vartheta)_{|_{\tilde\vartheta= \vartheta}}+ \sum_{i=1}^n  \alpha_{i,n}\nabla'\frac {\partial}{\partial \tilde\vartheta_j}L_n(X_i,\tilde \vartheta)_{|_{ \tilde\vartheta=\overline {\overline\vartheta}_{j,n}}}\sqrt n(\hat \vartheta _n -\vartheta)\\
=&{\sqrt n} \sum_{i=1}^n \alpha_{i,n} \frac {\partial }{\partial \tilde\vartheta_j} L_n(X_i, \tilde\vartheta)_{|_{\tilde \vartheta=\hat \vartheta_n}},~j=1,\dots,d,
\end{split}
\end{equation*}
where $\overline {\overline\vartheta}_{j,n}$ is on the line between $\hat \vartheta_n$ and $\vartheta$, $j=1,\dots,d$, and
\begin{equation*}
\begin{split}
&{\sqrt n}\sum_{i=1}^n  \nabla L_n( X_i,\tilde\vartheta)_{|_{\tilde\vartheta= \vartheta}}
+
\begin{pmatrix}
 \sum_{i=1}^n  \alpha_{i,n}\nabla'\frac {\partial}{\partial \tilde\vartheta_1}L_n(X_i, \tilde\vartheta)_{|_{ \tilde \vartheta=\overline {\overline\vartheta}_{1,n}}}\\
\vdots \\
 \sum_{i=1}^n  \alpha_{i,n}\nabla'\frac {\partial}{\partial \tilde\vartheta_d}L_n(X_i, \tilde\vartheta)_{|_{\tilde \vartheta=\overline {\overline\vartheta}_{d,n}}}
 \end{pmatrix}
\sqrt n(\hat \vartheta _n -\vartheta)\\
=&{\sqrt n} \sum_{i=1}^n \alpha_{i,n} \nabla L_n(X_i, \tilde\vartheta)_{|_{\tilde \vartheta=\hat \vartheta_n}}.
\end{split}
\end{equation*}
Suppose
\begin{equation*}
\lim_{n\rightarrow\infty}{\sqrt n} \sum_{i=1}^n  \alpha_{i,n}\nabla L_n(X_i, \tilde\vartheta)_{|_{ \tilde\vartheta=\hat \vartheta_n}}=0~\text{in probability}
\end{equation*}
and
\begin{equation*}
\lim_{n\rightarrow\infty}
\begin{pmatrix}
 \sum_{i=1}^n  \alpha_{i,n}\nabla'\frac {\partial}{\partial \tilde\vartheta_1}L_n(X_i, \tilde\vartheta)_{|_{ \tilde \vartheta=\overline {\overline\vartheta}_{1,n}}}\\
\vdots \\
 \sum_{i=1}^n \alpha_{i,n} \nabla'\frac {\partial}{\partial \tilde\vartheta_d}L_n(X_i, \tilde\vartheta)_{|_{\tilde \vartheta=\overline {\overline\vartheta}_{d,n}}}
 \end{pmatrix}
=-i(\vartheta)~\text{in probability},
\end{equation*}
too. Regard that ${\mathbb F}_n= G_n(\cdot, \vartheta)$. Consequently,
\begin{equation*}
\sqrt n(\hat \vartheta _n -\vartheta)={\sqrt n}\sum_{i=1}^n \alpha_{i,n}  \ell_n( X_i,\vartheta)+o_P(1)~\text{as}~n\rightarrow\infty.
\end{equation*}
Assume  for all $i=1,\dots,n$, $\int \ell_n(x,\vartheta) F_i(\mathrm d x)$ and $\int \ell_n(x,\vartheta)\ell_n'(x,\vartheta) F_i(\mathrm d x)$ exist  and have finite components. In fact,
\begin{equation*}
\int \ell_n(x,\vartheta){\mathbb F}_n(\mathrm d x)=\int \ell_n(x,\vartheta){ G}_n(\mathrm d x,\vartheta)=0.
\end{equation*}
Therefore, the estimator $\hat \vartheta_n$ has the structure required in (C3), too.
\end{rmk}
\begin{lem} \label{lem5} Assume (C1), (C2), (C5) and (C6). Then, ${\sqrt n}\sum_{i=1}^n \alpha_{i,n} \ell_n(X_i^{(n)},\vartheta_n)$ converges in distribution to a centered $d$-dimensional normal distribution as $n\rightarrow\infty$ and
\begin{equation*}
\lim_{n\rightarrow\infty}\hat \vartheta_n^{(n)}=\vartheta~\text{in probability}.
\end{equation*}
\end{lem}
Let $\mathbb U^{()}(\vartheta)=(U^{()}(x,\vartheta);x\in\overline\R^m)$ be a Gaussian process with expectation function identically equal to zero, covariance function
\begin{equation}\label{cov4}
\begin{split}
c^{()}(x,y,\vartheta):=&\kappa\Big(  G\big(\min(x,y), \vartheta\big)-  G(x,\vartheta)  G(y,\vartheta)\\
&- g'(x, \vartheta) w^{()}(y, \vartheta)- g'(y, \vartheta) w^{()}(x, \vartheta)+ g'(x, \vartheta) v^{()}(\vartheta) g(y, \vartheta)\Big),~x,y\in\overline\R^m,
\end{split}
\end{equation}
and a.s. uniformly $\rho$-continuous sample paths. This process is a modification of the process $\mathbb U(\vartheta)$.
\begin{thm}\label{thm5} Assume (C1), (C2), (C5) and (C6). Then,
\begin{equation*}
\mathbb U_n^{(n)}(\hat \vartheta_n^{(n)})\overset{\mathrm d}{\longrightarrow} \mathbb U^{()}(\vartheta)~\text{as}~n\rightarrow\infty.
\end{equation*}
\end{thm}
\begin{cor}\label{cor6}Assume (C1), (C2), (C5) and (C6). Then,
\begin{equation*}
\KS_n^{(n)}\overset{\mathrm d}{\longrightarrow}\sup_{x\in\overline\R^m}|U^{()}(x,\vartheta)|~\text{and}~\CvM_n^{(n)}\overset{\mathrm d}{\longrightarrow}\int \big(U^{()}(x,\vartheta)\big)^2   G(\mathrm d x,\vartheta)~\text{as}~n\rightarrow\infty.
\end{equation*}
\end{cor}
The following result implies that the test (\ref{Test2}) is of asymptotically exact size $\alpha$.
\begin{cor}\label{cor7}
Suppose  the existence of a $\vartheta\in\Theta$ with ${\mathbb F}_n= G_n(\cdot,\vartheta)$  for $n$ sufficiently large and  (C1) - (C4) hold for this parameter. In addition, assume  (C5) and (C6) hold for arbitrary sequences $(\vartheta_n)_{n\in\N}$  with $\vartheta_n\in\Theta$ for all $n\in\N$ and $\lim_{n\rightarrow\infty}\vartheta_n=\vartheta$ and  $ v^{()}(\vartheta)= v(\vartheta)$ and $ w^{()}(x,\vartheta)= w(x,\vartheta)$ for all $x\in \overline\R^m$. Moreover, suppose  the restriction of the covariance function (\ref{cov3}) to the diagonal of  $\overline \R^m\times\overline \R^m$ does not vanish everywhere or $ G(\cdot,\vartheta)$-almost everywhere. Then,
\begin{equation*} 
\lim_{n\rightarrow\infty}P(\KS_n\ge c_{n;1-\alpha})=\alpha~\text{or}~\lim_{n\rightarrow\infty}P(\CvM_n\ge d_{n;1-\alpha})=\alpha.
\end{equation*}
\end{cor}

\subsection{Limit results unter alternatives}

In order to establish limit results under alternatives, the weak convergence in probability of the parameter estimator is still required. Consider the following condition.
\begin{itemize}
\item[(C7)] There exists a $\vartheta\in\Theta$ with $\lim_{n\rightarrow\infty}\hat \vartheta_n=\vartheta$ in probability.
\end{itemize}
\begin{thm} \label{thm6} Assume (C7) and  the conditions (C1) and (C2) hold for the parameter $\vartheta$ given by (C7). Then,
\begin{equation*}
\forall\varepsilon >0:P\bigg(\frac{1}{\sqrt n}\KS_n\ge\sup_{x\in\overline\R^m}|{ F}(x)- G(x,\vartheta)|-\varepsilon+o_p(1)\bigg)\longrightarrow 1~\text{as}~n\rightarrow\infty,
\end{equation*}
and
\begin{equation*}
\forall \varepsilon >0:P\bigg(\frac 1 n \CvM_n\ge \int\big({ F}(x)- G(x,\vartheta)\big)^2   G(\mathrm d x,\vartheta)-\varepsilon+o_p(1)\bigg)\longrightarrow 1~\text{as}~n\rightarrow\infty.
\end{equation*}
\end{thm}
The following corollary is useful for consistency considerations  with respect to the test (\ref{Test2}).
\begin{cor}\label{cor8}
Assume (C7) and suppose that the conditions (C1) and (C2) hold for the parameter $\vartheta$ given by (C7). Moreover, assume  for $\vartheta$ given by (C7), (C5) and (C6) hold for arbitrary sequences $(\vartheta_n)_{n\in\N}$ with $\vartheta_n\in\Theta$ for all $n\in\N$ and $\lim_{n\rightarrow\infty}\vartheta_n=\vartheta$. Suppose $ F-  G(\cdot,\vartheta)$ does not vanish everywhere or $ G(\cdot,\vartheta)$-almost everywhere and  the restriction of the covariance function (\ref{cov4}) to the diagonal of $\overline \R^m\times\overline \R^m$ does not vanish  everywhere or $ G(\cdot,\vartheta)$-almost everywhere. Then,
\begin{equation*} 
\lim_{n\rightarrow\infty}P(\KS_n\ge c_{n;1-\alpha})=1~\text{or}~\lim_{n\rightarrow\infty}P(\CvM_n\ge d_{n;1-\alpha})=1.
\end{equation*}
\end{cor}

\section{Testing hypotheses formulated by Hadamard differentiable functionals}\label{had}
Now, consider a general two sample situation. All results can be simply extended to more than two samples or modified to the one sample case. Regard the definitions at the beginning of Section \ref{simplegood1}. The second sample comes from a sequence of independent but not necessarily identically distributed $\R^s$-valued random vectors $V_1,V_2,\dots$, $s\in\N$. Assume that $X_1,X_2,\dots$ and $V_1,V_2,\dots$ are independent. For $i\in\N$, let $G_i$ be the distribution function of $V_i$ defined on $\overline \R^s$ and let $G$ be an uniformly continuous distribution function defined on $\overline \R^s$ with
\begin{equation*}
\lim_{i\rightarrow\infty}G_i=G~\text{uniformly on}~\overline\R^s.
\end{equation*}
In addition, let $r\in\N$ be a second sample size and assume the sample sizes increasing such that
\begin{equation}\label{samplesize}
\lim_{\min N\rightarrow\infty}\frac{n}{r}=\eta\in(0,\infty),
\end{equation}
where $N:=\lbrace n,r \rbrace$. Furthermore,  let $\beta_{1,r},\dots,\beta_{r,r}$ be some additional weights, in particular real numbers, with 
\begin{equation}\label{seq2}
\beta_{i,r}\ge 0,~i=1,\dots,r,~\sum_{i=1}^r\beta_{i,r}=1,~\lim_{r\rightarrow\infty}\sqrt r\max_{1\le i\le r}\beta_{i,r}=0,~\lim_{r\rightarrow\infty}r\sum_{i=1}^r\beta_{i,r}^2=\tau\in[1,\infty).
\end{equation}
Define the mixture distribution 
\begin{equation*}
\mathbb G_r(v):=\sum_{i=1}^r \beta_{i,r}G_i(v),~v\in\overline\R^s,
\end{equation*}
based on the distribution functions $G_1,\dots,G_r$ of $V_1,\dots,V_r$ and the weighted empirical distribution function
\begin{equation*}
\hat{\mathbb G}_r(v):=\sum_{i=1}^n\beta_{i,r}\I(V_i\le v),~v\in\overline\R^s,
\end{equation*}
based on $V_1,\dots,V_r$

\medskip
The aim is to  formulate a general hypothesis  with the help of two functionals. Therefor, denote by $\mathcal D(\overline \R^l)$ the set of all distribution functions defined on $\overline \R^{l}$, $l\in\N$. Consider two maps  $T:\mathcal D(\overline \R^{m})\times\mathcal D(\overline \R^{s})\rightarrow \mathcal D(\overline \R^u)$, $Q:\mathcal D(\overline \R^{m})\times\mathcal D(\overline \R^{s})\rightarrow \mathcal D(\overline \R^u)$, $u\in\N$, and assume this maps are continuous with respect to the supremum metrics on $ \mathcal D(\overline \R^u)$ and $\mathcal D(\overline \R^{m})\times\mathcal D(\overline \R^{s})$. Suppose  $F_1,\dots,F_n$ and  $G_1,\dots,G_r$  are unknown, the maps $T$ and $Q$ as well as $\alpha_{1,n},\dots,\alpha_{n,n}$ and $\beta_{1,r},\dots,\beta_{r,r}$ are known and that the user has to verify the hypothesis
\begin{equation}\label{testprallg}
\mathrm H_N:T(\mathbb F_n,\mathbb G_r)=Q(\mathbb F_n,\mathbb G_r)
\end{equation}
on the basis of the observations $X_1,\dots,X_n$ and $V_1,\dots,V_r$.

\medskip
Another functional is required for the construction of the testing procedure. Let $ j:\mathcal D(\overline \R^{m})\times\mathcal D(\overline \R^{s})\rightarrow \mathcal D(\overline \R^{m})\times\mathcal D(\overline \R^{s})$ be a continuous map with respect to the supremum metric on  $\mathcal D(\overline \R^{m})\times\mathcal D(\overline \R^{s})$ such that
\begin{equation*}
T\circ j=Q\circ j
\end{equation*}
and
\begin{equation*}
j( K)= K~\text{for all}~ K\in \mathcal D(\overline \R^{m})\times\mathcal D(\overline \R^{s})~\text{with}~T( K)=Q( K).
\end{equation*}
The map $j$ could be a projection on the subspace of $\mathcal D(\overline \R^{m})\times\mathcal D(\overline \R^{s})$ given by the hypothesis. Define $\hat {\mathbb J}_{N}:=j(\hat {\mathbb F}_n,\hat {\mathbb G}_r)$ and write $\hat {\mathbb J}_{N}=(\hat {\mathbb J}_{N,m},\hat {\mathbb J}_{N,s})$, where $\hat {\mathbb J}_{N,m}$ has realizations in $\mathcal D(\overline \R^m)$ and $\hat {\mathbb J}_{N,s}$ has realizations in $\mathcal D(\overline \R^s)$. Put $ J:=j( F,G)$ and write $ J=(J_m, J_s)$, where $ J_m\in \mathcal D(\overline \R^{m})$ and $ J_s\in \mathcal D(\overline \R^{s})$.

\medskip
Define a process $\mathbb U_N:=(U_N(z);z\in\overline\R^u)$ by
\begin{equation}\label{procallg}
U_N(z):=\sqrt[4]{nr} \big(T(\hat {\mathbb F}_n,\hat {\mathbb G}_r)(z)-Q(\hat {\mathbb F}_n,\hat {\mathbb G}_r)(z)\big),~z\in\overline\R^u,
\end{equation}
and consider the Kolmogorov-Smirnov type statistic
\begin{equation}\label{KSallg}
\KS_N:=\sup_{z\in\overline\R^u}|U_N(z)|=\sqrt[4]{nr}\sup_{z\in\overline\R^u}\big| T(\hat {\mathbb F}_n,\hat {\mathbb G}_r)(z)-Q(\hat {\mathbb F}_n,\hat {\mathbb G}_r)(z)\big|
\end{equation}
as well as the Cram\'er-von-Mises type statistic 
\begin{equation}\label{CvMallg}
\begin{split}
\CvM_N:&=\int U^2_N(z) T( \hat {\mathbb J}_N)(\mathrm d z)=\int U^2_N(z) Q(\hat {\mathbb J}_N)(\mathrm d z)\\
&=\sqrt{nr}\int\big( T(\hat {\mathbb F}_n,\hat {\mathbb G}_r)(z)-Q(\hat {\mathbb F}_n,\hat {\mathbb G}_r)(z)\big)^2 T(\hat {\mathbb J}_N)(\mathrm d z).
\end{split}
\end{equation}
For simplicity, assume  $T(\hat {\mathbb F}_n,\hat {\mathbb G}_r)$, $Q(\hat {\mathbb F}_n,\hat {\mathbb G}_r)$ and $T(  \hat {\mathbb J}_N)$ are step functions with countable vertices. This guarantees measurability of suprema and integrals in this section. 

\medskip
In general, the statistics  (\ref{KSallg}) and (\ref{CvMallg}) are not distribution free, neither in the case of  $T(\mathbb F_n,\mathbb G_r)=Q(\mathbb F_n,\mathbb G_r)$.  In order to approximate the distribution of the statistics $\KS_N$ and $\CvM_N$  if  $T(\mathbb F_n,\mathbb G_r)=Q(\mathbb F_n,\mathbb G_r)$, a Monte-Carlo procedure is suggested. Therefor, simulate independently observations with joint distribution function
\begin{equation*}
\begin{split}
&\big((x_1,\dots,x_n),(v_1,\dots,v_r)\big)\longmapsto\prod_{i=1}^{n} \hat {\mathbb J}_{N,m}(x_i)\prod_{i=1}^{r} \hat {\mathbb J}_{N,s}(v_i),\\
&\big((x_1,\dots,x_n),(v_1,\dots,v_r)\big)\in \times_{i=1}^{n}\overline\R^{m}\times \times_{i=1}^{r}\overline\R^{s}.
\end{split}
\end{equation*}
Determine a significance level  $\alpha\in (0,1)$. Denote by $c_{N;1-\alpha}$ the $(1-\alpha)$-quantile of the distribution of  $\KS_N$ and by $d_{N;1-\alpha}$ the $(1-\alpha)$-quantile of the distribution of  $\CvM_N$ if $(F_1,\dots,F_{n})=( \hat {\mathbb J}_{N,m},\dots,\hat {\mathbb J}_{N,m})$ and $(G_1,\dots,G_{r}) =(\hat {\mathbb J}_{N,s},\dots,\hat {\mathbb J}_{N,s})$. The calculation procedure for practice is described above. Then, testing procedure
\begin{equation}\label{Testallg}
\text{``Reject}~\mathrm H_N,~\text{iff}~\KS_N\ge c_{N;1-\alpha}\text{''}~\text{or testing procedure}~\text{``Reject}~\mathrm H_N,~\text{iff}~\CvM_N\ge d_{N;1-\alpha}\text{''} 
\end{equation}
is suggested. In the case of  $\alpha_{i,n}=\frac 1 {n}$, $i=1,\dots,n$, and $\beta_{i,r}=\frac 1 {r}$, $i=1,\dots,r$, testing procedure (\ref{Testallg}),  in particular the statistic (\ref{KSallg}) or (\ref{CvMallg}), is invariant under transformation of the data from the set $\mathcal T_{n,m}\times \mathcal T_{r,s}$, where $\mathcal T_{n,m}$ and $\mathcal T_{r,s}$ are defined analogous to (\ref{traf}).

\subsection{Limit results under the hypothesis}

\medskip
The asymptotic results in this section based on a smoothness condition on the functionals $T$ and $Q$. Call a map $T:\mathcal D(\overline \R^{m})\times\mathcal D(\overline \R^{s})\rightarrow \mathcal D(\overline \R^u)$ uniformly Hadamard differentiable in $  K\in \mathcal D(\overline \R^{m})\times\mathcal D(\overline \R^{s})$, iff there exists a  linear map $\mathrm d T(  K): \ell^\infty(\overline \R^{m})\times \ell^\infty(\overline \R^{s})\rightarrow  \ell^\infty(\overline \R^u)$ such that $\mathrm d T(  K)$ is continuous with respect to the supremum metrics on $ \ell^\infty(\overline \R^u)$ and $ \ell^\infty(\overline \R^{m})\times \ell^\infty(\overline \R^{s})$ and 
\begin{equation*}
\lim_{n\rightarrow\infty}\frac{1}{t_n}\big(T(K_n+t_nL_n)-T(K_n)\big)=\mathrm d T(  K  )( L)~\text{uniformly on}~\overline \R^u
\end{equation*}
 for all sequences $(K_n)_{n\in\N}$  with $K_n\in\mathcal D(\overline \R^{m})\times\mathcal D(\overline \R^{s})$  for all $n\in\N$ and $\lim_{n\rightarrow\infty}K_n=  K$ uniformly on $\overline \R^{m}\times\overline\R^s$, for all sequences  $(L_n)_{n\in\N}$ with $L_n\in \ell^\infty(\overline \R^{m})\times \ell^\infty(\overline \R^{s})$ for all $n\in\N$ and the existence of a $  L\in\ell^\infty(\overline \R^{m})\times \ell^\infty(\overline \R^{s})$ with $\lim_{n\rightarrow\infty}L_n=  L$  uniformly on $\overline \R^{m}\times \overline \R^{s}$ and all sequences $(t_n)_{n\in\N}$  with $t_n\in(0,\infty)$ for all $n\in\N$ and $\lim_{n\rightarrow\infty}t_n=0$, where $(K_n+t_nL_n)\in\mathcal D(\overline \R^{m})\times\mathcal D(\overline \R^{s})$  for all $n\in\N$. $\mathrm d T(  K)$ is called Hadamard derivative in $  K$. Concepts of  differentiability of functionals  and applications in statistics are discussed by v. Mises \cite{Mis}, Shapiro \cite{sha}, van der Vaart \cite {van2} and Ren and Sen \cite{ren3}, \cite{ren}, \cite{ren2}.

\medskip
In order to determine the limit distribution of the process  $\mathbb U_N$ under the hypothesis, let $\mathbb M_m:=(M_m(x);x\in\overline\R^{m})$ be a Gaussian process with expectation function identically equal to zero, covariance function
\begin{equation*}
c_m(x,y):=\kappa\eta\Big( J_m\big(\min(x,y)\big)-  J_m(x)  J_m(y)\Big),~x,y\in\overline\R^{m},
\end{equation*}
and a.s. uniformly continuous sample paths with respect to the metric $\rho$ on $\overline\R^{m}$  defined in (\ref{metr}). In addition, let $\mathbb M_s:=(M_s(v);v\in\overline\R^{s})$ be a Gaussian process with expectation function identically equal to zero, covariance function
\begin{equation*}
c_s(v,w):=\tau\eta\Big( J_s\big(\min(v,w)\big)-  J_s(v)  J_s(w)\Big),~v,w\in\overline\R^{s},
\end{equation*}
and a.s. uniformly continuous sample paths with respect to a metric on $\overline\R^{s}$  defined analogous to the metric $\rho$ in (\ref{metr}).
Let $\mathbb M_m$ and $\mathbb M_s$ be independent and put $\mathbb M:=(\mathbb M_m,\mathbb M_s)$ and $M=(M_m,M_s)$. Finally, let $\mathbb U:=(U(z);z\in\overline \R^u)$ be a process defined by
\begin{equation*}
U(z):=\mathrm d T( J)( M)(z)-\mathrm d Q( J)( M)(z),~z\in\overline \R^u,
\end{equation*} 
if $T$ and $Q$ are uniformly Hadamard differentiable in $ J$ with Hadamard derivative $\mathrm d T( J)$ and $\mathrm d Q( J)$, respectively, 
\begin{thm}\label{thmallg1} Assume $T(\mathbb F_n,\mathbb G_r)=Q(\mathbb F_n,\mathbb G_r)$ for $\min N$ sufficiently large. Suppose  $T$ and $Q$ are uniformly Hadamard differentiable in $(F,G)$ with Hadamard derivative $\mathrm d T( F,G)$ and $\mathrm d Q( F,G)$, respectively. Then, $ J=(F,G)$ and
\begin{equation*}
\mathbb U_N\overset{\mathrm d}{\longrightarrow} \mathbb U~\text{as}~\min N\rightarrow\infty.
\end{equation*}
\end{thm}
\begin{cor}\label{corallg1}
Assume $T(\mathbb F_n,\mathbb G_r)=Q(\mathbb F_n,\mathbb G_r)$ for $\min N$ sufficiently large. Suppose  $T$ and $Q$ are uniformly Hadamard differentiable in $(F,G)$ with Hadamard derivative $\mathrm d T( F,G)$ and $\mathrm d Q( F,G)$, respectively. Then,
\begin{equation*}
\KS_N\overset{\mathrm d}{\longrightarrow}\sup_{z\in\overline\R^u}|U(z)|~\text{and}~\CvM_N\overset{\mathrm d}{\longrightarrow}\int U^2(z) T(F,G)(\mathrm d z)~\text{as}~\min N\rightarrow\infty.
\end{equation*}
\end{cor}

\medskip
In order to show that testing procedure (\ref{Testallg}) works asymptotically, it is necessary to study the asymptotic distribution of the statistics  (\ref{KSallg}) and (\ref{CvMallg}) if $(F_1,\dots,F_{n})=( \hat {\mathbb J}_{N,m},\dots,\hat {\mathbb J}_{N,m})$ and $(G_1,\dots,G_{r}) =(\hat {\mathbb J}_{N,s},\dots,\hat {\mathbb J}_{N,s})$, too.

\medskip
Well, for an arbitrary $J_N\in  \mathcal D(\overline \R^{m})\times\mathcal D(\overline \R^{s})$   with $T(J_N)=Q(J_N)$ and $\lim_{n\rightarrow\infty}J_N= J$ uniformly on $ \mathcal D(\overline \R^{m})\times\mathcal D(\overline \R^{s})$, $J_N=(J_{N,m},J_{N,s})$, $J_{N,m}\in \mathcal D(\overline \R^{m})$,  $J_{N,s}\in \mathcal D(\overline \R^{s})$, let $X_1^{(N)},X_2^{(N)},\dots$ be a sequence of independent and identically distributed random vectors with underlying distribution function $J_{N,m}$ and let $V_1^{(N)},V_2^{(N)},\dots$ be a sequence of independent and identically distributed random vectors with underlying distribution function $J_{N,s}$.  Assume that $X_1^{(N)},X_2^{(N)},\dots$ and $V_1^{(N)},V_2^{(N)},\dots$  are independent. In addition, let $\mathbb U_N^{(N)}=(U_N^{(N)}(z);z\in\overline\R^u)$ be the process (\ref{procallg}) based on the random vectors $X_1^{(N)},\dots,X_{n}^{(N)}$ and $V_1^{(N)},\dots,V_{r}^{(N)}$ and denote by $\KS_N^{(N)}$ and $\CvM_N^{(N)}$ the statistics (\ref{KSallg}) and (\ref{CvMallg}) based on this random vectors.

\begin{thm}\label{thmallg2} Suppose  $T$ and $Q$ are uniformly Hadamard differentiable in $ J$ with Hada\-mard derivative $\mathrm d T( J)$ and $\mathrm d Q( J)$, respectively. Assume  $ J$ is uniformly continuous. Then,
\begin{equation*}
\mathbb U_N^{(N)}\overset{\mathrm d}{\longrightarrow} \mathbb U~\text{as}~\min N\rightarrow\infty.
\end{equation*}
\end{thm}

\begin{cor}\label{corallg2}
Suppose  $T$ and $Q$ are uniformly Hadamard differentiable in $ J$ with Hadamard derivative $\mathrm d T( J)$ and $\mathrm d Q( J)$, respectively. Assume  $ J$ is uniformly continuous. Then,
\begin{equation*}
\KS_N^{(N)}\overset{\mathrm d}{\longrightarrow}\sup_{z\in\overline\R^u}|U(z)|~\text{and}~\CvM_N^{(N)}\overset{\mathrm d}{\longrightarrow}\int U^2(z) T(  J)(\mathrm d z)~\text{as}~\min N\rightarrow\infty.
\end{equation*}
\end{cor}
The following result implies that the test (\ref{Testallg}) is a test of asymptotically exact size $\alpha$.
\begin{cor}\label{corallg3}
Assume $T(\mathbb F_n,\mathbb G_r)=Q(\mathbb F_n,\mathbb G_r)$ for $\min N$ sufficiently large. Suppose  $T$ and $Q$ are uniformly Hadamard differentiable in $(F,G)$ with Hadamard derivative $\mathrm d T( F,G)$ and $\mathrm d Q( F,G)$, respectively, and the distribution function of  $\sup_{z\in\overline\R^u}|U(z)|$ or $\int U^2(z) T( F,G)(\mathrm d z)$ is strictly increasing on the non-negative half-line. Then,
\begin{equation*}
\lim_{\min N\rightarrow\infty}P(\KS_N\ge c_{N;1-\alpha})=\alpha~\text{or}~\lim_{\min N\rightarrow\infty}P(\CvM_N\ge d_{N;1-\alpha})=\alpha.
\end{equation*}
\end{cor}

\subsection{Limit results under alternatives}
The limit behavior  in probability of the test statistics   (\ref{KSallg}) and (\ref{CvMallg}) is given in the following result. It holds under the hypotheses as well as under alternatives.
\begin{thm} \label{thmallg3}It is
\begin{equation*}
\forall\varepsilon >0:P\bigg(\frac{1}{\sqrt[4]{nr}}\KS_N\ge\sup_{z\in\overline\R^u}|T( F,G)(z)-Q( F,G)(z)|-\varepsilon+o_p(1)\bigg)\longrightarrow 1~\text{as}~\min N\rightarrow\infty,
\end{equation*}
and
\begin{equation*}
\begin{split}
\forall\varepsilon >0:&P\bigg(\frac {1}{\sqrt{nr}} \CvM_N\ge \int\big(T(F,G)(z)-Q(F,G)(z)\big)^2 T( J)(\mathrm d z)-\varepsilon+o_p(1) \bigg)\longrightarrow 1\\
&\text{as}~\min N\rightarrow\infty.
\end{split}
\end{equation*}
\end{thm}
The following corollary is important for consistency considerations with respect to the test (\ref{Testallg}).
\begin{cor}\label{corallg4}
Suppose  $T$ and $Q$ are uniformly Hadamard differentiable in $ J$ with Hadamard derivative $\mathrm d T( J)$ and $\mathrm d Q( J)$, respectively. Assume $ J$ is uniformly continuous and  the distribution function of  $\sup_{z\in\overline\R^u}|U(z)|$ or $\int U^2(z) T(  J)(\mathrm d z)$ is strictly increasing on the non-negative half-line. Moreover, assume  $T( F,G)-Q( F,G)$ does not vanish everywhere or $T( J)$-almost everywhere. Then,
\begin{equation*} 
\lim_{\min N\rightarrow\infty}P(\KS_N\ge c_{N;1-\alpha})=1~\text{or}~\lim_{\min N\rightarrow\infty}P(\CvM_N\ge d_{N;1-\alpha})=1.
\end{equation*}
\end{cor}

\subsection{Testing homogeneity}\label{hom}

Consider the two sample case with equal dimensions $m=s=u$ and equal sample sizes $n=r$. Suppose $F_1,\dots,F_n$ and $G_1,\dots,G_n$ are unknown, $\alpha_{1,n},\dots,\alpha_{n,n}$ and $\beta_{1,n},\dots,\beta_{n,n}$ are known and that the user has to treat the hypothesis of homogeneity
\begin{equation}\label{testpr4}
\mathrm H_n:\mathbb F_n=\mathbb G_n
\end{equation}
on the basis of the observations $X_1,\dots,X_n$ and $V_1,\dots,V_n$.

\medskip
In this situation, the maps  $T:\mathcal D(\overline \R^{m})\times \mathcal D(\overline \R^{m})\rightarrow \mathcal D(\overline \R^m)$ and $Q:\mathcal D(\overline \R^{m})\times \mathcal D(\overline \R^{m})\rightarrow \mathcal D(\overline \R^m)$ are given by the continuous maps
\begin{equation*}
T( K_1, K_2)= K_1,~Q( K_1, K_2)= K_2,~K=( K_1,K_2)\in \mathcal D(\overline \R^{m})\times \mathcal D(\overline \R^{m}).
\end{equation*}
Let  $K=( K_1, K_2)$ and $L=( L_1, L_2)$ be arbitrary objects from the definition of the uniformly Hadamard differentiability given above. Now, the maps $T$ and $Q$ are uniformly Hadamard differentiable in $K$ with Hadamard derivative 
\begin{equation*}
\mathrm d T(K)(L)=L_1~\text{and}~\mathrm d Q(K)(L)= L_2,
\end{equation*}
respectively. Let $ j:\mathcal D(\overline \R^{m})\times \mathcal D(\overline \R^{m})\rightarrow\mathcal D(\overline \R^{m})\times \mathcal D(\overline \R^{m})$ be the continuous map 
\begin{equation*}
j(K)=\Big(\frac 1 2( K_1+ K_2),\frac 1 2( K_1+ K_2)\Big),~K=( K_1, K_2)\in \mathcal D(\overline \R^{m})\times \mathcal D(\overline \R^{m}).
\end{equation*}
Obviously, it is $T\circ j=Q\circ j$ and $j(K)=(K)$ for all $K\in \mathcal D(\overline \R^{m})\times \mathcal D(\overline \R^{m})$ with $T(K)=Q(K)$. Thus,
\begin{equation*}
\hat{\mathbb J} _{N,m}=\hat{\mathbb J} _{N,s}=\frac 1 2\big(\hat{\mathbb F}_n +\hat{\mathbb G}_n\big),
\end{equation*}
the test statistics have the expressions
\begin{equation}\label{KS4}
\KS_n=\sqrt n\sup_{x\in\overline\R^m}|\hat{\mathbb F}_n (x) -\hat{\mathbb G}_n(x)|
\end{equation}
as well as
\begin{equation}\label{CvM4}
\CvM_n=n\int\big( \hat{\mathbb F}_n (x) -\hat{\mathbb G}_n(x)\big)^2\frac 1 2\big(\hat{\mathbb F}_n +\hat{\mathbb G}_n\big)(\mathrm d x)
\end{equation}
and testing procedure
\begin{equation}\label{Test4}
\text{``Reject}~\mathrm H_n,~\text{iff}~\KS_n\ge c_{n;1-\alpha}\text{''}~\text{or testing procedure}~\text{``Reject}~\mathrm H_n,~\text{iff}~\CvM_n\ge d_{n;1-\alpha}\text{''} 
\end{equation}
is suggested, where $c_{n;1-\alpha}$ and $d_{n;1-\alpha}$ are calculated under $(F_1,\dots,F_n)=(G_1,\dots,G_n)=(\hat{\mathbb J} _{N,m},\dots,\hat{\mathbb J} _{N,m})$. If $\alpha_{i,n}=\beta_{i,n}$, $i=1,\dots,n$, testing procedure (\ref{Test4}), in particular the statistic (\ref{KS4}) or (\ref{CvM4}), is invariant under the transformation of the data
\begin{equation*}
\begin{split}
&\big((x_1,\dots,x_n),(v_1,\dots,v_n)\big)\longmapsto\big((v_1,\dots,v_n),(x_1,\dots,x_n)\big),\\
&\big((x_1,\dots,x_n),(v_1,\dots,v_n)\big)\in\times_{i=1}^n\R^m\times\times_{i=1}^n\R^m.
\end{split}
\end{equation*}
Moreover, $\mathbb U$ is a Gaussian process with expectation function identically equal to zero. Moreover,
\begin{equation*}
\Var\big(U(x)\big)=\Var\big(M_m(x)\big)+\Var\big(M_s(x)\big)=(\kappa+\tau)\big(J_m(x)-J_m^2(x)\big),~x\in\overline \R^m.
\end{equation*}
Because $F$ and $G$ are uniformly continuous,  $J_m$ is uniformly continuous, too, and the map $x \mapsto \Var(U(x))$, $x\in\overline \R^m$, does not vanish everywhere or $T(  J)$-almost everywhere. For that reason, the distribution function of  $\sup_{x\in\overline\R^m}|U(x)|$ or $\int U^2(x) T(  J)(\mathrm d x)$ is   strictly increasing on the non-negative half-line.
\begin{rmk}
\begin{itemize}
\item[a)] Modify the model introduced in Section \ref{einfuhr} in the following way. Add independent and identically distributed real valued random variables $U_1,U_2,\dots$  with unknown underlying distribution and  independent but not necessarily identically distributed real valued random variables $W_1,W_2,\dots$ with distribution $\mathcal L(W_i)=\mathcal L(Z_i)$, $i\in\N$. Now, the distribution of the error variable $\mathcal L(Z_i)$ and the error function $e_i$ are possibly unknown  for all $i\in\N$.  Assume the random variables  $Y_1,Y_2,\dots$,  $Z_1,Z_2,\dots$,  $U_1,U_2,\dots$ and  $W_1,W_2,\dots$ are independent. Suppose the user has to treat the testing problem of homogeneity
\begin{equation*}
\mathrm H:\mathcal L(Y_1)=\mathcal L(U_1),~\mathrm K:\mathcal L(Y_1)\neq \mathcal L(U_1)
\end{equation*}
on the basis of the observations $e_1(Y_1,Z_1),\dots,e_n(Y_n,Z_n)$ and $ e_1(U_1, W_1),\dots, \linebreak e_n(U_n, W_n)$. Putting $\alpha_{i,n}=\beta_{i,n}=\frac 1 n$, $X_i=e_i(Y_i,Z_i)$ and $V_i= e_i(U_i,W_i)$, $i=1,\dots,n$, the test (\ref{Test4}) is applicable to this testing problem. If the distribution function of $e(U_1,Z)$ is continuous, too, it follows from the results in this section that the test is of asymptotically exact size $\alpha$ and consistent.
\item[b)] The test (\ref{Test4}) is applicable to the hypothesis 
\begin{equation*}
\forall i\in\lbrace 1,\dots,n\rbrace :F_i=G_i.
\end{equation*}
 Putting $\alpha_{i,n}=\beta_{i,n}=\frac 1 n$, $i=1,\dots,n$, the test (\ref{Test4}) is also applicable to the hypothesis 
\begin{equation*}
\exists \pi_n:\lbrace 1,\dots,n\rbrace\rightarrow\lbrace 1,\dots,n\rbrace~\text{bijective}~\forall i\in\lbrace 1,\dots,n\rbrace :F_i=G_{\pi_n(i)}.
\end{equation*}
It follows from the results in this section that the test is of asymptotically exact size $\alpha$ and  consistent  with respect to suitable alternatives to this hypotheses.
\item[c)] Assume $F_i=F$ and $G_i=G$, $i\in\N$. Putting $\alpha_{i,n}=\beta_{i,n}=\frac 1 n$, $i=1,\dots,n$, the hypothesis (\ref{testpr4}) is a common hypothesis of homogeneity
\begin{equation*}
\mathrm H:F=G,
\end{equation*}
the test statistic (\ref{KS4}) or (\ref{CvM4}) is a common Kolmogorov-Smirnov  or  Cram\'er-von-Mises statistic for testing the hypothesis of homogeneity and the test (\ref{Test4}) is a common  Kolmogorov-Smirnov  or  Cram\'er-von-Mises test for testing the hypothesis of homogeneity in a multivariate setting. In this sense, a generalization of the identically distributed case  is considered.
\item[d)]
If $\alpha_{i,n}=\beta_{i,n}=\frac 1 n$, $i=1,\dots,n$, the test statistic (\ref{KS4}) or (\ref{CvM4}) has the form of a common Kolmogorov-Smirnov  or  Cram\'er-von-Mises type statistic for testing the hypothesis of homogeneity and the test (\ref{Test4}) hasthe form of a common  Kolmogorov-Smirnov  or  Cram\'er-von-Mises type test for testing the hypothesis of homogeneity in the identically distributed case. I.e., the results in this section imply that the common  Kolmogorov-Smirnov  or  Cram\'er-von-Mises type testing procedure  for testing the hypothesis of homogeneity  is applicable if the data come from different distributions and the mentioned conditions are fulfilled.
\end{itemize}
\end{rmk}

\subsection{Testing central symmetry}\label{sym}

\medskip
In order to apply the general results above to the hypothesis of central symmetry, another sample is needed. Let $U_1,U_2,\dots$ be a sequence of independent but not necessarily identically distributed random vectors  with values in  $\R^u$. For $i\in\N$, the random vector $U_i$ has the distribution function $C_i$ defined on $\overline \R^u$. Assume the existence of an uniformly continuous distribution function $C$ defined on $\overline \R^u$ with
\begin{equation*}
\lim_{i\rightarrow\infty}C_i=C~\text{uniformly on}~\overline\R^u.
\end{equation*}
For a given distribution function $ K$ defined on $\overline\R^u$ with related  probability measure $\mu$ on $\overline\BB^u$, i.e.,  $ K(z)=\mu([-\infty_u,z])$,  $u\in\overline \R^u$, write $ K^+:= K$ and $ K^-(z):=\mu([-z,\infty_u])$, $z\in\overline \R^u$. $\overline \BB^u$ denotes the Borel $\sigma$-field on $\overline\R^u$ and $[z,w]:=\times_{i=1}^u[z_i,w_i]$, $z,w\in\overline\R^u$, $z=(z_1,\dots,z_u)'$,  $w=(w_1,\dots,w_u)'$, if $z_i\le w _i$ for $i=1,\dots,u$, and  $[z,w]:=\emptyset$ else. Moreover, $\infty_{l}:=(\infty,\dots,\infty)'\in\overline\R^{l}$, ${l}\in\N$. Clearly, $ K^-$ is a distribution function. A distribution function $ K^+$  or a probability measure $\mu$  is called centrally symmetric, iff $ K^+= K^-$ or $\mu([-\infty_u,z])=\mu([-z,\infty_u])$ for all $z\in\overline \R^u$. A random vector $\xi$ with values in $ \R^u$ has a centrally symmetric distribution function iff $\xi$  and $-\xi$ have the same distribution. Concepts of symmetry in multivariate settings are presented by Serfling \cite{ser}.

\medskip
For $i\in\N$, denote by $C_i^+=C_i$ and $C_i^-$ the distribution functions of $U_i$ and $-U_i$, respectively.  Suppose  $C_1,\dots,C_n$ are unknown,  $\alpha_{1,n},\dots,\alpha_{n,n}$ are known and that the user has to treat the hypothesis
\begin{equation}\label{testpr3}
\mathrm H_n: \sum_{i=1}^n\alpha_{i,n} C_i^+= \sum_{i=1}^n\alpha_{i,n} C_i^-
\end{equation}
on the basis of the observations $U_1,\dots,U_n$. Regarding
\begin{equation*}
\begin{split}
&P(- U_i\le z)=1-\sum_{k=1}^u(-1)^{k+1}\sum_{1\le j_1<\dots<j_k\le u}P(U_{i,j_1}<-z_{j_1},\dots,U_{i,j_k}< -z_{j_k}),\\
&z\in\overline\R^{u},~i\in\N, 
\end{split}
\end{equation*}
it is obvious that the hypothesis (\ref{testpr3}) is equivalent to the hypothesis of central symmetry about the mixture distribution $\mathbb C_n(z):=\sum_{i=1}^n \alpha_{i,n}C_i(z)$, $z\in\overline\Rû$,
\begin{equation*}
\mathrm H_n: \mathbb C_n^+=\mathbb C_n^-.
\end{equation*}
In order to apply the general results above, consider the random vectors
\begin{equation}\label{vec}
\begin{pmatrix}U_1\\-U_1\end{pmatrix},\begin{pmatrix}U_2\\-U_2\end{pmatrix},\dots
\end{equation}
with values in $\R^{2u}$. Because
\begin{equation}\label{symVF}
\begin{split}
&P(U_i\le z,-U_i\le w)\\
=&P(U_i\le z)-\sum_{k=1}^u(-1)^{k+1}\sum_{1\le j_1<\dots<j_k\le u}P(U_{i,j_1}<-w_{j_1},\dots,U_{i,j_k}< -w_{j_k},U_i\le z),\\
&z,w\in\overline\R^u,~i\in\N,
\end{split}
\end{equation}
the uniformly convergence of the sequence of distribution functions $(C_i)_{i\in\N}$ to the uniformly continuous distribution function $C$ implies that the sequence of distribution functions of the random vectors (\ref{vec}) has those limit properties, too. For that reason, the sequence of distribution functions of the random vectors (\ref{vec}) fulfills the model assumptions in the general model described above.

\medskip
Well, consider the one sample case with dimension $m=2u$ and the sequence of random vectors $X_1=(\begin{smallmatrix}U_1\\-U_1\end{smallmatrix}),X_2=(\begin{smallmatrix}U_2\\-U_2\end{smallmatrix}),\dots$ with values in $\R^{m}$. Then, the hypothesis $(\ref{testpr3})$ is equivalent to the hypothesis of homogeneity of the first $u$-dimensional marginal distribution and the last $u$-dimensional marginal distribution about the mixture distribution $\mathbb F_n$.

\medskip
In this situation, the maps $T:\mathcal D(\overline \R^{m})\rightarrow\mathcal D(\overline \R^u)$ and $Q:\mathcal D(\overline \R^{m})\rightarrow \mathcal D(\overline \R^u)$ are given by the continuous maps
\begin{equation*}
T( K)(z)= K(\begin{smallmatrix}z\\\infty_u\end{smallmatrix}),~Q( K)(z)= K(\begin{smallmatrix}\infty_u\\z\end{smallmatrix}),~z\in\overline\R^u,~ K\in \mathcal D(\overline \R^{m}).
\end{equation*}
Let  $ K$ and $ L$ be arbitrary objects from the definition of the uniformly Hadamard differentiability given above. Now, the maps $T$ and $Q$ are uniformly Hadamard differentiable in $ K$ with Hadamard derivative 
\begin{equation*}
\mathrm d T( K)( L)(z)= L(\begin{smallmatrix}z\\\infty_u\end{smallmatrix})~\text{and}~\mathrm d Q( K)( L)(z)= L(\begin{smallmatrix}\infty_u\\z\end{smallmatrix}),~z\in\overline\R^u,
\end{equation*}
respectively. Let $ j:\mathcal D(\overline \R^{m})\rightarrow\mathcal D(\overline \R^{m})$ be the continuous map
\begin{equation*}
j( K)(\begin{smallmatrix}z\\w\end{smallmatrix})=\frac 1 2\big( K(\begin{smallmatrix}z\\w\end{smallmatrix})+ K(\begin{smallmatrix}w\\z\end{smallmatrix})\big),~z,w\in\overline\R^u,~ K\in \mathcal D(\overline \R^{m}).
\end{equation*}
It is $T\circ j=Q\circ j$ and $j( K)= K$ for all $ K\in \mathcal D(\overline \R^{m})$ with $T( K)=Q( K)$. Thus,
\begin{equation*}
\hat{\mathbb J} _{N,m}(\begin{smallmatrix}z\\w\end{smallmatrix})=\frac 1 2 \big(\hat{\mathbb F}_n(\begin{smallmatrix}z\\w\end{smallmatrix}) +\hat{\mathbb F}_n(\begin{smallmatrix} w\\z\end{smallmatrix})\big),~z,w\in\overline\R^u,
\end{equation*}
and
\begin{equation*}
T\big(\hat{\mathbb J} _{N,m}\big)(z)=\frac 1 2\big(\hat{\mathbb C}_n^+(z)+\hat{\mathbb C}_n^-(z)\big),~z\in\overline \R^{u},
\end{equation*}
where 
\begin{equation*}
\hat{\mathbb C}_n^+(z):= \sum_{i=1}^n\alpha_{i,n}\I(U_i\le z),~\hat{\mathbb C}_n^-(z):= \sum_{i=1}^n\alpha_{i,n}\I(-U_i\le z),~z\in\overline\R^u.
\end{equation*}
The test statistics have the expressions
\begin{equation}\label{KS3}
\KS_n=\sqrt n\sup_{z\in\overline\R^u}| \mathbb C_n^+(z)-\mathbb C_n^-(z)|
\end{equation}
and
\begin{equation}\label{CvM3}
\CvM_n=n\int\big( \mathbb C_n^+(z)-\mathbb C_n^-(z)\big)^2\frac 1 2\big(\hat{\mathbb C}_n^++\hat{\mathbb C}_n^-\big)(\mathrm d z).
\end{equation}
In the special case of $m=1$ and  $\alpha_{i,n}=\frac 1 n$, $i=1,\dots,n$, the statistic (\ref{KS3}) is considered in \cite{cha}. Now, testing procedure
\begin{equation}\label{Test3}
\text{``Reject}~\mathrm H_n,~\text{iff}~\KS_n\ge c_{n;1-\alpha}\text{''}~\text{or testing procedure}~\text{``Reject}~\mathrm H_n,~\text{iff}~\CvM_n\ge d_{n;1-\alpha}\text{''} 
\end{equation}
is suggested, where $c_{n;1-\alpha}$ and $d_{n;1-\alpha}$ are calculated under $(F_1,\dots,F_n)=(\hat{ \mathbb J} _{N,m},\dots,\hat{ \mathbb J} _{N,m})$.  Testing procedure (\ref{Test3}),  in particular the statistic (\ref{KS3}) or (\ref{CvM3}), is invariant under the transformation of the data
\begin{equation*}
(z_1,\dots,z_n)\longmapsto(-z_1,\dots,-z_n),~(z_1,\dots,z_n)\in\times_{i=1}^n\R^u.
\end{equation*}
Moreover, $\mathbb U$ is a Gaussian process with expectation function identically equal to zero and
\begin{equation*}
\begin{split}
&\Var\big(U(z)\big)=\Var\big(M_m(\begin{smallmatrix}z\\\infty_u\end{smallmatrix})\big)+\Var\big (M_m(\begin{smallmatrix}\infty_u\\z\end{smallmatrix})\big)-2\Cov\big(M_m(\begin{smallmatrix}z\\\infty_u\end{smallmatrix}),M_m(\begin{smallmatrix}\infty_u\\z\end{smallmatrix})\big)
=2\kappa J(\begin{smallmatrix}z\\\infty_u\end{smallmatrix}),\\
&z\in\overline \R^u.
\end{split}
\end{equation*}
Because $ F$ is uniformly continuous, $ J$ is uniformly continuous, too, and the map $z \mapsto \Var(U(z))$, $z\in\overline \R^u$, does not vanish everywhere or $T( J)$-almost everywhere. For that reason, the distribution function of  $\sup_{z\in\overline\R^u}|U(z)|$ or $\int U^2(z) T(  J)(\mathrm d z)$ is strictly increasing on the non-negative half-line.
\begin{rmk}
\begin{itemize}
\item[a)] Modify the model introduced in Section \ref{einfuhr} in the following way. Assume the error variable $Z_i$ has a centrally symmetric distribution function for all $i\in\N$. Furthermore, the distribution $\mathcal L(Z_i)$ is possibly unknown  for all $i\in\N$. Suppose the user has to treat the testing problem of central symmetry
\begin{equation*}
\mathrm H:\mathcal L(Y_1)=\mathcal L(-Y_1),~\mathrm K:\mathcal L(Y_1)\neq \mathcal L(-Y_1)
\end{equation*}
on the basis of the observations $Y_1+Z_1,\dots,Y_n+Z_n$. Putting $\alpha_{i,n}=\frac 1 n$ and $U_i=Y_i+Z_i$, $i=1,\dots,n$, the test (\ref{Test3}) is applicable to this testing problem. It follows from the results in this section that the test is of asymptotically exact size $\alpha$ and consistent.
\item[b)] The test (\ref{Test3}) is applicable to the hypothesis 
\begin{equation*}
\forall i\in\lbrace 1,\dots,n\rbrace :C_i^+=C_i^-.
\end{equation*}
 Putting $\alpha_{i,n}=\frac 1 n$, $i=1,\dots,n$, the test (\ref{Test3}) is also applicable to the hypothesis 
\begin{equation*}
\exists \pi_n:\lbrace 1,\dots,n\rbrace\rightarrow\lbrace 1,\dots,n\rbrace~\text{bijective}~\forall i\in\lbrace 1,\dots,n\rbrace :C_i^+=C_{\pi_n(i)}^-.
\end{equation*}
The test is of asymptotically exact size $\alpha$ and  consistent with respect to suitable alternatives to this hypotheses. This follows from the results in this section. 
\item[c)] Assume $C_i=C$, $i\in\N$. Putting $\alpha_{i,n}=\frac 1 n$, $i=1,\dots,n$, the hypothesis (\ref{testpr3}) is a common hypothesis of central symmetry
\begin{equation*}
\mathrm H:C^+=C^-,
\end{equation*}
the test statistic (\ref{KS3}) or (\ref{CvM3}) is a common Kolmogorov-Smirnov  or  Cram\'er-von-Mises statistic for testing the hypothesis of central symmetry and the test (\ref{Test3}) is a common  Kolmogorov-Smirnov  or  Cram\'er-von-Mises test for testing the hypothesis of central symmetry in a multivariate setting. In this sense, a generalization of the identically distributed case  is considered.
\item[d)]
If $\alpha_{i,n}=\frac 1 n$, $i=1,\dots,n$, the test statistic (\ref{KS3}) or (\ref{CvM3}) has the form of a  common Kolmogorov-Smirnov  or  Cram\'er-von-Mises type statistic for testing the hypothesis of central symmetry and the test (\ref{Test3}) has the form of a common  Kolmogorov-Smirnov  or  Cram\'er-von-Mises type test for testing the hypothesis of central symmetry in the identically distributed case. I.e., the results in this section imply that the common  Kolmogorov-Smirnov  or  Cram\'er-von-Mises type testing procedure  for testing the hypothesis of central symmetry is applicable if the data come from different distributions and the mentioned conditions are fulfilled.
\end{itemize}
\end{rmk}

\subsection{Testing independence}\label{ind}

Consider the one sample case with dimension $m\ge 2$, $m=k+\ell$, $k,\ell\in\N$, and that the sequence of random vectors $X_1,X_2,\dots$ is given by 
\begin{equation*} 
X_1=\begin{pmatrix}A_1\\B_1\end{pmatrix},X_2=\begin{pmatrix}A_2\\B_2\end{pmatrix},\dots,
\end{equation*}
with $ \R^k$-valued random vectors $A_1,A_2,\dots$ and $ \R^\ell$-valued random vectors $B_1,B_2,\dots$. For $i\in\N$, denote by  $F_i^A$ and $F_i^B$ the distribution functions defined on $\overline\R^k$ and $\overline\R^\ell$ of  $A_i$ and $B_i$, respectively. Suppose $F_1,\dots,F_n$ are unknown, $\alpha_{1,n},\dots,\alpha_{n,n}$ are known and that the user has to verify the hypothesis of independence
\begin{equation}\label{testpr5}
\forall (a,b)\in\overline\R^k\times\overline\R^\ell: \mathbb F_n(\begin{smallmatrix}a\\b\end{smallmatrix})=\mathbb F_n(\begin{smallmatrix}a\\\infty_{\ell}\end{smallmatrix}) \mathbb F_n(\begin{smallmatrix}\infty_k\\b\end{smallmatrix})
\end{equation}
on the basis of the observations $X_1,\dots,X_n$.

\medskip
In this situation, the maps  $T:\mathcal D(\overline \R^{m})\rightarrow \mathcal D(\overline \R^m)$ and $Q:\mathcal D(\overline \R^{m})\rightarrow \mathcal D(\overline \R^m)$ are given by the continuous maps 
\begin{equation*}
T( K)(\begin{smallmatrix}a\\b\end{smallmatrix})= K(\begin{smallmatrix}a\\b\end{smallmatrix}),~Q( K)(\begin{smallmatrix}a\\b\end{smallmatrix})= K(\begin{smallmatrix}a\\\infty_\ell\end{smallmatrix}) K(\begin{smallmatrix}\infty_k\\b\end{smallmatrix}),~(a,b)\in\overline\R^k\times\overline\R^\ell,~ K\in \mathcal D(\overline \R^{m}).
\end{equation*}
Let $ K$ and $L$ be arbitrary objects from the definition of the uniformly Hadamard differentiability given above. Now, the maps $T$ and $Q$ are uniformly Hadamard differentiable in $ K$ with Hadamard derivative 
\begin{equation*}
\mathrm d T( K)( L)(\begin{smallmatrix}a\\b\end{smallmatrix})= L(\begin{smallmatrix}a\\b\end{smallmatrix})~\text{and}~\mathrm d Q( K)( L)(\begin{smallmatrix}a\\b\end{smallmatrix})= K(\begin{smallmatrix}a\\\infty_\ell\end{smallmatrix}) L(\begin{smallmatrix}\infty_k\\b\end{smallmatrix})+ K(\begin{smallmatrix}\infty_k\\b\end{smallmatrix}) L(\begin{smallmatrix}a\\\infty_\ell\end{smallmatrix}),~(a,b)\in\overline\R^k\times\overline\R^\ell,
\end{equation*}
respectively. Let $ j:\mathcal D(\overline \R^{m})\rightarrow\mathcal D(\overline \R^{m})$ be the continuous map
\begin{equation*}
j( K)(\begin{smallmatrix}a\\b\end{smallmatrix})= K(\begin{smallmatrix}a\\\infty_\ell\end{smallmatrix}) K(\begin{smallmatrix}\infty_k\\b\end{smallmatrix}),~(a,b)\in\overline\R^k\times\overline\R^\ell,~ K\in \mathcal D(\overline \R^{m}).
\end{equation*}
Obviously, $T\circ j=Q\circ j$ and $j( K)= K$ for all $ K\in \mathcal D(\overline \R^{m})$ with $T( K)=Q( K)$. Thus,
\begin{equation*}
\hat{ \mathbb J} _{N,m}(\begin{smallmatrix}a\\b\end{smallmatrix})= \hat{\mathbb F}_n(\begin{smallmatrix}a\\\infty_\ell\end{smallmatrix}) \hat{\mathbb F}_n(\begin{smallmatrix}\infty_k\\b\end{smallmatrix})= \hat{\mathbb F}_n^A(a) \hat{\mathbb F}_n^B(b),~(a,b)\in\overline\R^k\times\overline\R^\ell,
\end{equation*}
where 
\begin{equation*}
\hat{\mathbb F}_n^A(a):=\sum_{i=1}^n\alpha_{i,n}\I(A_i\le a),~\hat{\mathbb F}_n^B(b):=\sum_{i=1}^n\alpha_{i,n}\I(B_i\le b),~(a,b)\in\overline\R^k\times\overline\R^\ell.
\end{equation*}
Now, the test statistics have the expressions
\begin{equation}\label{KS5}
\KS_n=\sqrt n\sup_{(a,b)\in\overline\R^k\times\overline\R^\ell}|  \hat{\mathbb F}_n(\begin{smallmatrix}a\\b\end{smallmatrix})-\hat{\mathbb F}_n^A(a)\hat{\mathbb F}_n^B(b)|
\end{equation}
and 
\begin{equation}\label{CvM5}
\CvM_n=n\int\int\big(\hat{\mathbb F}_n(\begin{smallmatrix}a\\b\end{smallmatrix})-\hat{\mathbb F}_n^A(a)\hat{\mathbb F}_n^B(b)\big)^2\hat{\mathbb F}_n^A(\mathrm d a)\hat{\mathbb F}_n^B(\mathrm d b).
\end{equation}
In addition, testing procedure
\begin{equation}\label{Test5}
\text{``Reject}~\mathrm H_n,~\text{iff}~\KS_n\ge c_{n;1-\alpha}\text{''}~\text{or testing procedure}~\text{``Reject}~\mathrm H_n,~\text{iff}~\CvM_n\ge d_{n;1-\alpha}\text{''} 
\end{equation}
is suggested, where $c_{n;1-\alpha}$ and $d_{n;1-\alpha}$ are calculated under $(F_1,\dots,F_n)=(\hat{ \mathbb J} _{N,m},\dots,\hat{ \mathbb J} _{N,m})$.  If $k=\ell$, testing procedure (\ref{Test5}),  in particular the statistic (\ref{KS5}) or (\ref{CvM5}), is invariant under the transformation of the data
\begin{equation*}
\begin{split}
\big((a_1,b_1),\dots,(a_n,b_n)\big)\longmapsto\big((b_1,a_1),\dots,(b_n,a_n)\big),~\big((a_1,b_1),\dots,(a_n,b_n)\big)\in\times_{i=1}^n(\R^k\times\R^k).
\end{split}
\end{equation*}
Moreover, $\mathbb U$ is a Gaussian process with expectation function identically equal to zero. Moreover, a calculation yields
\begin{equation*}
\begin{split}
\Var\big(U(\begin{smallmatrix}a\\b\end{smallmatrix})\big)=&\Var\big(M_m(\begin{smallmatrix}a\\b\end{smallmatrix})\big)+ J^2 (\begin{smallmatrix}a\\\infty_\ell\end{smallmatrix})\Var\big(M_m(\begin{smallmatrix}\infty_k\\b\end{smallmatrix})\big)+ J^2 (\begin{smallmatrix}\infty_k\\b\end{smallmatrix})\Var\big(M_m(\begin{smallmatrix}a\\\infty_\ell\end{smallmatrix})\big)
\\
&-2 J (\begin{smallmatrix}a\\\infty_\ell\end{smallmatrix})\Cov\big(M_m(\begin{smallmatrix}a\\b\end{smallmatrix}),M_m(\begin{smallmatrix}\infty_k\\b\end{smallmatrix})\big)
-2 J (\begin{smallmatrix}\infty_k\\b\end{smallmatrix})\Cov\big(M_m(\begin{smallmatrix}a\\b\end{smallmatrix}),M_m(\begin{smallmatrix}a\\\infty_\ell\end{smallmatrix})\big)\\
&+2 J (\begin{smallmatrix}a\\\infty_\ell\end{smallmatrix}) J (\begin{smallmatrix}\infty_k\\b\end{smallmatrix})\Cov\big(M_m(\begin{smallmatrix}\infty_k\\b\end{smallmatrix}),M_m(\begin{smallmatrix}a\\\infty_\ell\end{smallmatrix})\big)\\
=& J (\begin{smallmatrix}a\\\infty_\ell\end{smallmatrix})\big(1- J (\begin{smallmatrix}a\\\infty_\ell\end{smallmatrix})\big) J (\begin{smallmatrix}\infty_k\\b\end{smallmatrix})\big(1- J (\begin{smallmatrix}\infty_k\\b\end{smallmatrix})\big),~(a,b)\in\overline\R^k\times\overline\R^\ell.
\end{split}
\end{equation*}
Because $F$ is uniformly continuous, $J$ is uniformly continuous, too, and the map $(\begin{smallmatrix}a\\b\end{smallmatrix}) \mapsto \Var(U(\begin{smallmatrix}a\\b\end{smallmatrix}))$, $(a,b)\in\overline \R^k\times\overline\R^\ell$, does not vanish everywhere or $T( J)$-almost everywhere. For that reason, the distribution function of  $\sup_{(a,b)\in\overline\R^k\times\overline\R^l}|U(\begin{smallmatrix}a\\b\end{smallmatrix})|$ or $\int U^2(\begin{smallmatrix}a\\b\end{smallmatrix}) T(  J)(\mathrm d (\begin{smallmatrix}a\\b\end{smallmatrix}))$ is strictly increasing on the non-negative half-line.
\begin{rmk}
\begin{itemize}
\item[a)] Modify the model introduced in Section \ref{einfuhr} in the following way. Replace the sequence of random variables $Y_1,Y_2,\dots$ with a sequence of independent and identically distributed $\R\times\R$-valued random vectors $(Y_1,W_1),(Y_2,W_2)$ and assume this sequence is still independent of  $Z_1,Z_2,\dots$. Now, the distribution of the error variable $\mathcal L(Z_i)$ and the error function $e_i$ are possibly unknown for all $i\in\N$.  Suppose the user has to treat the testing problem of independence
\begin{equation*}
\mathrm H:\mathcal L(Y_1,W_1)=\mathcal L(Y_1)\otimes\mathcal L(W_1),~\mathrm K:\mathcal L(Y_1,W_1)\neq\mathcal L(Y_1)\otimes\mathcal L(W_1)
\end{equation*}
on the basis of the observations $(e_1(Y_1,Z_1),W_1),\dots,(e_n(Y_n,Z_n),W_n)$. Putting \linebreak $\alpha_{i,n}=\frac 1 n$ and $X_i=(e_i(Y_i,Z_i),W_i)$, $i=1,\dots,n$, the test (\ref{Test5}) is applicable to this testing problem. If the distribution function of $(e_i(Y_i,Z_i),W_i)$ converges uniformly on $\R\times\R$ to the distribution function of $(e(Y_1,Z),W_1)$ and the distribution function of $(e(Y_1,Z),W_1)$ is uniformly continuous, it follows from the results in this section that the test is of asymptotically exact size $\alpha$ and consistent.
\item[b)] Assume $F_1^B=\dots=F_n^B$. Then,  the test (\ref{Test5}) is applicable to the hypothesis 
\begin{equation*}
\forall i\in\lbrace 1,\dots,n\rbrace~\forall(a,b)\in\overline\R^k\times\overline\R^\ell  :F_i(\begin{smallmatrix}a\\b\end{smallmatrix})=F_i^A(a)F_i^B(b).
\end{equation*}
The test is of asymptotically exact size $\alpha$ and consistent with respect to suitable alternatives to this hypothesis. This follows from the results in this section. 
\item[c)] Assume $F_i=F$, $i\in\N$. Putting $\alpha_{i,n}=\frac 1 n$, $i=1,\dots,n$, the hypothesis (\ref{testpr5}) is a common hypothesis of independence
\begin{equation*}
\mathrm H:\forall(a,b)\in\overline\R^k\times\overline\R^\ell  :F(\begin{smallmatrix}a\\b\end{smallmatrix})=F^A(a)F^B(b),
\end{equation*}
the test statistic (\ref{KS5}) or (\ref{CvM5}) is a common Kolmogorov-Smirnov  or  Cram\'er-von-Mises statistic for testing the hypothesis of independence and the test (\ref{Test5}) is a common  Kolmogorov-Smirnov  or  Cram\'er-von-Mises test for testing the hypothesis of independence in a multivariate setting. In this sense, a generalization of the identically distributed case  is considered.
\item[d)] If $\alpha_{i,n}=\frac 1 n$, $i=1,\dots,n$, the test statistic (\ref{KS5}) or (\ref{CvM5}) has the form of a  common Kolmogorov-Smirnov  or  Cram\'er-von-Mises type statistic for testing the hypothesis of independence and the test (\ref{Test5}) has the form of a common  Kolmogorov-Smirnov  or  Cram\'er-von-Mises type test for testing the hypothesis of independence in the identically distributed case. I.e., the results in this section imply that the common  Kolmogorov-Smirnov  or  Cram\'er-von-Mises type testing procedure  for testing the hypothesis of independence is applicable if $A_1,A_2,\dots$ come from different distributions, $B_1,B_2,\dots$ are identically distributed and the mentioned conditions are fulfilled.
\end{itemize}
\end{rmk}

\section{Simulations}
Empirical results for the probabilities of the error of the first kind and the power values of the mentioned tests for finite sample sizes are presented. The simulation based on 1000 replications and the Monte-Carlo procedures based on 500 replications.
\subsection{Goodness-of-fit with hypotheses given by a specific distribution}\label{ex1}
Assume $m=1$, $\alpha_{i,n}=\frac 1 n$, $i=1,\dots,n$, and consider the distribution function $\ell(\mu_i,1)$ of the logistic distribution  with parameters $\mu_i\in\R$ and $1$, $i\in\N$. Suppose $(\mu_i)_{i\in\N}$ is known and there exists a $\mu\in \R$ with $\lim_{i\rightarrow\infty}\mu_i=\mu$. The upper part of Table \ref{tabelle1} shows the empirical error probabilities of the first kind in the case of $F_i=\ell(\frac{1}{\log(i+1)},1)$ and $ G_n=\frac 1 n \sum_{i=1}^n\ell(\frac{1}{\log(i+1)},1)$, $F_i=\ell(\frac{1}{\sqrt i},1)$ and $ G_n=\frac 1 n \sum_{i=1}^n\ell(\frac{1}{\sqrt i},1)$ or $F_i=\ell(\frac{1}{ i},1)$ and $ G_n=\frac 1 n \sum_{i=1}^n\ell(\frac{1}{ i},1)$, $i\in\N$. The lower part of Table \ref{tabelle1} shows the empirical power values in the case of $F_i=\mathrm L(-\frac{1}{ i},\frac 1 2)$ (Laplace distribution) and $ G_n=\frac 1 n \sum_{i=1}^n\ell(\frac{1}{\log( i+1)},1)$,  $F_i=\mathrm C(-\frac{1}{\log (i+1)},\frac12)$ (Cauchy distribution) and $ G_n=\frac 1 n \sum_{i=1}^n\ell(\frac{1}{\sqrt i},1)$, or $F_i=\mathrm N(-\frac{1}{ \log(i+1)},\frac12)$ (normal distribution) and $ G_n=\frac 1 n \sum_{i=1}^n\ell(\frac{1}{ i},1)$, $i\in\N$. 
\begin{table}
\begin{center}
\setlength{\tabcolsep}{2pt}
\begin{tabular}{cccccccccc}
&&&\multicolumn{2}{c}{$\alpha=0.025$}& \multicolumn{2}{c}{$\alpha=0.05$}&\multicolumn{2}{c}{$\alpha=0.1$} \\
$F_i$&$ G_n$&& KS & CvM   &  KS & CvM  & KS & CvM \\ 
\hline
\multirow{2}{*}{$\ell(\frac{1}{\log (i+1)},1)$}&\multirow{2}{*}  {$\frac 1 n \sum_{i=1}^n\ell(\frac{1}{\log (i+1)},1)$} &$n=25$ & 0.033 & 0.033 & 0.049 & 0.049 & 0.097 & 0.097\\
&&$n=50$& 0.026 & 0.026 & 0.044 & 0.044 & 0.096 & 0.096  \\
\hline
\multirow{2}{*}{$\ell(\frac{1}{\sqrt i},1)$}&\multirow{2}{*} {$\frac 1 n \sum_{i=1}^n\ell(\frac{1}{\sqrt i},1)$} &$n=25$ & 0.032 & 0.032 & 0.051 & 0.051 & 0.105 & 0.105\\
&&$n=50$& 0.026 & 0.026 & 0.066 & 0.066 & 0.106 & 0.106  \\
\hline
\multirow{2}{*}{$\ell(\frac{1}{ i},1)$}&\multirow{2}{*} {$\frac 1 n \sum_{i=1}^n\ell(\frac{1}{ i},1)$} &$n=25$ & 0.030 & 0.030 & 0.049 & 0.049 & 0.102 & 0.102\\
&&$n=50$& 0.023 & 0.023 & 0.051 & 0.051 & 0.109 & 0.109  \\
\hline
\multirow{2}{*} {$\mathrm L(-\frac{1}{ i},\frac 1 2)$ } &\multirow{2}{*}{$\frac 1 n \sum_{i=1}^n\ell(\frac{1}{\log( i+1)},1)$}&$n=25$& 0.510 & 0.510 & 0.715 & 0.715 & 0.879 & 0.879\\
&&$n=50$& 0.640 & 0.640& 0.804 & 0.804 & 0.928 & 0.928  \\
\hline
\multirow{2}{*} {$\mathrm C(-\frac{1}{\log (i+1)},\frac12)$}&\multirow{2}{*}{$\frac 1 n \sum_{i=1}^n\ell(\frac{1}{\sqrt i},1)$} &$n=25$& 0.780 & 0.780 & 0.862 & 0.862 & 0.931 & 0.931\\
&&$n=50$& 0.732 & 0.732 & 0.837 & 0.837 & 0.913 & 0.913    \\
\hline
\multirow{2}{*} {$\mathrm N(-\frac{1}{ \log(i+1)},\frac12)$}&\multirow{2}{*}{$\frac 1 n \sum_{i=1}^n\ell(\frac{1}{ i},1)$} &$n=25$& 0.805 & 0.805 & 0.954 & 0.954 & 0.991 & 0.991\\
&&$n=50$& 0.750 & 0.750 & 0.909 & 0.909 & 0.991 & 0.991   \\
\end{tabular}
\caption{\label{tabelle1}Simulation results for goodness-of-fit with hypotheses given by a specific distribution.}
\end{center}
\end{table}

\subsection{Goodness-of-fit with hypotheses given by a family of distributions}\label{ex2}
Let $m=d=1$, $\Theta=(0,\infty)$, $ \alpha_{i,n}=\frac 1 n$, $i=1,\dots,n$, and consider the distribution function Exp$(\vartheta+\mu_i)$ of the exponential distribution  with rate $(\vartheta+\mu_i)\in (0,\infty)$, $i\in\N$. Suppose $(\mu_i)_{i\in\N}$ is known and there exists a $\mu\in \R$ with $\lim_{i\rightarrow\infty}\mu_i=\mu$ and $\vartheta+\mu>0$. An estimator introduced in Remark \ref{bem3} is used. The upper part of Table \ref{tabelle3} shows the empirical error probabilities of the first kind in the case of $F_i=\operatorname{Exp}(1+\frac{1}{\log (i+1)})$ and $ G_n(\cdot,\vartheta)=\frac 1 n \sum_{i=1}^n \operatorname{Exp}(\vartheta+\frac{1}{\log (i+1)})$, $F_i=\operatorname{Exp}(2+\frac{1}{\sqrt i})$ and $ G_n(\cdot,\vartheta)=\frac 1 n \sum_{i=1}^n \operatorname{Exp}(\vartheta+\frac{1}{\sqrt i})$ or  $F_i=\operatorname{Exp}(3+\frac{1}{ i})$ and $ G_n(\cdot,\vartheta)=\frac 1 n \sum_{i=1}^n \operatorname{Exp}(\vartheta+\frac{1}{ i})$, $i\in\N$.  The lower part of Table \ref{tabelle3} shows the empirical power values in the case of $F_i=\mathrm W(1+\frac{1}{\log(1+i)},1)$ (Weibull distribution) and $ G_n(\cdot,\vartheta)=\frac 1 n \sum_{i=1}^n \operatorname{Exp}(\vartheta+\frac{1}{\log (i+1)})$, $F_i=\operatorname{IG}(\frac 2 3,1+\frac{1}{\sqrt i})$ (inverse Gaussian distribution) and $ G_n(\cdot,\vartheta)=\frac 1 n \sum_{i=1}^n \operatorname{Exp}(\vartheta+\frac{1}{\sqrt i})$, or $F_i=\mathrm G(\frac 1 2,\frac{1}{1+\frac{1}{i}})$ (gamma distribution) and $ G_n(\cdot,\vartheta)=\frac 1 n \sum_{i=1}^n \operatorname{Exp}(\vartheta+\frac{1}{ i})$,  $i\in\N$.
\begin{table}
\begin{center}
\setlength{\tabcolsep}{2pt}
\begin{tabular}{cccccccccc}
&&&\multicolumn{2}{c}{$\alpha=0.025$}& \multicolumn{2}{c}{$\alpha=0.05$}&\multicolumn{2}{c}{$\alpha=0.1$} \\
$F_i$&$ G_n(\cdot,\vartheta)$&& KS & CvM   &  KS & CvM  & KS & CvM \\ 
\hline
\multirow{2}{*}{$\operatorname{Exp}(1+\frac{1}{\log (i+1)})$}&\multirow{2}{*}  {$\frac 1 n \sum_{i=1}^n \operatorname{Exp}(\vartheta+\frac{1}{\log (i+1)})$} &$n=25$ & 0.035 & 0.035 & 0.066 & 0.066 & 0.102 & 0.102\\
&&$n=50$ & 0.024 & 0.024 & 0.039 & 0.039 & 0.083 & 0.083  \\
\hline
\multirow{2}{*}{$\operatorname{Exp}(2+\frac{1}{\sqrt i})$}&\multirow{2}{*} {$\frac 1 n \sum_{i=1}^n \operatorname{Exp}(\vartheta+\frac{1}{\sqrt i})$} &$n=25$ & 0.025 & 0.025 & 0.048 & 0.048 & 0.088 & 0.088\\
&&$n=50$& 0.025 & 0.025 & 0.049 & 0.049 & 0.096 & 0.096 \\
\hline
\multirow{2}{*}{$\operatorname{Exp}(3+\frac{1}{ i})$}&\multirow{2}{*} {$\frac 1 n \sum_{i=1}^n \operatorname{Exp}(\vartheta+\frac{1}{ i})$} &$n=25$ & 0.029 & 0.029 & 0.055 & 0.055 & 0.101 & 0.101\\
&&$n=50$& 0.018 & 0.018 & 0.045 & 0.045 & 0.087 & 0.087 \\
\hline
\multirow{2}{*} {$\mathrm W(1+\frac{1}{\log(1+i)},1)$ } &\multirow{2}{*}{$\frac 1 n \sum_{i=1}^n \operatorname{Exp}(\vartheta+\frac{1}{\log (i+1)})$}&$n=25$ & 0.241 & 0.241 & 0.386 & 0.386 & 0.584 & 0.584\\
&&$n=50$& 0.257 & 0.257 & 0.421 & 0.421 & 0.622 & 0.622  \\
\hline
\multirow{2}{*} {$\operatorname{IG}(\frac 2 3,1+\frac{1}{\sqrt i})$}&\multirow{2}{*}{$\frac 1 n \sum_{i=1}^n \operatorname{Exp}(\vartheta+\frac{1}{\sqrt i})$} &$n=25$ & 0.546 & 0.546 & 0.722 & 0.722 & 0.833 & 0.833\\
&&$n=50$ & 0.796 & 0.796 & 0.902 & 0.902 & 0.955 & 0.955\\
\hline
\multirow{2}{*} {$\mathrm G(\frac 1 2,\frac{1}{1+\frac{1}{i}})$}&\multirow{2}{*}{$\frac 1 n \sum_{i=1}^n \operatorname{Exp}(\vartheta+\frac{1}{ i})$} &$n=25$ & 0.628 & 0.628 & 0.731 & 0.731 & 0.803 & 0.803\\
&&$n=50$& 0.870 & 0.870 & 0.910 & 0.910 & 0.939 & 0.939\\
\end{tabular}
\caption{\label{tabelle3}Simulation results for goodness-of-fit with hypotheses given by a family of distributions.}
\end{center}
\end{table}

\subsection{Homogeneity}\label{ex4}
Let $m=1$ and $\alpha_{i,n}=\beta_{i,n}=\frac 1 n$, $i=1,\dots,n$. The upper part of Table \ref{tabelle7} shows the empirical error probabilities of the first kind in the case of $F_i=G_i=\mathrm W(1+\frac{1}{\log(1+i)},1)$  (Weibull distribution), $F_i=G_i=\operatorname{IG}(\frac 2 3,1+\frac{1}{\sqrt i})$  (inverse Gaussian distribution), or $F_i=G_i=\operatorname{Exp}(1+\frac{1}{ i})$ (exponential distribution),  $i\in\N$. The lower  part of Table \ref{tabelle7} shows the empirical power values in the case of  $F_i=\operatorname{Exp}(1+\frac{1}{ i})$ and  $G_i=\mathrm W(1+\frac{1}{i},1/3)$, $F_i=\operatorname{IG}( 2 ,1+\frac{1}{\sqrt i})$ and  $G_i=\mathrm G(\frac 1 2,\frac{1}{1+\frac{1}{\sqrt i}})$ (gamma distribution), or $F_i=\operatorname{IG}( \frac 1 2 ,1+\frac{1}{\log (i+1)})$ and  $ G_i=\mathrm W(1+\frac{1}{\log(i+1)},\frac 1 2)$,  $i\in\N$.
\begin{table}
\begin{center}
\setlength{\tabcolsep}{2pt}
\begin{tabular}{cccccccccc}
&&&\multicolumn{2}{c}{$\alpha=0.025$}& \multicolumn{2}{c}{$\alpha=0.05$}&\multicolumn{2}{c}{$\alpha=0.1$} \\
$F_i$&$G_i$&& KS & CvM   &  KS & CvM  & KS & CvM \\ 
\hline
\multirow{2}{*}{$\mathrm W(1+\frac{1}{\log(1+i)},1)$}&\multirow{2}{*}  {$\mathrm W(1+\frac{1}{\log(1+i)},1)$} &$n=25$ & 0.026 & 0.026 & 0.056 & 0.052 & 0.110 & 0.102\\
&&$n=50$ & 0.026 & 0.030 & 0.047 & 0.050 & 0.093 & 0.098\\
\hline
\multirow{2}{*}{$\operatorname{IG}(\frac 2 3,1+\frac{1}{\sqrt i})$}&\multirow{2}{*} {$\operatorname{IG}(\frac 2 3,1+\frac{1}{\sqrt i})$} &$n=25$ & 0.031 & 0.027 & 0.053 & 0.055 & 0.101 & 0.106\\
&&$n=50$& 0.026 & 0.024 & 0.045 & 0.051 & 0.097 & 0.109 \\
\hline
\multirow{2}{*}{$\operatorname{Exp}(1+\frac{1}{ i})$}&\multirow{2}{*} {$\operatorname{Exp}(1+\frac{1}{ i})$} &$n=25$ & 0.022 & 0.021 & 0.048 & 0.040 & 0.104 & 0.088\\
&&$n=50$& 0.028 & 0.023 & 0.053 & 0.059 & 0.106 & 0.105 \\
\hline
\multirow{2}{*}{$\operatorname{Exp}(1+\frac{1}{ i})$}&\multirow{2}{*}  {$\mathrm W(1+\frac{1}{i},1/3)$} &$n=25$ & 0.266 & 0.242 & 0.363 & 0.365 & 0.566 & 0.555\\
&&$n=50$ & 0.622 & 0.571 & 0.762 & 0.759 & 0.876 & 0.891 \\
\hline
\multirow{2}{*}{$\operatorname{IG}( 2 ,1+\frac{1}{\sqrt i})$}&\multirow{2}{*} {$\mathrm G(\frac 1 2,\frac{1}{1+\frac{1}{\sqrt i}})$} &$n=25$ & 0.555 & 0.504 & 0.686 & 0.659 & 0.787 & 0.799\\
&&$n=50$& 0.899 & 0.932 & 0.964 & 0.979 & 0.990 & 0.997 \\
\hline
\multirow{2}{*}{$\operatorname{IG}( \frac 1 2 ,1+\frac{1}{\log (i+1)})$}&\multirow{2}{*} {$\mathrm W(1+\frac{1}{\log(i+1)},\frac 1 2)$} &$n=25$& 0.075 & 0.075 & 0.119 & 0.131 & 0.224 & 0.218\\
&&$n=50$& 0.281 & 0.235 & 0.389 & 0.361 & 0.569 & 0.556 \\
\end{tabular}
\caption{\label{tabelle7}Simulation results for homogeneity.}
\end{center}
\end{table}

\subsection{Central symmetry}\label{ex3}
Assume $m=1$ and $\alpha_{i,n}=\frac 1 n$, $i=1,\dots,n$. The upper part of Table \ref{tabelle5} shows the empirical error probabilities of the first kind in the case of $C_i=\mathrm N(0,1+\frac{1}{ \log(i+1)})$ (normal distribution), $C_i=\ell(0, \frac 1 2 +\frac{1}{\sqrt i})$ (logistic distribution), or $C_i=\mathrm C(0,2+\frac{1}{ i})$ (Cauchy distribution),  $i\in\N$. The lower part of Table \ref{tabelle5} shows the empirical power values in the case of $C_i=\mathrm N(\frac 1 2,1+\frac{1}{ \log(i+1)})$, $C_i=\ell(\frac 1 3, \frac 1 2 +\frac{1}{\sqrt i})$, or $C_i=\mathrm C(\frac 3 2,2+\frac{1}{ i})$,  $i\in\N$.
\begin{table}
\begin{center}
\setlength{\tabcolsep}{2pt}
\begin{tabular}{ccccccccc}
&&\multicolumn{2}{c}{$\alpha=0.025$}& \multicolumn{2}{c}{$\alpha=0.05$}&\multicolumn{2}{c}{$\alpha=0.1$} \\
$C_i$&& KS & CvM   &  KS & CvM  & KS & CvM \\ 
\hline
\multirow{2}{*}{$\mathrm N(0,1+\frac{1}{ \log(i+1)})$} &$n=25$ & 0.020 & 0.025 & 0.038 & 0.045 & 0.083 & 0.087\\
&$n=50$& 0.027 & 0.028 & 0.054 & 0.053 & 0.095 & 0.098 \\
\hline
\multirow{2}{*}{$\ell(0, \frac 1 2 +\frac{1}{\sqrt i})$} &$n=25$ & 0.033 & 0.037 & 0.055 & 0.063 & 0.088 & 0.105\\
&$n=50$& 0.024 & 0.016 & 0.045 & 0.044 & 0.102 & 0.101
 \\
\hline
\multirow{2}{*}{$\mathrm C(0,2+\frac{1}{ i})$}&$n=25$ & 0.016 & 0.023 & 0.035 & 0.046 & 0.071 & 0.089\\
&$n=50$& 0.036 & 0.035 & 0.059 & 0.059 & 0.104 & 0.111  \\
\hline
\multirow{2}{*}{$\mathrm N(\frac 1 2,1+\frac{1}{ \log(i+1)})$} &$n=25$& 0.210 & 0.248 & 0.320 & 0.374 & 0.422 & 0.472\\
&$n=50$& 0.491 & 0.581 & 0.620 & 0.699 & 0.729 & 0.787  \\
\hline
\multirow{2}{*}{$\ell(\frac 1 3, \frac 1 2 +\frac{1}{\sqrt i})$} &$n=25$ & 0.135 & 0.154 & 0.201 & 0.230 & 0.284 & 0.324\\
&$n=50$& 0.308 & 0.364 & 0.421 & 0.468 & 0.530 & 0.593 \\
\hline
\multirow{2}{*}{$\mathrm C(\frac 3 2,2+\frac{1}{ i})$}&$n=25$  & 0.379 & 0.346 & 0.493 & 0.443 & 0.604 & 0.583
\\
&$n=50$ & 0.758 & 0.705 & 0.848 & 0.800 & 0.904 & 0.873  \\
\end{tabular}
\caption{\label{tabelle5}Simulation results for central symmetry.}
\end{center}
\end{table}

\subsection{Independence}\label{ex5}
Assume $m=2$, $k=\ell=1$ and $\alpha_{i,n}=\frac 1 n$, $i=1,\dots,n$. The upper part of Table \ref{tabelle9} shows the empirical error probabilities of the first kind in the case of $F_i=\mathrm N(\frac 1 i,1)\otimes\mathrm N(0,1)$ (product of normal distributions), $F_i=\ell(\frac{1}{\sqrt i},  1 )\otimes\ell(0,  1 )$ (product of logistic distributions), or $F_i=\mathrm C(\frac{1}{\log(1+ i)},1)\otimes \mathrm C(0,1)$ (product of Cauchy distributions),  $i\in\N$. The lower part of Table \ref{tabelle9} shows the empirical power values in the case of $F_i=\mathrm N_2(\frac 1 i (\begin{smallmatrix}1\\ 1\end{smallmatrix}),(\begin{smallmatrix}2& 1 \\ 1 &2\end{smallmatrix}))$ (bivariate normal distribution), $F_i=t_1( \frac{1}{\sqrt i}(\begin{smallmatrix}1\\ 1\end{smallmatrix}),(\begin{smallmatrix}2& 1 \\ 1 &2\end{smallmatrix}))$ (bivariate $t$-distribution), or $F_i=\ell_2(\frac{1}{\log( i+1)},  1,\frac{1}{\log( i+1)},  1 )$ (bivariate logistic distribution),  $i\in\N$. 
\begin{table}
\begin{center}
\setlength{\tabcolsep}{2pt}
\begin{tabular}{ccccccccc}
&&\multicolumn{2}{c}{$\alpha=0.025$}& \multicolumn{2}{c}{$\alpha=0.05$}&\multicolumn{2}{c}{$\alpha=0.1$} \\
$F_i$&& KS & CvM   &  KS & CvM  & KS & CvM \\ 
\hline
\multirow{2}{*}{$\mathrm N(\frac 1 i,1)\otimes\mathrm N(0,1)$} &$n=25$ & 0.047 & 0.027 & 0.088 & 0.056 & 0.170 & 0.115\\
&$n=50$& 0.035 & 0.024 & 0.070 & 0.052 & 0.137 & 0.105 \\
\hline
\multirow{2}{*}{$\ell(\frac{1}{\sqrt i},  1 )\otimes\ell(0,  1 )$} &$n=25$ & 0.040 & 0.023 & 0.068 & 0.046 & 0.134 & 0.083\\
&$n=50$& 0.034 & 0.018 & 0.066 & 0.045 & 0.119 & 0.104
 \\
\hline
\multirow{2}{*}{$\mathrm C(\frac{1}{\log(1+ i)},1)\otimes\mathrm  C(0,1)$}&$n=25$& 0.052 & 0.028 & 0.095 & 0.057 & 0.160 & 0.110\\
&$n=50$& 0.034 & 0.016 & 0.071 & 0.053 & 0.130 & 0.100
  \\
\hline
\multirow{2}{*}{$\mathrm N_2(\frac 1 i (\begin{smallmatrix}1\\ 1\end{smallmatrix}),(\begin{smallmatrix}2& 1 \\ 1 &2\end{smallmatrix}))$} &$n=25$& 0.487 & 0.548 & 0.600 & 0.642 & 0.720 & 0.730\\
&$n=50$& 0.725 & 0.851 & 0.816 & 0.908 & 0.895 & 0.945  \\
\hline
\multirow{2}{*}{$t_1( \frac{1}{\sqrt i}(\begin{smallmatrix}1\\ 1\end{smallmatrix}),(\begin{smallmatrix}2& 1 \\ 1 &2\end{smallmatrix}))$} &$n=25$& 0.664 & 0.678 & 0.786 & 0.787 & 0.871 & 0.870\\
&$n=50$& 0.985 & 0.987 & 0.992 & 0.993 & 0.998 & 0.997\\
\hline
\multirow{2}{*}{$\ell_2(\frac{1}{\log( i+1)},  1,\frac{1}{\log( i+1)},  1 )$}&$n=25$ & 0.965 & 0.990 & 0.986 & 0.997 & 0.997 & 0.999
\\
&$n=50$& 1.000 & 1.000 & 1.000 & 1.000 & 1.000 & 1.000\\
\end{tabular}
\caption{\label{tabelle9}Simulation results for independence.}
\end{center}
\end{table}

\section{Poofs}

\begin{rmk} \label{lem-1}
Let $\gamma_{1,n},\dots,\gamma_{n,n}$ be a real numbers with
\begin{equation*}
\gamma_{i,n}\ge 0,~i=1,\dots,n,~\lim_{n\rightarrow\infty}\max_{1\le i\le n}\gamma_{i,n}=0,~\lim_{n\rightarrow\infty}\sum_{i=1}^n\gamma_{i,n}=\gamma\in[1,\infty),
\end{equation*}
and let $r_{1,n},\dots,r_{n,n}$ be other real numbers such that for all $n\in\N$ $|r_{i,n}|<c\in(0,\infty)$, $i=1,\dots,n$, and 
\begin{equation*}
\forall \varepsilon>0~\exists i_\varepsilon\in\N~\exists n_\varepsilon\in\N,~i_\varepsilon\le n_\varepsilon,~\forall i>i_\varepsilon~\forall n>n_\varepsilon,~i\le n:|r_{i,n}|\le \varepsilon.
\end{equation*}
Then, 
\begin{equation*}
\lim_{n\rightarrow\infty}\sum_{i=1}^n\gamma_{i,n}r_{i,n}=0.
\end{equation*}
\end{rmk}
\begin{proof}[\bf Proof of Lemma \ref{lem0}]
Define another triangular array $\eta_{1,n},\dots,\eta_{n,n}$ of  row-wise independent  random vectors  with values in  $\mathcal X:=\R^m\times[0,\infty)$ by $\eta_{1,n}:=(\xi_{1,n},n\alpha_{1,n}),\dots,\eta_{n,n}:=(\xi_{n,n},n\alpha_{n,n})$. Regard the process $\mathbb W_n:=(W_n(x);x\in\overline\R^m)$ as a process  $\widetilde {\mathbb W}_n:=(\widetilde { W}_n(f);f \in \mathcal V)$, where 
\begin{equation*}
\mathcal V:=\lbrace f; f:\mathcal X\rightarrow \R, f(w,u)=u\I(w\le x), (w,u)\in\mathcal X,x\in \overline\R^m\rbrace
\end{equation*}
and
\begin{equation*}
\widetilde { W}_n(f):=\frac{1}{\sqrt n}\bigg( \sum_{i=1}^nf(\eta_{i,n})- \sum_{i=1}^n\E\big(f(\eta_{i,n})\big)\bigg),~f\in\mathcal V.
\end{equation*}
It is possible to follow the argumentation in 4.2 in \cite{zie} and to check the conditions mentioned there in order to show that the process $\widetilde {\mathbb W}_n $ is asymptotically equicontinuous. By Dudley \cite{dud}, $\lbrace \lbrace (w,u)\in\mathcal X;w_j-x>0\rbrace; x\in\overline\R\rbrace$ is a Vapnik-\v{C}hervonenkis class, $j=1,\dots,m$.  For definition and details of Vapnik-\v{C}hervonenkis classes, see \cite{van}. Using Lemma 2.6.17 (i) and (ii) in \cite{van}, it follows that $\mathcal C:=\lbrace \lbrace (w,u)\in\mathcal X;w\le x\rbrace; x\in\overline\R^m\rbrace$  is a Vapnik-\v{C}hervonenkis class, too. Consequently, $\mathcal H:=\lbrace h; h:\mathcal X\rightarrow \R,h(w,u)=\I((w,u)\in C), (w,u)\in\mathcal X,C\in\mathcal C\rbrace=\lbrace h; h:\mathcal X\rightarrow \R,h(w,u)=\I(w\le x), (w,u)\in\mathcal X,x\in \overline\R^m\rbrace$ is a Vapnik-\v{C}hervonenkis graph class. Putting $g:\mathcal X\rightarrow \R$, $g(w,u):=u$, $(w,u)\in\mathcal X$, Lemma 2.16.18 (vi) in \cite{van} implies that $g\cdot\mathcal H:=\lbrace gh;h\in\mathcal H\rbrace$ is  a Vapnik-\v{C}hervonenkis graph class, too. Finally, $\mathcal V=g\cdot\mathcal H$ is a Vapnik-\v{C}hervonenkis graph class.

\medskip
Let $E$ be the distribution function of the exponential distribution with rate parameter one. The map
\begin{equation*}
d(f_1,f_2):=\int|f_1(z)-f_2(z)|(\Phi\otimes E)(\mathrm d z),~f_1,f_2\in\mathcal V,
\end{equation*}
defines a metric on $\mathcal V$ and  $(\mathcal V,d)$ is totally bounded. In fact,
\begin{equation*}
\begin{split}
\rho(x,y)&=\int|f_1(z)-f_2(z)|(\Phi\otimes E)(\mathrm d z)=d(f_1,f_2),\\
&f_1(w,u):=u\I(w\le x),~f_2(w,u):=u\I(w\le y),~(w,u)\in\mathcal X,~x,y\in \overline\R^m.
\end{split}
\end{equation*}
Clearly, $\mathcal V$ is a class of  functions with envelope $g$ given above. Moreover, from Lemma 2.4 in \cite{zie}, $\mathcal V$ has uniformly integrable entropy. (\ref{seq}) yields
\begin{equation*}
\sup_{n\in\N}\frac 1 n\sum_{i=1}^n\E\big(g^2(\xi_{i,n},n\alpha_{i,n})\big)=\sup_{n\in\N}\sum_{i=1}^nn\alpha_{i,n}^2<\infty.
\end{equation*}
Remark \ref{lem-1} and (\ref{seq}) imply
\begin{equation*}
\begin{split}
&\limsup_{n\rightarrow\infty}\sup_{f_1,f_2\in\mathcal V,d(f_1,f_2)\le \delta}\sqrt{\frac 1 n\sum_{i=1}^n\E\Big(\big(f_1(\xi_{i,n},n\alpha_{i,n})-f_2(\xi_{i,n},n\alpha_{i,n})\big)^2\Big)}\\
\le&\sqrt{\sup_{x,y\in \overline\R^m,\rho(x,y)\le \delta}\kappa\Big( K(x)-2 K\big(\min(x,y)\big)+ K( y)\Big)}\longrightarrow 0~\text{as}~\delta\downarrow 0.
\end{split}
\end{equation*}
Moreover,
\begin{equation*}
\begin{split}
\forall t>0:&\frac 1 n\sum_{i=1}^n\E\Big(g^2(\xi_{i,n},n\alpha_{i,n})\I\big(g(\xi_{i,n},n\alpha_{i,n})>\sqrt nt\big)\Big)\\
&\le n\sum_{i=1}^n\alpha_{i,n}^2\I(\sqrt n\max_{1\le j\le n}\alpha_{j,n}>t)\longrightarrow 0~\text{as}~n\rightarrow\infty.
\end{split}
\end{equation*}
In all, it results from 4.2 in \cite{zie} that the process $\widetilde {\mathbb W}_n$ is asymptotically equicontinuous with respect to the metric space  $(\mathcal V,d)$, i.e.,
\begin{equation*}
\forall\varepsilon>0:\limsup_{n\rightarrow\infty}P^*\bigg(\sup_{f_1,f_2\in\mathcal V,d(f_1,f_2)\le \delta}|\widetilde W_n(f_1)-\widetilde W_n(f_2)|>\varepsilon\bigg) \longrightarrow 0~\text{as}~\delta\downarrow 0.
\end{equation*}
It follows that
\begin{equation*}
\begin{split}
\forall\varepsilon>0:&\limsup_{n\rightarrow\infty}P^*\bigg(\sup_{x,y\in \overline\R^m,\rho(x,y)\le \delta}|W_n(x)-W_n(y)|>\varepsilon\bigg)\\
&=\limsup_{n\rightarrow\infty}P^*\bigg(\sup_{f_1,f_2\in\mathcal V,d(f_1,f_2)\le \delta}|\widetilde W_n(f_1)-\widetilde W_n(f_2)|>\varepsilon\bigg) \longrightarrow 0~\text{as}~\delta\downarrow 0,
\end{split}
\end{equation*}
i.e., the statement.
\end{proof}

\begin{rmk}\label{bem1}
Assume a sequence of distribution functions $K_1,K_2,\dots$ defined on $\overline\R^m$ and a uniformly continuous distribution function $K$ defined on $\overline\R^m$ is given such that
\begin{equation*}
\lim_{i\rightarrow\infty}K_i= K~\text{uniformly on}~\overline\R^m.
\end{equation*}
Then, Remark \ref{lem-1} and (\ref{seq}) imply
\begin{equation*}
\lim_{n\rightarrow\infty}\sum_{i=1}^n\alpha_{i,n} K_i= K~\text{uniformly on}~ \overline\R^m,~\lim_{n\rightarrow\infty}n\sum_{i=1}^n\alpha_{i,n}^2 K_i=\kappa K~\text{uniformly on}~ \overline\R^m,
\end{equation*}
as well as 
\begin{equation*}
\lim_{n\rightarrow\infty} \sup_{x,y\in\overline\R^m}\bigg| n\sum_{i=1}^n\alpha_{i,n}^2 K_i(x) K_i(y)-\kappa K(x)K(y)\bigg|=0.
\end{equation*}
\end{rmk}
\begin{proof}[\bf Proof of Theorem \ref{thm1}] Consider $n$ sufficiently large. Because ${\mathbb F}_n= G_n$, 
\begin{equation*}
U_n(x)=\sqrt n \big( \hat{\mathbb F}_n(x)-{\mathbb F}_n(x)\big),~x\in\overline\R^m.
\end{equation*}
For that reason, the process $\mathbb U_n$ is asymptotically equicontinuous with respect to the metric space $(\overline\R^m,\rho)$, see Lemma \ref{lem0}. For arbitrary $p\in\N$, $x\in\times_{i=1}^{p}\overline\R^m$, $x=(x_1,\dots,x_p)$, and arbitrary $a\in\R^p\setminus\lbrace 0 \rbrace$, put
\begin{equation*}
s_n^2:=n\sum_{i=1}^n\alpha_{i,n}^2\Var
\left(a'
\begin{pmatrix}
\I(X_i\le x_1)-F_i(x_1)\\
\vdots\\
\I(X_i\le x_p)-F_i(x_p)
\end{pmatrix}
\right).
\end{equation*}
Remark \ref{bem1} and ${\mathbb F}_n= G_n$ imply $\lim_{i\rightarrow\infty}F_i= G$ uniformly on $\overline \R^m$. Regarding Remark \ref{bem1} again,
\begin{equation*}
\lim_{n\rightarrow\infty}s_n^2=\sum_{1\le j,k\le p}a_ja_k\kappa\Big(  G\big(\min(x_j,x_k)\big)-  G(x_j)  G(x_k)\Big).
\end{equation*}
Assume without loss of generality $\lim_{n\rightarrow\infty}s_n^2>0$. With $|a|_1:=|a_1|+\dots+|a_p|$, it is
\begin{equation*}
\begin{split}
\forall t>0:&\frac {1}{ s_n^2}\sum_{i=1}^n\E\vast(\left|a'
\begin{pmatrix}
\sqrt n\alpha_{i,n}\big(\I(X_i\le x_1)-F_i(x_1)\big)\\
\vdots\\
\sqrt n\alpha_{i,n}\big(\I(X_i\le x_p)-F_i(x_p)\big)
\end{pmatrix}
\right|^2\\
&\I\left(\left|a'
\begin{pmatrix}
\sqrt n\alpha_{i,n}\big(\I(X_i\le x_1)-F_i(x_1)\big)\\
\vdots\\
\sqrt n\alpha_{i,n}\big(\I(X_i\le x_p)-F_i(x_p)\big)
\end{pmatrix}
\right|>\sqrt{s_n^2} t\right)\vast)\\
&\le\frac {|a|_1^2}{s_n^2}n\sum_{i=1}^n\alpha_{i,n}^2\I\Big(|a|_1 { \sqrt n\max_{1\le j\le n}\alpha_{j,n}}>\sqrt{s_n^2}t\Big)\longrightarrow 0~\text{as}~n\rightarrow\infty.
\end{split}
\end{equation*}
For that reason, Lindeberg's condition is fulfilled and from Cram\'er-Wold device, the convergence of the finite dimensional marginal distributions of the process $\mathbb U_n$ to centered multivariate normal distributions follows.  Finally,
\begin{equation*}
\begin{split}
&\lim_{n\rightarrow\infty}\Cov\big(U_n(x),U_n(y)\big)=c(x,y),~x,y\in\overline\R^m.
\end{split}
\end{equation*}
The statement follows from Theorem 1.5.4 in \cite{van}.\end{proof}
\begin{proof}[\bf Proof of Corollary \ref{cor1}]
With Theorem \ref{thm1} and the Continuous Mapping Theorem, Theorem 1.3.6 in \cite{van}. Details are given by Baringhaus and Gaigall in the proof of Theorem 3 in \cite{bar}.
\end{proof}
\begin{proof}[\bf Proof of Theorem \ref{thm2}] The process $\mathbb U_n^{(n)}$ is a process based on the triangular array \linebreak  $X_1^{(n)},\dots,X_n^{(n)}$ of row-wise independent and identically distributed random vectors with values in $ \R^{m}$. The asymptotic of this process can be treated similar to the approach related to Theorem \ref{thm1} without any additional conditions by using $\lim_{n\rightarrow\infty} G_n= G$ uniformly on $\overline \R^m$. It follows the convergence of $\mathbb U_n^{(n)}$ to a Gaussian process with a.s. uniformly $\rho$-continuous sample paths and expectation function identically equal to zero. Finally, $\lim_{n\rightarrow\infty} G_n= G$ uniformly on $\overline \R^m$ yields
\begin{equation*}
\lim_{n\rightarrow\infty}\Cov\big(U_n^{(n)}(x),U_n^{(n)}(y)\big)=c(x,y),~x,y\in\overline\R^m.\qedhere
\end{equation*}
\end{proof}
\begin{proof}[\bf Proof of Corollary \ref{cor2}] 
With Theorem \ref{thm2} analogous to the Proof of Corollary \ref{cor1}.
\end{proof}

\begin{proof}[\bf Proof of Corollary \ref{cor3}]
Because $ G$ is uniformly continuous, the restriction of the covariance function (\ref{cov1}) to the diagonal of $\overline \R^m\times\overline \R^m$ does not vanish everywhere or $ G$-almost everywhere. For that reason, the distribution function of $\sup_{x\in\overline\R^m}|U(x)|$ or $\int U^2(x)   G(\mathrm d x)$ is strictly increasing on the non-negative half-line. For details, consider the arguments in Remark 1 and in the Proof of Corollary 1 in \cite{bar}. Consequently, $\lim_{n\rightarrow\infty} c_{n;1-\alpha}=c_{1-\alpha}$, where $c_{1-\alpha}$ is the $(1-\alpha)$-Quantile of $\sup_{x\in\overline\R^m}|U(x)|$,  or $\lim_{n\rightarrow\infty} d_{n;1-\alpha}=d_{1-\alpha}$,  where $d_{1-\alpha}$ is the $(1-\alpha)$-Quantile of $\int U^2(x)   G(\mathrm d x)$, see Corollary \ref{cor2}. Finally, the statement follows from Corollary \ref{cor1}.
\end{proof}

\begin{proof}[\bf Proof of Theorem  \ref{thm3}]Using triangle inequality,
\begin{equation*}
\frac{1}{\sqrt n}\KS_n\ge\sup_{x\in\overline\R^m}| F(x)-  G(x)|-\sup_{x\in\overline\R^m}|-\hat{\mathbb F}_n(x)+  G_n(x)+ F(x)-  G(x)|.
\end{equation*}
From Corollary \ref{cor1} and Slutsky's Theorem,
\begin{equation*}
\lim_{n\rightarrow\infty}\sup_{x\in\overline\R^m}|\hat{\mathbb F}_n(x)- {\mathbb F}_n(x)|=\lim_{n\rightarrow\infty}\frac {1}{\sqrt n}\sqrt n\sup_{x\in\overline\R^m}|\hat{\mathbb F}_n(x)- {\mathbb F}_n(x)|=0~\text{in~probability}.
\end{equation*}
For that reason,
\begin{equation*}
\begin{split}
&\sup_{x\in\overline\R^m}|-\hat{\mathbb F}_n(x)+  G_n(x)+ F(x)-  G(x)|\\
\le&\sup_{x\in\overline\R^m}|\hat{\mathbb F}_n(x)- {\mathbb F}_n(x)|+\sup_{x\in\overline\R^m}| {\mathbb F}_n(x)- F (x)|+\sup_{x\in\overline\R^m}|{ G}_n(x)- G (x)|\overset{P}{\longrightarrow}0~\text{as}~n\rightarrow\infty.
\end{split}
\end{equation*}
For arbitrary $\varepsilon>0$, Slutsky's Theorem implies
\begin{equation*}
\begin{split}
&P\bigg(\frac{1}{\sqrt n}\KS_n\ge\sup_{x\in\overline\R^m}|{ F}(x)- G(x)|-\varepsilon+o_p(1) \bigg)\\
\ge&P\bigg(\sup_{x\in\overline\R^m}|-\hat{\mathbb F}_n(x)+  G_n(x)+ F(x)-  G(x)|+o_p(1) \le\varepsilon \bigg)\longrightarrow 1~\text{as}~n\rightarrow\infty.
\end{split}
\end{equation*}
This yields the first statement. To show the second statement, note that
\begin{equation*}
\begin{split}
&\frac 1 n \CvM_n\\
=& \int\big({ F}(x)- G(x)\big)^2   G(\mathrm d x)+ \int\big({ F}(x)- G(x)\big)^2   G_n(\mathrm d x)- \int\big({ F}(x)- G(x)\big)^2   G(\mathrm d x)\\
&+\int\Big(\big({\mathbb F}_n(x)- G_n(x)\big)^2-\big({ F}(x)- G(x)\big)^2\Big)   G_n(\mathrm d x)\\
&-\int\Big(\big({\mathbb F}_n(x)- G_n(x)\big)^2-\big(\hat{\mathbb F}_n(x)- G_n(x)\big)^2\Big)   G_n(\mathrm d x).
\end{split}
\end{equation*}
It is
\begin{equation*}
\lim_{n\rightarrow\infty}\bigg|\int\big({ F}(x)- G(x)\big)^2   G_n(\mathrm d x)- \int\big({ F}(x)- G(x)\big)^2   G(\mathrm d x)\bigg|=0.
\end{equation*}
Moreover, with Cauchy-Schwarz inequality,
\begin{equation*}
\begin{split}
&\bigg|\int\Big(\big({\mathbb F}_n(x)- G_n(x)\big)^2-\big({ F}(x)- G(x)\big)^2\Big)   G_n(\mathrm d x)\bigg|\\
\le&2\sqrt{\int\Big(\big({\mathbb F}_n(x)- F(x)\big)-\big( G_n(x)- G (x)\big)\Big)^2 G_n(\mathrm d x)}\longrightarrow 0~\text{as}~n\rightarrow\infty.
\end{split}
\end{equation*}
Theorem \ref{thm1} implies the convergence in distribution of the process defined by $\sqrt n({\mathbb F}_n(x)-\hat{\mathbb F}_n(x))$, $x\in\overline\R^m$. $ G_n$ converges uniformly on $\overline \R^m$ to $ G$.  Argumentation analogous to the Proof of Corollary \ref{cor1} yields the convergence in distribution of $n\int({\mathbb F}_n(x)-\hat{\mathbb F}_n(x))^2  G_n(\mathrm d x)$. From Slutsky's Theorem,
\begin{equation*}
\lim_{n\rightarrow\infty}\int\Big({\mathbb F}_n(x)-\hat{\mathbb F}_n(x)\Big)^2  G_n(\mathrm dx)=\lim_{n\rightarrow\infty}\frac 1 n n\int\Big({\mathbb F}_n(x)-\hat{\mathbb F}_n(x)\Big)^2  G_n(\mathrm dx)=0~\text{in~probability}.
\end{equation*}
Therefore, with Cauchy-Schwarz inequality again,
\begin{equation*}
\begin{split}
&\bigg|\int\Big(\big({\mathbb F}_n(x)- G_n(x)\big)^2-\big(\hat{\mathbb F}_n(x)- G_n(x)\big)^2\Big)   G_n(\mathrm d x)\bigg|\\
\le&2\sqrt{\int\Big({\mathbb F}_n(x)-\hat{\mathbb F}_n(x)\Big)^2 G_n(\mathrm d x)}\overset{P}{\longrightarrow}0~\text{as}~n\rightarrow\infty.
\end{split}
\end{equation*}
The rest follows analogous to the proof of the first statement.
\end{proof}

\begin{proof}[\bf Proof of Corollary \ref{cor4}]
It is $\lim_{n\rightarrow\infty} c_{n;1-\alpha}=c_{1-\alpha}$, where $c_{1-\alpha}$ is the $(1-\alpha)$-Quantile of $\sup_{x\in\overline\R^m}|U(x)|$, or $\lim_{n\rightarrow\infty} d_{n;1-\alpha}=d_{1-\alpha}$, where $d_{1-\alpha}$ is the $(1-\alpha)$-Quantile of $\int U^2(x)   G(\mathrm d x)$, see the Proof of Corollary \ref{cor3}. Thus, Theorem \ref{thm3} implies
\begin{equation*}
\begin{split}
&\lim_{n\rightarrow\infty}P(\KS_n\ge c_{n;1-\alpha})\\
=&\lim_{n\rightarrow\infty}P\bigg(\frac {1}{\sqrt n}\KS_n\ge \sup_{x\in\overline\R^m}| F(x)-  G(x)|-\sup_{x\in\overline\R^m}| F(x)-  G(x)|+\frac {1}{\sqrt n} c_{n;1-\alpha}\bigg)=1,
\end{split}
\end{equation*}
or
\begin{equation*}
\begin{split}
&\lim_{n\rightarrow\infty}P(\CvM_n\ge d_{n;1-\alpha})\\
=&\lim_{n\rightarrow\infty}P\bigg(\frac {1}{ n}\CvM_n\ge \int\big({ F}(x)- G(x)\big)^2   G(\mathrm d x)-\int\big({ F}(x)- G(x)\big)^2   G(\mathrm d x)+\frac {1}{ n} d_{n;1-\alpha}\bigg)\\
=&1,
\end{split}
\end{equation*}
i.e., the statement.
\end{proof}

\begin{lem}\label{lem4}Assume (C1) and (C2). Then, for all sequences of parameters $(\vartheta_n)_{n\in\N}$ with $\vartheta_n\in\Theta$ for all $n\in\N$ and $\lim_{n\rightarrow\infty}\vartheta_n=\vartheta$,
\begin{equation*}
\lim_{n\rightarrow\infty}\sup_{x\in\overline\R^m}| G_n(x, \vartheta_n)- { G}(x,\vartheta)|=0.
\end{equation*}
\end{lem}
\begin{proof}[\bf Proof of Lemma \ref{lem4}]
The statement follows with  (C1), (C2), Taylor expansion and triangle inequality.
\end{proof}

\begin{proof}[\bf Proof of Lemma \ref{lem2}]
Consider $n$ sufficiently large. The first statement will be shown firstly. For arbitrary  $a\in\R^d\setminus\lbrace 0 \rbrace$, (C4) implies
\begin{equation*}
\lim_{n\rightarrow\infty}{ n}\sum_{i=1}^n\alpha_{i,n}^2\Var \bigg(a'\Big(\ell_n(X_i,\vartheta)-\int \ell_n(x,\vartheta)F_i(\mathrm d x)\Big)\bigg)= a'\kappa v(\vartheta)a.
\end{equation*}
 For simplicity, assume in the following calculation without loss of generality $\int \ell_n(x,\vartheta)F_i(\mathrm d x)=0$ for all $i\in\N$. Put $s_n^2:= n\sum_{i=1}^n\alpha_{i,n}^2\Var (a'\ell_n(X_i,\vartheta))$ and assume without loss of generality $\lim_{n\rightarrow\infty}s_n^2>0$. (C4) yields
\begin{equation*}
\begin{split}
\forall t>0:&\lim_{n\rightarrow\infty}\frac {1}{ s_n^2} n\sum_{i=1}^n\alpha_{i,n}^2\E\bigg(|a'\ell_n(X_i,\vartheta)|^2\I\Big({\sqrt n}\alpha_{i,n}|a'\ell_n(X_i,\vartheta)|>\sqrt{s_n^2} t\Big)\bigg)\\
&\le\lim_{n\rightarrow\infty}\frac {1}{ {s_n^2}} n\sum_{i=1}^n\alpha_{i,n}^2\int a'\ell_n(x,\vartheta)\ell_n'(x,\vartheta)a\\
&\I\big( (\sqrt n\max_{1\le j\le n}\alpha_{j,n})^2a'\ell_n(x,\vartheta)\ell_n'(x,\vartheta)a>{s_n^2}t^2\big)  F_i(\mathrm d x)= 0.
\end{split}
\end{equation*}
For that reason, the triangular array of row-wise independent random variables \linebreak ${\sqrt n}\alpha_{1,n}a'\ell_n(X_1,\vartheta),\dots, {\sqrt n}\alpha_{n,n}a'\ell_n(X_n,\vartheta)$ fulfills Lindeberg's condition and from Cram\'er-Wold device, ${\sqrt n}\sum_{i=1}^n\alpha_{i,n} \ell_n(X_i,\vartheta)$ converges in distribution to a centered $d$-dimensional normal distribution as $n\rightarrow\infty$. For the second statement, (C3), the first statement and Slutsky's theorem imply
\begin{equation*}
\begin{split}
\hat \vartheta_n-\vartheta&=\sum_{i=1}^n\alpha_{i,n}  \ell_n(X_i,\vartheta)-\int \ell_n(x,\vartheta){\mathbb F}_n(\mathrm d x)+o_p(1)\\
&=\frac {1}{\sqrt n}{\sqrt n}\sum_{i=1}^n\alpha_{i,n}  \bigg(\ell_n(X_i,\vartheta)-\int \ell_n(x,\vartheta) F_i(\mathrm d x)\bigg)+o_p(1)\overset{P}{\longrightarrow} 0~\text{as}~n\rightarrow\infty,
\end{split}
\end{equation*}
i.e., the statement.
\end{proof}
\begin{lem}\label{lem3} 
Put
\begin{equation}\label{We}
 \begin{split}
&W_{i,n}(y,\vartheta):={\sqrt n}\alpha_{i,n} \Bigg(\I( X_i\le y)-  F_i(y)- g'(y, \vartheta) \bigg(\ell_n(X_i,\vartheta)-\int \ell_n(z,\vartheta) F_i(\mathrm d z)\bigg)\Bigg),\\
&y\in\overline\R^m,~i\in\N.
\end{split}
\end{equation}
Assume ${\mathbb F}_n= G_n(\cdot,\vartheta)$ for $n$ sufficiently large and (C1) - (C4). Then, for all $p\in\N$, $x\in\times_{i=1}^p\overline\R^m$, $x=(x_1,\dots,x_p)$, and all $a\in\R^p$, $a=(a_1,\dots,a_p)'$,
\begin{equation*}
\begin{split}
\left|a'\begin{pmatrix}
W_{i,n}(x_1,\vartheta)\\
\vdots\\
W_{i,n}(x_p,\vartheta)
\end{pmatrix}\right|^2\le & 2 n\alpha_{i,n}^2 |a|_1^2+ 2 n\alpha_{i,n} ^2 b'(x,a)\bigg(\ell_n(X_i,\vartheta)-\int \ell_n(z,\vartheta) F_i(\mathrm d z)\bigg)\\
&\bigg(\ell_n(X_i,\vartheta)-\int \ell_n(z,\vartheta) F_i(\mathrm d z)\bigg)'b(x,a),~i\in\N,
\end{split}
\end{equation*}
with $|a|_1:=|a_1|+\dots+|a_p|$ and
\begin{equation*}
b(x,a):=
\begin{pmatrix}
a_1 g_{1}(x_1, \vartheta)+\dots+a_p g_{1}(x_p, \vartheta)\\
\vdots\\
a_1 g_{d}(x_1, \vartheta)+\dots+a_p g_{d}(x_p, \vartheta)
\end{pmatrix}.
\end{equation*}
\end{lem}
\begin{proof}[\bf Proof of Lemma \ref{lem3}]Define
\begin{equation*}
\begin{split}
\underline W_{i,n}(y,\vartheta)&:={\sqrt n}\alpha_{i,n} \big(\I( X_i\le y)-F_i(y)\big),\\
\overline W_{i,n}(y,\vartheta)&:=- {\sqrt n}\alpha_{i,n}  g'(y, \vartheta)\bigg(\ell_n(X_i,\vartheta)-\int \ell_n(z,\vartheta) F_i(\mathrm d z)\bigg),~y\in\overline\R^m,~i\in\N.
\end{split}
\end{equation*}
Thus, $W_{i,n}(y,\vartheta)=\underline W_{i,n}(y,\vartheta)+\overline W_{i,n}(y,\vartheta)$. Using triangle inequality,
\begin{equation*}
\begin{split}
\left|a'\begin{pmatrix}
W_{i,n}(x_1,\vartheta)\\
\vdots\\
W_{i,n}(x_p,\vartheta)
\end{pmatrix}\right|^2
&=
\left|
a'\begin{pmatrix}
\underline W_{i,n}(x_1,\vartheta)\\
\vdots\\
\underline W_{i,n}(x_p,\vartheta)
\end{pmatrix}
+a'\begin{pmatrix}
\overline W_{i,n}(x_1,\vartheta)\\
\vdots\\
\overline W_{i,n}(x_p,\vartheta)
\end{pmatrix}
\right|^2
\\
&\le
\left(
{\sqrt n}\alpha_{i,n} |a|_1
+\left|a'\begin{pmatrix}
\overline W_{i,n}(x_1,\vartheta)\\
\vdots\\
\overline W_{i,n}(x_p,\vartheta)
\end{pmatrix}
\right|\right)^2.
\end{split}
\end{equation*}
Furthermore, with a simple calculation, 
\begin{equation*}
a'\begin{pmatrix}
\overline W_{i,n}(x_1,\vartheta)\\
\vdots\\
\overline W_{i,n}(x_p,\vartheta)
\end{pmatrix}
=- {\sqrt n}\alpha_{i,n} b'(x,a) \bigg(\ell_n(X_i,\vartheta)-\int \ell_n(z,\vartheta) F_i(\mathrm d z)\bigg),~i\in\N.
\end{equation*}
Finally, the inequality $(y+z)^2\le 2 y^2+2z^2$, $y,z\in\R$, yields the statement.
\end{proof}

\begin{proof}[\bf Proof of Theorem \ref{thm4}]  Consider $n$ sufficiently large. From Lemma \ref{lem2}, $\lim_{n\rightarrow\infty}\hat \vartheta_n=\vartheta$ in probability. Therefore, assume without loss of generality that $\hat \vartheta_n$ has realizations in $U_{\vartheta}$ , with a set $ U_{\vartheta}$ given by (C1). Using Taylor expansion,
\begin{equation*}
\begin{split}
 G_n(x,\hat \vartheta_n)=& G_n(x, \vartheta)+ g_n'(x,\overline \vartheta_{x,n})(\hat \vartheta_n-\vartheta),~x\in\overline\R^m,
\end{split}
\end{equation*}
where  $\overline \vartheta_{x,n}$ is on the line between $\hat \vartheta_n$ and $ \vartheta$. Using this and  ${\mathbb F}_n=\mathbb G_n(\cdot, \vartheta)$,
\begin{equation*}
\begin{split}
U_n(x,\hat \vartheta_n)=&\sqrt n \big( \hat{ \mathbb F}_n(x)-{\mathbb F}_n(x)\big)- g'(x, \vartheta)\sqrt n(\hat \vartheta_n-\vartheta)\\
&+\big( g'(x, \vartheta)- g_n'(x, \overline \vartheta_{x,n})\big)\sqrt n(\hat \vartheta_n-\vartheta),~x\in\overline\R^m.
\end{split}
\end{equation*}
Because of (C3), this is equivalent to
\begin{equation*}
\begin{split}
U_n(x,\hat \vartheta_n)=&\sqrt n \big(\hat{ \mathbb F}_n(x)-{\mathbb F}_n(x)\big)- g'(x, \vartheta){\sqrt n}\sum_{i=1}^n \alpha_{i,n}  \bigg(\ell_n(X_i,\vartheta)-\int \ell_n(y,\vartheta) F_i(\mathrm d y)\bigg)\\
&+\big( g'(x, \vartheta)- g_n'(x, \overline \vartheta_{x,n})\big)\sqrt n(\hat \vartheta_n-\vartheta)- g'(x, \vartheta)o_p(1),~x\in\overline\R^m.
\end{split}
\end{equation*}
For simplicity, assume without loss of generality $\int \ell_n(y,\vartheta)F_i(\mathrm d y)=0$ for all $i\in\N$. Using (C2), Lemma \ref{lem2} and Slutsky's Theorem,
\begin{equation*}
\begin{split}
&\sup_{x\in\overline \R^m}\big|\big( g'(x, \vartheta)- g_n'(x, \overline \vartheta_{x,n})\big)\sqrt n(\hat \vartheta_n-\vartheta)- g'(x,\vartheta)o_p(1)\big|\overset{P}{\longrightarrow} 0~\text{as}~n\rightarrow\infty.
\end{split}
\end{equation*}
Because of Slutsky's theorem, Example 1.4.7 in \cite{van}, it is sufficient to show the convergence statement for the process defined by
\begin{equation*}
V_n(x, \vartheta):=\sqrt n \big( \hat{ \mathbb F}_n(x)- {\mathbb F}_n(x)\big)- g'(x, \vartheta){\sqrt n}\sum_{i=1}^n \alpha_{i,n} \ell_n(X_i,\vartheta),~x\in\overline\R^m.
\end{equation*}
The process defined by $\sqrt n (\hat{ \mathbb F}_n(x)-{\mathbb F}_n(x))$, $x\in\overline\R^m$, is asymptotically equicontinuous with respect to the metric space $(\overline\R^m,\rho)$, see Lemma \ref{lem0}. In addition, from (C2) and Lemma \ref{lem2}, the process defined by $ g'(x, \vartheta){\sqrt n}\sum_{i=1}^n\alpha_{i,n}  \ell_n(X_i,\vartheta)$, $x\in\overline\R^m$, is asymptotically equicontinuous with respect to the metric space $(\overline\R^m,\rho)$, too. Finally, the process $(V_n(x, \vartheta);x\in\overline\R^m)$ is asymptotically equicontinuous with respect to the metric space $(\overline\R^m,\rho)$. Consider $W_{i,n}(\cdot,\vartheta)$ defined in (\ref{We}), $i\in\N$. Thus,
\begin{equation*}
V_n(y, \vartheta)=\sum_{i=1}^nW_{i,n}(y,\vartheta),~y\in\overline\R^m.
\end{equation*}
For arbitrary $p\in\N$, $x\in\times_{i=1}^p\overline\R^m$, $x=(x_1,\dots,x_p)$, and arbitrary $a\in\R^p\setminus\lbrace 0 \rbrace$, put 
\begin{equation*}
s_n^2:=\sum_{i=1}^n\Var
\left(a'
\begin{pmatrix}
W_{i,n}(x_1,\vartheta)\\
\vdots\\
W_{i,n}(x_p,\vartheta)
\end{pmatrix}
\right).
\end{equation*}
With $a=(a_1,\dots,a_p)'$,  (C4) yields
\begin{equation*}
\begin{split}
\lim_{n\rightarrow\infty}s_n^2=&\sum_{1\le j,k\le p}a_ja_k\kappa\Big(  G\big(\min(x_j,x_k), \vartheta\big)- G(x_j,\vartheta)\mathbb G(x_k,\vartheta)\\
&- g'(x_j, \vartheta) w(x_k, \vartheta)- g'(x_k, \vartheta) w(x_j, \vartheta)+ g'(x_j, \vartheta) v(\vartheta) g(x_k, \vartheta)\Big).
\end{split}
\end{equation*}
Assume without loss of generality $\lim_{n\rightarrow\infty}s_n^2>0$. With $b(x,a)$ and $|a|_1$ defined in Lemma \ref{lem3}, Lemma \ref{lem3} and (C4) implies
\begin{equation*}
\begin{split}
\forall t>0:&\frac {1}{ s_n^2}\sum_{i=1}^n\E\left(\left|a'
\begin{pmatrix}
W_{i,n}(x_1,\vartheta)\\
\vdots\\
W_{i,n}(x_p,\vartheta)
\end{pmatrix}
\right|^2\I\left(\left|a'
\begin{pmatrix}
W_{i,n}(x_1,\vartheta)\\
\vdots\\
W_{i,n}(x_p,\vartheta)
\end{pmatrix}
\right|>\sqrt{s_n^2} t\right)\right)\\
&\le\frac {2}{s_n^2}|a|_1^2 n\sum_{i=1}^n\alpha_{i,n}^2\int \I\bigg(  (\sqrt n\max_{1\le j\le n}\alpha_{j,n})^2b'(x,a)\ell_n(y,\vartheta)\\
&\ell_n'(y,\vartheta)b(x,a)>{s_n^2}\frac{t^2}{2}- (\sqrt n\max_{1\le j\le n}\alpha_{j,n})^2|a|_1^2\bigg)  F_i( \mathrm d y)\\ 
&+\frac {2}{s_n^2}n \sum_{i=1}^n\alpha_{i,n}^2\int b'(x,a)\ell_n(y,\vartheta)\ell_n'(y,\vartheta)b(x,a)\\
&\I\bigg( (\sqrt n\max_{1\le j\le n}\alpha_{j,n})^2 b'(x,a)\ell_n(y,\vartheta)\\
&\ell_n'(y,\vartheta)b(x,a)>{s_n^2}\frac{t^2}{2}- (\sqrt n\max_{1\le j\le n}\alpha_{j,n})^2|a|_1^2\bigg)  F_i( \mathrm d y)\longrightarrow 0~\text{as}~n\rightarrow\infty.
\end{split}
\end{equation*}
For that reason, Lindeberg's condition is fulfilled and from Cram\'er-Wold device, the convergence of the finite dimensional marginal distributions of the process $(V_n(x, \vartheta);x\in\overline\R^m)$ to centered multivariate normal distributions follows. Finally, (C4) yields
\begin{equation*}
\lim_{n\rightarrow\infty}\Cov\big(V_n(x,\vartheta),V_n(y,\vartheta)\big)=c(x,y,\vartheta),~x,y\in\overline\R^m,
\end{equation*}
and the statement follows from Theorem 1.5.4 in \cite{van}.
\end{proof}

\begin{proof}[\bf Proof of Corollary \ref{cor5}]
With Theorem \ref{thm4}, Lemma \ref{lem2} and Lemma \ref{lem4} analogous to the Proof of Corollary \ref{cor1}.
\end{proof}

\begin{proof}[\bf Proof of Lemma \ref{lem5}]
This follows analogous to the Proof of Lemma \ref{lem2}.
\end{proof}
\begin{lem}\label{lemtech2} 
Put
\begin{equation}\label{We2}
 Z_{i,n}(y,\tilde\vartheta):={\sqrt n} \alpha_{i,n}\Bigg(\I( X_i^{(n)}\le y)-  \mathbb G_n(y,\tilde\vartheta)- g'(y, \vartheta)\ell_n(X_i^{(n)},\tilde\vartheta)\Bigg),~(y,\tilde\vartheta)\in\overline\R^m\times U_{\vartheta},~i\in\N.
\end{equation} Assume (C1), (C2), (C5) and (C6). Then, for all $p\in\N$, $x\in\times_{i=1}^p\overline\R^m$, $x=(x_1,\dots,x_p)$, and all $a\in\R^p$, $a=(a_1,\dots,a_p)'$,
\begin{equation*}
\left|a'\begin{pmatrix}
Z_{i,n}(x_1,\tilde\vartheta)\\
\vdots\\
Z_{i,n}(x_p,\tilde\vartheta)
\end{pmatrix}\right|^2\le 2n \alpha_{i,n}^2|a|_1^2+ 2 n \alpha_{i,n}^2 b'(x,a)\ell_n(X_i^{(n)},\tilde\vartheta)\ell_n'(X_i^{(n)},\tilde\vartheta)b(x,a),~\tilde\vartheta\in U_{\vartheta},~i\in\N,
\end{equation*}
with $b(x,a)$ and $|a|_1$ defined in Lemma \ref{lem3}.
\end{lem}
\begin{proof}[\bf Proof of Lemma \ref{lemtech2}] The proof is analogous to the Proof of Lemma \ref{lem3}.
\end{proof}

\begin{proof}[\bf Proof of Theorem \ref{thm5}] Consider $n$ sufficiently large. From Lemma \ref{lem5}, $\lim_{n\rightarrow\infty}\hat \vartheta_n^{(n)}=\vartheta$ in probability. Therefore, assume without loss of generality that $\hat \vartheta_n^{(n)}$ has realizations in $U_{\vartheta}$ and $ \vartheta_n\in U_{\vartheta}$, with a set $ U_{\vartheta}$ given by (C1).  Put
\begin{equation*}
\hat{\mathbb F}_n^{(n)}(x):=\sum_{i=1}^n\alpha_{i,n}\I(X_i^{(n)}\le x),~x\in\overline\R^m.
\end{equation*}
Using Taylor expansion,
\begin{equation*}
\begin{split}
 G_n(x,\hat \vartheta_n^{(n)})=& G_n(x, \vartheta_n)+ g_n'(x,\overline \vartheta_{x,n})(\hat \vartheta_n^{(n)}-\vartheta_n),~x\in\overline\R^m,
\end{split}
\end{equation*}
where $\overline \vartheta_{x,n}$ is on the line  between $\hat \vartheta_n^{(n)}$ and $ \vartheta_n$. Then,
\begin{equation*}
\begin{split}
&U_n(x,\hat \vartheta_n^{(n)})=\sqrt n \big(\hat{\mathbb F}^{(n)}_n(x)-  G_n(x, \vartheta_n)\big)- g'(x, \vartheta)\sqrt n(\hat \vartheta_n^{(n)}-\vartheta_n)\\
&+\big( g'(x, \vartheta)- g_n'(x, \overline \vartheta_{x,n})\big)\sqrt n(\hat \vartheta_n^{(n)}-\vartheta_n),~x\in\overline\R^m.
\end{split}
\end{equation*}
Using (C5), this is equivalent to
\begin{equation*}
\begin{split}
U_n(x,\hat \vartheta_n)=&\sqrt n \big( \hat{\mathbb F}^{(n)}_n(x)-{ G}_n(x, \vartheta_n)\big)- g'(x, \vartheta){\sqrt n}\sum_{i=1}^n\alpha_{i,n}  \ell_n(X_i^{(n)},\vartheta_n)\\
&+\big( g'(x, \vartheta)- g_n'(x, \overline \vartheta_{x,n})\big)\sqrt n(\hat \vartheta_n^{(n)}-\vartheta_n)- g'(x, \vartheta)o_p(1),~x\in\overline\R^m.
\end{split}
\end{equation*}
(C2), Lemma \ref{lem5} and Slutsky's Theorem yield
\begin{equation*}
\begin{split}
&\sup_{x\in\overline \R^m}\big|\big( g'(x, \vartheta)- g_n'(x, \overline \vartheta_{x,n})\big)\sqrt n(\hat \vartheta_n^{(n)}-\vartheta_n)- g'(x,\vartheta)o_p(1)\big|\overset{P}{\longrightarrow} 0~\text{as}~n\rightarrow\infty.
\end{split}
\end{equation*}
Because of Slutsky's theorem, it is sufficient to show the convergence statement for the process defined by
\begin{equation*}
V_n^{(n)}(x, \vartheta_n):=\sqrt n \big( \hat{\mathbb F}^{(n)}_n(x)- { G}_n(x,\vartheta_n)\big)- g'(x, \vartheta){\sqrt n}\sum_{i=1}^n \alpha_{i,n}\ell_n(X_i^{(n)},\vartheta_n),~x\in\overline\R^m.
\end{equation*}
Because $\lim_{n\rightarrow\infty} G(\cdot,\vartheta_n)= G(\cdot,\vartheta)$ uniformly on $\overline \R^m$, the process defined by $\sqrt n ( \hat{\mathbb F}^{(n)}_n(x)-{ G}_n(x,\vartheta_n))$, $x\in\overline\R^m$, is asymptotically equicontinuous with respect to the metric space $(\overline\R^m,\rho)$, see Lemma \ref{lem0}. In addition, from (C2) and Lemma \ref{lem5}, the process defined by $ g'(x, \vartheta){\sqrt n}\sum_{i=1}^n\alpha_{i,n} \ell_n(X_i^{(n)},\vartheta_n)$, $x\in\overline\R^m$, is asymptotically equicontinuous with respect to the metric space $(\overline\R^m,\rho)$, too. Finally, the process  $(V_n^{(n)}(x, \vartheta_n);x\in\overline\R^m)$ is asymptotically equicontinuous with respect to the metric space $(\overline\R^m,\rho)$. Consider $Z_{i,n}$ defined in (\ref{We2}), $i\in\N$. Thus,
\begin{equation*}
V_n^{(n)}(y, \vartheta_n)=\sum_{i=1}^nZ_{i,n}(y,\vartheta_n),~y\in\overline\R^m.
\end{equation*}
For arbitrary $p\in\N$, $x\in\times_{i=1}^p\overline\R^m$, $x=(x_1,\dots,x_p)$, and arbitrary $a\in\R^p\setminus\lbrace 0 \rbrace$, put 
\begin{equation*}
s_n^2:=\sum_{i=1}^n\Var
\left(a'
\begin{pmatrix}
Z_{i,n}(x_1,\vartheta_n)\\
\vdots\\
Z_{i,n}(x_p,\vartheta_n)
\end{pmatrix}
\right).
\end{equation*}
With $a=(a_1,\dots,a_p)'$,  (C6) yields
\begin{equation*}
\begin{split}
\lim_{n\rightarrow\infty}s_n^2=&\sum_{1\le j,k\le p}a_ja_k\kappa\Big(  G\big(\min(x_j,x_k), \vartheta\big)-  G(x_j,\vartheta)  G(x_k,\vartheta)\\
&- g'(x_j, \vartheta) w^{()}(x_k, \vartheta)- g'(x_k, \vartheta) w^{()}(x_j, \vartheta)+ g'(x_j, \vartheta) v^{()}(\vartheta) g(x_k, \vartheta)\Big).
\end{split}
\end{equation*}
Assume without loss of generality $\lim_{n\rightarrow\infty}s_n^2>0$. With $b(x,a)$ and $|a|_1$ defined in Lemma \ref{lem3}, Lemma \ref{lemtech2} and (C6) yields
\begin{equation*}
\begin{split}
\forall t>0:~
&\frac {1}{ s_n^2}\sum_{i=1}^n\E\left(\left|a'
\begin{pmatrix}
Z_{i,n}(x_1,\vartheta_n)\\
\vdots\\
Z_{i,n}(x_p,\vartheta_n)
\end{pmatrix}
\right|^2\I\left(\left|a'
\begin{pmatrix}
Z_{i,n}(x_1,\vartheta_n)\\
\vdots\\
Z_{i,n}(x_p,\vartheta_n)
\end{pmatrix}
\right|>\sqrt{s_n^2} t\right)\right)\\
&\le\frac {2}{ {s_n^2}}|a|_1^2{ n}\sum_{i=1}^n \alpha_{i,n}^2\int\I\Big( (\sqrt n\max_{1\le j\le n}\alpha_{j,n})^2b'(x,a)\ell_n(y,\vartheta_n)\\
&\ell_n'(y,\vartheta_n)b(x,a)>{s_n^2}\frac{t^2}{2}-(\sqrt n\max_{1\le j\le n}\alpha_{j,n})^2|a|_1^2\Big)   G_n( \mathrm d y,\vartheta_n)\\
&+\frac {2}{ {s_n^2}}{ n}\sum_{i=1}^n \alpha_{i,n}^2\int b'(x,a)\ell_n(y,\vartheta_n)\ell_n'(y,\vartheta_n)b(x,a)\\
&\I\Big( (\sqrt n\max_{1\le j\le n}\alpha_{j,n})^2b'(x,a)\ell_n(y,\vartheta_n)\\
&\ell_n'(y,\vartheta_n)b(x,a)>{s_n^2}\frac{t^2}{2}-(\sqrt n\max_{1\le j\le n}\alpha_{j,n})^2|a|_1^2\Big)   G_n( \mathrm d y,\vartheta_n)\longrightarrow 0~\text{as}~n\rightarrow\infty.
\end{split}
\end{equation*}
For that reason, Lindeberg's condition is fulfilled and from Cram\'er-Wold device, the convergence of the finite dimensional marginal distributions of the process $(V_n^{(n)}(x, \vartheta_n);x\in\overline\R^m)$ to centered multivariate normal distributions follows. Finally, Lemma \ref{lem4} and (C6) yields
\begin{equation*}
\begin{split}
\lim_{n\rightarrow\infty}\Cov\big(V_n^{(n)}(x, \vartheta_n),V_n^{(n)}(y, \vartheta_n)\big)=c^{()}(x,y,\vartheta),~x,y\in\overline\R^m,
\end{split}
\end{equation*}
and the statement follows from Theorem 1.5.4 in \cite{van}.
\end{proof}

\begin{proof}[\bf Proof of Corollary \ref{cor6}]
With Theorem \ref{thm5}, Lemma \ref{lem4} and Lemma \ref{lem5} analogous to the Proof of Corollary \ref{cor1}.
\end{proof}

\begin{proof}[\bf Proof of Corollary \ref{cor7}]
Because the restriction of the covariance function (\ref{cov3}) to the diagonal of $\overline \R^m\times\overline \R^m$ does not vanish everywhere or $ G(\cdot,\vartheta)$-almost everywhere, the distribution function of $\sup_{x\in\overline\R^m}|U(x,\vartheta)|$ or $\int U^2(x,\vartheta)   G(\mathrm d x,\vartheta)$ is strictly increasing on the non-negative half-line. In addition, the assumptions yield that the covariance functions (\ref{cov3}) and (\ref{cov4}) coincide. For that reason, $\lim_{n\rightarrow\infty} c_{n;1-\alpha}=c_{1-\alpha}$ in probability, where $c_{1-\alpha}$ is the $(1-\alpha)$-Quantile of $\sup_{x\in\overline\R^m}|U(x,\vartheta)|$, or $\lim_{n\rightarrow\infty} d_{n;1-\alpha}=d_{1-\alpha}$ in probability, where  $d_{1-\alpha}$ is the $(1-\alpha)$-Quantile of $\int U^2(x,\vartheta)   G(\mathrm d x,\vartheta)$, see Corollary \ref{cor6} and Lemma \ref{lem2}. The statement follows from Corollary \ref{cor5}.
\end{proof}

\begin{proof}[\bf Proof of Theorem \ref{thm6}] This follows with arguments similar to the arguments in the Proof of Theorem \ref{thm3} by using Lemma \ref{lem4} and (C7).
\end{proof}

\begin{proof}[\bf Proof of Corollary \ref{cor8}] Because the restriction of the covariance function (\ref{cov4}) to the diagonal of $\overline \R^m\times\overline \R^m$ does not vanish everywhere or $ G(\cdot,\vartheta)$-almost everywhere, the distribution function of $\sup_{x\in\overline\R^m}|U^{()}(x,\vartheta)|$ or $\int (U^{()}(x,\vartheta))^2   G(\mathrm d x,\vartheta)$ is strictly increasing on the non-negative half-line. For that reason, $\lim_{n\rightarrow\infty} c_{n;1-\alpha}=c_{1-\alpha}$ in probability, where $c_{1-\alpha}$ is the $(1-\alpha)$-Quantile of $\sup_{x\in\overline\R^m}|U^{()}(x,\vartheta)|$, or $\lim_{n\rightarrow\infty} d_{n;1-\alpha}=d_{1-\alpha}$ in probability, where $d_{1-\alpha}$ is the $(1-\alpha)$-Quantile of $\int (U^{()}(x,\vartheta))^2   G(\mathrm d x,\vartheta)$, see Corollary \ref{cor6} and regard (C7). With Theorem \ref{thm6},  the rest follows similar to the Proof of Corollary \ref{cor4}.
\end{proof}

\begin{lem}\label{lemallg}It is
\begin{equation*}
\lim_{\min N\rightarrow\infty}\sup_{(x,v)\in \overline \R^{m}\times  \overline \R^{s}}\big|\big|\big(\hat {\mathbb F}_n(x),\hat {\mathbb G}_r(v)\big)- \big(F(x),G(v)\big)\big|\big|=0~\text{in~probability}.
\end{equation*}
\end{lem}
\begin{proof}[\bf Proof of Lemma \ref{lemallg}]
Applying Corollary \ref{cor1} yields the convergence in distribution of \linebreak $\sqrt {n}\sup_{x\in\overline\R^{m}}|\hat {\mathbb F}_n(x)-{\mathbb F}_{n}(x)|$ and  $\sqrt {r}\sup_{v\in\overline\R^{s}}|\hat {\mathbb G}_r(v)-{\mathbb G}_{r}(v)|$. Then, from Slutsky's Theorem, $\sup_{x\in\overline\R^{m}}|\hat {\mathbb F}_{n}(x)- { F}(x)|\overset{P}{\longrightarrow} 0$ as $n\rightarrow\infty$ as well as $\sup_{v\in\overline\R^{s}}|\hat {\mathbb G}_{r}(v)- { G}(v)|\overset{P}{\longrightarrow} 0$ as $r\rightarrow\infty$.
\end{proof}

\begin{proof}[\bf Proof of Theorem \ref{thmallg1}]Consider $n$ sufficiently large. Denote by $\mathbb M_{N,m}:=(M_{N,m}(x);x\in\overline\R^{m})$ the process
\begin{equation*}
M_{N,m}(x):=\sqrt[4]{nr}\big(\hat{\mathbb F}_{n}(x)-{\mathbb F_n}(x)\big),~x\in\overline\R^{m},
\end{equation*} 
denote by $\mathbb M_{N,s}:=(M_{N,s}(v);v\in\overline\R^{s})$ the process
\begin{equation*}
M_{N,s}(v):=\sqrt[4]{nr}\big(\hat{\mathbb G}_{r}(v)-{\mathbb G_r}(v)\big),~v\in\overline\R^{s},
\end{equation*} 
put $\mathbb M_N:=(\mathbb M_{N,m},\mathbb M_{N,s})$ and $ M_N:=( M_{N,m}, M_{N,s})$. Thus, $ M_N=\sqrt[4]{nr}( (\hat{\mathbb F}_{n},\hat{\mathbb  G}_{r})-( {\mathbb F}_n, {\mathbb G}_r))$. Because the processes $\mathbb M_{N,m}$ and $\mathbb M_{N,s}$ are independent, Theorem \ref{thm1} and (\ref{samplesize}) yield the convergence in distribution
\begin{equation*}
\mathbb M_N\overset{\mathrm d}{\longrightarrow} \mathbb M~\text{as}~\min N\rightarrow\infty
\end{equation*}
on the related product space, see \cite{van}. $T(\mathbb F_n,\mathbb G_r)=Q(\mathbb F_n,\mathbb G_r)$  and the continuity of $T$ and $Q$ imply $T( F,G)=Q( F,G)$ and therefore $ J=j( F,G)=(F,G)$. Regarding Theorem 1.10.4 in \cite{van}, one can assume  a.s. $\lim_{\min N\rightarrow\infty}\mathbb M_N=\mathbb M$ uniformly on $ \overline \R^{m}\times \overline \R^{s}$ without loss of generality. Because  $T(\mathbb F_n,\mathbb G_r)=Q(\mathbb F_n,\mathbb G_r)$,
\begin{equation*}
\begin{split}
\mathbb U_N=&\sqrt[4]{nr} \bigg(T\Big((\mathbb F_n,\mathbb G_r)+\frac{1}{\sqrt[4]{nr}} M_N\Big)-T(\mathbb F_n,\mathbb G_r)\bigg)-\mathrm d T( F,G)( M)\\
&-\Bigg(\sqrt[4]{nr} \bigg(Q\Big((\mathbb F_n,\mathbb G_r)+\frac{1}{\sqrt[4]{nr}} M_N\Big)-Q(\mathbb F_n,\mathbb G_r)\bigg)-\mathrm d Q( F,G)( M)\Bigg)\\
&+\mathrm d T( F,G)( M)-\mathrm d Q( F,G)( M).
\end{split}
\end{equation*}
The assumtions on $T$ and $Q$ yield a.s. $\lim_{\min N\rightarrow\infty}\mathbb U_N=\mathbb U$ uniformly on $\overline\R^u$.
\end{proof}

\begin{proof}[\bf Proof of Corollary \ref{corallg1}]
The continuity of $T$ and $j$ and Lemma \ref{lemallg} imply
\begin{equation*}
\lim_{n\rightarrow\infty}\sup_{z\in \overline \R^{u}}|T(\hat{\mathbb J}_N)(z)- T( J)(z)|=0~\text{in~probability}.
\end{equation*}
Theorem \ref{thmallg1} yields $ J=(F,G)$. The rest follows with Theorem \ref{thmallg1} analogous to the Proof of Corollary \ref{cor1}.
\end{proof}

\begin{proof}[\bf Proof of Theorem \ref{thmallg2}] Put
\begin{equation}\label{triest1}
\hat{\mathbb F}_{n}^{(N)}(x):= \sum_{i=1}^{n}\alpha_{i,n}\I(X_i^{(N)}\le x),~x\in\overline\R^{m},
\end{equation}
as well as
\begin{equation}\label{triest2}
\hat{\mathbb G}_{r}^{(N)}(v):= \sum_{i=1}^{r}\beta_{i,r}\I(V_i^{(N)}\le v),~v\in\overline\R^{s},
\end{equation}
and  $\hat {\mathbb J}_{N}^{(N)}:=j(\hat{\mathbb F}_{n}^{(N)},\hat{\mathbb G}_{r}^{(N)})$. Denote by $\mathbb M^{(N)}_{N,m}:=(M^{(N)}_{N,m}(x);x\in\overline\R^{m})$ the process
\begin{equation*}
M^{(N)}_{N,m}(x):=\sqrt[4]{nr}\big(\hat{\mathbb F}_{n}^{(N)}(x)-J_{N,m}(x)\big),~x\in\overline\R^{m},
\end{equation*} 
by $\mathbb M^{(N)}_{N,s}:=(M^{(N)}_{N,s}(v);v\in\overline\R^{s})$ the process
\begin{equation*}
M^{(N)}_{N,s}(v):=\sqrt[4]{nr}\big(\hat{\mathbb G}_{r}^{(N)}(v)-J_{N,s}(v)\big),~v\in\overline\R^{s},
\end{equation*} 
put $\mathbb M_N^{(N)}:=(\mathbb M^{(N)}_{N,m},\mathbb M^{(N)}_{N,s})$ and $ M_N^{(N)}:=( M^{(N)}_{N,m}, M^{(N)}_{N,s})$.  Thus, $ M_N^{(N)}= \sqrt[4]{nr}( (\hat{\mathbb F}_{n}^{(N)},\hat{\mathbb G}_{r}^{(N)})-J_N)$. The processes $\mathbb M^{(N)}_{N,m}$ and $\mathbb M^{(N)}_{N,s}$  based on triangular arrays of row-wise independent and identically distributed random vectors. The asymptotic of this processes can be treated similar to the approach related to Theorem \ref{thm1} without any additional conditions by using $\lim_{\min N\rightarrow\infty} J_{N,m}= J_m$ uniformly on $\overline \R^{m}$ and $\lim_{\min N\rightarrow\infty} J_{N,s}= J_s$ uniformly on $\overline \R^{s}$, the uniformly continuity of $ J_m$ and $ J_s$  and  (\ref{samplesize}).  Because the processes $\mathbb M^{(N)}_{N,m}$ and $\mathbb M^{(N)}_{N,s}$ are independent, it follows that
\begin{equation*}
\mathbb M_N^{(N)}\overset{\mathrm d}{\longrightarrow} \mathbb M~\text{as}~\min N\rightarrow\infty
\end{equation*}
on the related product space. Again, one can assume  a.s. $\lim_{\min N\rightarrow\infty}\mathbb M_N^{(N)}=\mathbb M$ uniformly on $\overline \R^{m}\times \overline \R^{s}$ without loss of generality. Because  $T(J_N)=Q(J_N)$,
\begin{equation*}
\begin{split}
\mathbb U_N^{(N)}=&\sqrt[4]{nr} \bigg(T\Big(J_N+\frac{1}{\sqrt[4]{nr}} M_N^{(N)}\Big)-T(J_N)\bigg)-\mathrm d T( J)( M)\\
&-\Bigg(\sqrt[4]{nr} \bigg(Q\Big(J_N+\frac{1}{\sqrt[4]{nr}} M_N^{(N)}\Big)-Q(J_N)\bigg)-\mathrm d Q( J)( M)\Bigg)\\
&+\mathrm d T( J)( M)-\mathrm d Q( J)( M).
\end{split}
\end{equation*}
The assumtions on $T$, $Q$ and $J_N$ yield a.s. $\lim_{n\rightarrow\infty}\mathbb U_N^{(N)}=\mathbb U$ uniformly on $\overline\R^u$ and the statement follows.
\end{proof}

\begin{proof}[\bf Proof of Corollary \ref{corallg2}]
The asymptotic of $\hat {\mathbb F}_{n}^{(N)}$ and $\hat{\mathbb G}_{r}^{(N)}$ defined in (\ref{triest1}) and (\ref{triest2}) can be treated similar to the approach related to Corollary \ref{thm1} and Lemma \ref{lemallg} without any additional conditions by using $\lim_{\min N \rightarrow\infty} J_{N,m}= J_{m}$ uniformly on $\overline \R^{m}$, $\lim_{\min N \rightarrow\infty} J_{N,s}= J_{s}$ uniformly on $\overline \R^{s}$ and the uniformly continuity of $ J_m$ and $ J_s$. The continuity of $T$ and $j$ yield
\begin{equation*}
\lim_{\min N\rightarrow\infty}\sup_{z\in\overline \R^{u}}|T(\hat{\mathbb J}_N^{(N)})(z)- T( J)(z)|=0~\text{in~probability},
\end{equation*}
where $\hat {\mathbb J}_{N}^{(N)}$ is defined in the Proof of Theorem  \ref{thmallg2}. The rest follows with Theorem \ref{thmallg2} analogous to the Proof of Corollary \ref{cor1}.
\end{proof}

\begin{proof}[\bf Proof of Corollary \ref{corallg3}] The proof is analogous to the Proof of Corollary \ref{cor7} by applying Corollary \ref{corallg1}, Corollary \ref{corallg2}, Lemma \ref{lemallg} and the continuity of $j$. Regard that $ J=(F,G)$, see Theorem \ref{thmallg1}.
\end{proof}

\begin{proof}[\bf Proof of Theorem \ref{thmallg3}]This follows with arguments similar to the arguments in the Proof of  Theorem \ref{thm3} by using Lemma \ref{lemallg} and the continuity of $T$, $Q$ and $j$.
\end{proof}

\begin{proof}[\bf Proof of Corollary \ref{corallg4}] The proof is analogous to the Proof of Corollary \ref{cor8} by applying Corollary \ref{corallg2}, Lemma \ref{lemallg}, the continuity of $j$ and Theorem \ref{thmallg3}.
\end{proof}

\bibliographystyle{plain}

\end{document}